\documentclass[11pt,reqno]{amsart}
\usepackage{fullpage,amssymb,amsmath,verbatim,amstext,eucal,amscd}
\usepackage[all]{xy}
\usepackage{amsrefs}
\numberwithin{equation}{section}
\newtheorem{lemma}[equation]{Lemma}
 \newtheorem{proposition}[equation]{Proposition}
\newtheorem{theorem}[equation]{Theorem}
\newtheorem{conjecture}[equation]{Conjecture}
\newtheorem{corollary}[equation]{Corollary}

\newtheorem{question}[equation]{Question}

\theoremstyle{definition}
\newtheorem{definition}[equation]{Definition}
\newtheorem{example}[equation]{Example}

\newenvironment{remark}[1]{\refstepcounter{equation}%
\vskip 5pt \par\noindent {\bf #1\ \thelemma .}}{\vskip 5pt \par}

\newenvironment{remark*}[1]{\par \vskip 5pt \noindent 
{\bf #1.}}{\vskip 5pt \par}

\newcommand{\FLEX}{\relax}
\newcommand{\flex}[1]{\renewcommand{\FLEX}{#1}}
\newtheorem{flexthm}[equation]{\FLEX}

\newcommand{\cstar}{\hbox{$C^*$}}
\newcommand{\cstaralg}{$C^*$-algebra}

\providecommand{\dual}[1]{\ensuremath{#1^{\#}}}

\newcommand{\dom}{\operatorname{dom}}
\newcommand{\dperp}{{\perp\perp}}
\newcommand{\ds}{\displaystyle}
\newcommand{\dstext}[1]{\quad\text{#1}\quad}
\newcommand{\eps}{\ensuremath{\varepsilon}}

\newcommand{\id}{\operatorname{id}}
\newcommand{\innerprod}[1]{\left\langle #1\right\rangle}
\newcommand{\inv}[1]{{#1}^{-1}}

\newcommand{\norm}[1]{\left\|{#1}\right\|}

\providecommand{\proof}{\noindent {\em Proof.}\,\,}

\DeclareMathOperator{\PsExp}{\operatorname{PsExp}}

\providecommand{\qed}%
{\hfill \vrule height5pt width4pt depth1pt \vspace{+2.00ex}}
\newcommand{\rad}{\operatorname{Rad}}

\newcommand{\ran}{\operatorname{range}}

\newcommand{\cover}{package}
\newcommand{\Ccover}{Cartan \cover}
\newcommand{\envelope}{envelope}
\newcommand{\Inv}{\operatorname{\textsc{Inv}}}
\newcommand{\paut}{\operatorname{\textsc{pAut}}}

\newcommand{\ropen}{\operatorname{\textsc{Ropen}}}
\newcommand{\rideal}{\operatorname{\textsc{Rideal}}}
\newcommand{\ideal}{\operatorname{\text{ideal}}}

\newcommand{\spn}{\operatorname{span}}
\newcommand{\supp}{\operatorname{supp}}

\newcommand{\unit}[1]{#1^{(0)}}

\newcommand{\bbC}{{\mathbb{C}}}

\newcommand{\bbE}{{\mathbb{E}}}

\newcommand{\bbN}{{\mathbb{N}}}

\newcommand{\bbT}{{\mathbb{T}}}

\newcommand{\bbX}{{\mathbb{X}}}

\newcommand{\bbZ}{{\mathbb{Z}}}

  \newcommand{\A}{{\mathcal{A}}}
  
  \newcommand{\B}{{\mathcal{B}}}
  \newcommand{\C}{{\mathcal{C}}}
  
  \newcommand{\D}{{\mathcal{D}}}
  
  \newcommand{\E}{{\mathcal{E}}}

\renewcommand{\H}{{\mathcal{H}}}

  \newcommand{\K}{{\mathcal{K}}} 
\renewcommand{\L}{{\mathcal{L}}}
  \newcommand{\M}{{\mathcal{M}}}
  \newcommand{\N}{{\mathcal{N}}}
\renewcommand{\O}{{\mathcal{O}}}

\renewcommand{\S}{{\mathcal{S}}}
  
  \newcommand{\U}{{\mathcal{U}}}
  
  \newcommand{\W}{{\mathcal{W}}}
  \newcommand{\X}{{\mathcal{X}}}

\newcommand{\fA}{{\mathfrak{A}}}

\newcommand{\fc}{\mathfrak{c}}

\newcommand{\fG}{{\mathfrak{G}}}

\newcommand{\fJ}{{\mathfrak{J}}}

\newcommand{\fM}{{\mathfrak{M}}}

\newcommand{\fq}{{\mathfrak{q}}}
\newcommand{\fR}{{\mathfrak{R}}}
\newcommand{\fr}{{\mathfrak{r}}}
\newcommand{\fS}{{\mathfrak{S}}}
\newcommand{\fs}{{\mathfrak{s}}}

\newcommand{\fX}{{\mathfrak{X}}}

\providecommand{\Eigone}{\E^1}
\providecommand{\cstardiag}{\text{$C^*$-diagonal}}

\newcommand{\fg}{\mathfrak{g}}
\newcommand{\fix}[1]{\operatorname{fix} #1}

\newcommand{\Mod}{\text{Mod}}

\newcommand{\ce}{\E_c}
\newcommand{\ceo}{\E^1_c}
\newcommand{\ceoF}{\E_F^1}
\newcommand{\fRF}{\fR_F}

\begin{document}


\title{Structure for Regular Inclusions. II\\ \tiny Cartan Envelopes,
Pseudo-Expectations and Twists \normalsize}
\author[D.R. Pitts]{David R. Pitts}

\thanks{The author is grateful for the support of the University of
  Nebraska's NSF ADVANCE grant \#0811250 in the completion of this
  paper.  This work was also partially supported by a grant from the Simons
  Foundation (\#316952 to David Pitts).}

\address{Dept. of Mathematics\\
University of Nebraska-Lincoln\\ Lincoln, NE\\ 68588-0130}
\email{dpitts2@unl.edu}


\begin{abstract}
We introduce
the notion of a Cartan envelope for a regular inclusion $(\C,\D)$.
When a Cartan envelope exists, it is the unique, minimal Cartan pair
into which $(\C,\D)$ regularly embeds.  We prove a Cartan envelope
exists if and only if $(\C,\D)$ has the unique faithful
pseudo-expectation property and also give a characterization of the
Cartan envelope using the ideal intersection property.

For any covering inclusion, we construct a Hausdorff
twisted groupoid using appropriate linear functionals and we give a
description of the Cartan envelope for $(\C,\D)$ in terms of a twist
whose unit space is a set of states on $\C$ constructed using the
unique pseudo-expectation.   For a regular MASA inclusion,
this twist  differs from the
Weyl twist; in this setting, we show that the Weyl
twist is Hausdorff precisely when there exists a conditional
expectation of $\C$ onto  $\D$.

We show that a regular inclusion with the unique pseudo-expectation
property is a covering inclusion and give other consequences of the
unique pseudo-expectation property.

\end{abstract}

\maketitle

\setcounter{tocdepth}{1}
\tableofcontents

\section{Introduction}\label{intropre}

In their influential 1977 paper~\cite{FeldmanMooreErEqReII}, Feldman
and Moore showed that the collection of pairs $(\M,\D)$ consisting of
a Cartan MASA $\D$ in the separably acting von Neumann algebra $\M$
is, up to suitable notions of equivalence, equivalent to the family
$(R,\sigma)$ consisting of measured equivalence relations and
$2$-cocycles $\sigma$ on $R$.  The success of the Feldman-Moore
program naturally led to attempts to find appropriate \cstar-algebraic
analogs.  An early attempt was by Kumjian in
1986~\cite{KumjianOnC*Di}, who introduced the notion of
\cstar-diagonals and proved a Feldman-Moore type result for them 
using suitable twists.  However, Kumjian's setting was somewhat 
restrictive, and excluded several classes of desirable examples.  In a
2008 paper, Renault~\cite{RenaultCaSuC*Al} extended Kumjian's work.
Renault gave a definition of a Cartan MASA $\D$ in a \cstaralg\ $\C$
and gave a method for associating a twist $(\Sigma, G)$
to each such pair $(\C,\D)$.  The
philosophy is to loosely regard the passage from $(\C,\D)$ to
$(\Sigma,G)$ as somewhat akin to ``analysis'' in harmonic analysis.
It is of course an interesting problem to determine when the original
regular inclusion can be reconstructed (``synthesized'') from
$(\Sigma, G)$.  In~\cite{RenaultCaSuC*Al}, Renault shows the class of
Cartan inclusions is, to use Leibnitz's immortal phrase, `the best of
all possible worlds.'
Indeed, for any Cartan inclusion $(\C,\D)$, the topologies on $\Sigma$ and $G$
are Hausdorff, and the associated twist
$(\Sigma, G)$
 contains enough of the information about
$(\C,\D)$ to completely recover  $(\C,\D)$.  More precisely, 
Renault shows that if
$\unit{G}$ is the unit space of $G$ and $C^*_r(\Sigma,G)$
denotes the reduced \cstaralg\ of $(\Sigma,G)$, then
$(C^*_r(\Sigma,G), C(\unit{G}))$ is a Cartan inclusion isomorphic
to the original Cartan inclusion $(\C,\D)$.   Thus, for Cartan
inclusions, both analysis and synthesis are possible.  
With his results,
Renault makes a very convincing case that his definition of Cartan
MASA for \cstaralg s is the appropriate analog of the Feldman-Moore
notion of a Cartan MASA in a von Neumann algebra.

While Renault's notion of Cartan MASA appears in a wide variety of
examples, there are also quite natural examples of regular MASA
inclusions $(\C,\D)$ which are not Cartan because they lack a
conditional expectation of $\C$ onto $\D$.  A large class of examples
of regular MASA inclusions with no conditional expectation which arise
from crossed products of abelian \cstaralg s by discrete groups is
constructed in~\cite[Section~6.1]{PittsStReInI}. The lack of a
conditional expectation leads to serious problems when one attempts to
apply the Kumjian-Renault methods to coordinatize $(\C,\D)$ using a
twist.  Indeed, Theorem~\ref{CET2} below shows that for a regular MASA
inclusion $(\C,\D)$, the associated Weyl groupoid $G$ is Hausdorff if
and only if there is a conditional expectation of $\C$ onto $\D$.
Thus we are confronted with the problem of whether a suitable
coordinatization of such pairs $(\C,\D)$ exists, and what that would
mean.  If one is willing to utilize non-Hausdorff twists, it is
possible to obtain a Kumjian-Renault type characterization of a class
of non-Cartan inclusions, and this was recently done
in~\cite{ExelPittsChGrC*AlNoHaEtGr}.  However, here we shall primarily
be interested in Hausdorff twists.

One approach to analyzing a non-Cartan inclusion is to attempt to
embed it into a Cartan inclusion.  In~\cite[Theorem~5.7]{PittsStReInI}
we characterized when a regular inclusion $(\C,\D)$ regularly embeds
into a \cstar-diagonal, or equivalently, when it embeds into a Cartan
inclusion.  Applying this result produces a Cartan pair $(\C_1,\D_1)$
into which $(\C,\D)$ embeds, but $(\C_1,\D_1)$ is in general not
closely related to the original pair $(\C,\D)$.

To address this issue, we introduce the notion of a \textit{Cartan
  envelope} for a regular inclusion $(\C,\D)$, see
Definition~\ref{Cenvdef}. This is the ``smallest'' Cartan pair
$(\C_1,\D_1)$ into which the original pair $(\C,\D)$ can be regularly
embedded.  We show the Cartan envelope is unique when it exists, and
that the image of $\C$ in $\C_1$ is dense in a suitable pointwise
topology.

In~\cite{PittsStReInI}, we introduced the notion of a
pseudo-expectation for an inclusion $(\C,\D)$.  For some purposes,
pseudo-expectations can be used as a replacement for a conditional
expectation.  The advantage of pseudo-expectations is that they always
exist, and for regular MASA inclusions, are unique
\cite[Theorem~3.5]{PittsStReInI}.  Furthermore, a regular inclusion is
a Cartan inclusion if and only if it has a unique pseudo-expectation
which is actually a faithful conditional expectation (see
Proposition~\ref{!fps->vc}(b) below).  Thus, regular inclusions with a
unique and faithful pseudo-expectation are a natural class of regular
inclusions containing the Cartan inclusions.  We do not know a
characterization of those regular inclusions $(\C,\D)$ for which the
pseudo-expectation is unique.

The issue of existence of a Cartan envelope for $(\C,\D)$ is addressed
in Theorem~\ref{ccequiv}:   we characterize the
regular inclusions $(\C,\D)$ which have Cartan envelope as those
which have a unique pseudo-expectation
which is also \textit{faithful}.  We also characterize the existence
of the Cartan envelope in terms of the ideal intersection property.  

Suppose $(\C,\D)$ is a regular inclusion having  Cartan envelope
$(\C_1,\D_1)$.  If $(\Sigma_1, G_1)$ is the twist associated to
$(\C_1,\D_1)$, elements of $\Sigma_1$ and $G_1$ can be viewed as
functions (non-linear in the case of $G_1$) on $\C_1$, and by
restricting these functions to the image of $\C$ under the embedding
of $(\C,\D)$ into $(\C_1,\D_1)$, we obtain families of functions on
$\C$.  These restriction mappings are both one-to-one.  In this way,
$(\Sigma_1, G_1)$ may be thought of as a ``weak-coordinitization'' of
$(\C,\D)$, or as a weak form of ``spectral analysis'' for $(\C,\D)$.
Unsurprisingly, it is possible for two distinct regular
inclusions to have the same Cartan envelope, so in general it is not
possible to synthesize the original inclusion from a
weak-coordinitization without further data.  We give examples of this
phenomena in Example~\ref{noce1}.

For a Cartan MASA $\D$ in a von Neumann algebra $\M$, Aoi's theorem
shows that $\D$ is also a Cartan MASA in any intermediate von Neumann
subalgebra $\D\subseteq \N\subseteq \M$, see
\cite[Theorem~1.1]{AoiCoEqSuInSu}
or~\cite[Theorem~2.5.9]{CameronPittsZarikianBiCaMASAvNAlNoAlMeTh} for
an alternate approach which does not require separability of the
predual.   While Aoi's theorem is not true in full generality in the
\cstaralg\ setting, partial results are obtained
in~\cite{BrownExelFullerPittsReznikoffInC*AlCaEm}.   In a 
sense, Proposition~\ref{subsCart}
of the present paper complements these
results:  it  gives a description of 
those regular subinclusions $(\C_0,\D_0)$ of a given
Cartan pair $(\C,\D)$ which are ``nearly intermediate''  in the sense
that  $\D_0$ is an essential subalgebra of $\D$.

In~\cite[Section~4]{PittsStReInI}, we introduced the notion of a
\textit{compatible state} for a regular inclusion $(\C,\D)$.  The
restriction of any compatible state on $\C$ to $\D$ is a pure state on
$\D$, and when the regular inclusion $(\C,\D)$ has enough compatible
states to cover $\hat\D$, it is a \textit{covering inclusion}.  We
define the notion of a compatible cover for $\hat\D$ (see
Definition~\ref{covering}) and Theorem~\ref{inctoid} shows that
associated to each compatible cover, there is a Hausdorff twist.  When
$(\C,\D)$ has a Cartan envelope, Theorem~\ref{upse=>cov} shows it has
a minimal (necessarily compatible) cover, and by
Corollary~\ref{inctoidCor3}, the twist associated to the minimal cover
is the twist for the Cartan envelope.

We now give an outline of the sections of the paper.  Section~\ref{GenPrelim} gives
provides a reference for some notation and preliminary results. 
Section~\ref{twc*al} is also a preliminary section, but deals
with twists and reduced \cstaralg s associated to Hausdorff twists.
Section~\ref{CEHWG} establishes our motivational result that 
 the Weyl groupoid of a regular MASA inclusion is Hausdorff if and
 only if there is a conditional expectation.

 Our main results are in Sections~\ref{ccce}, \ref{UPSEC}, and \ref{GpReIn}.  In
 Section~\ref{ccce} we introduce and describe the Cartan envelope.
 Theorem~\ref{ccequiv} shows uniqueness and minimality of the Cartan
 envelope and characterizes its existence in terms of essential
 inclusions and also the unique faithful pseudo-expectation property.
 Section~\ref{UPSEC} provides some interesting structural consequences
 of the unique pseudo-expectation property and we propose
 Conjecture~\ref{upseconj} as possible characterizations for regular
 inclusions with the unique pseudo-expectation property.
 Example~\ref{narelco} provides a negative answer to
 \cite[Question~5]{PittsZarikianUnPsExC*In}.  Finally,
 Section~\ref{GpReIn} contains a main result, Theorem~\ref{inctoid},
 which associates a twist to each compatible cover for an inclusion
 $(\C,\D)$. A consequence of this result is description of the twist
 associated to the Cartan envelope of a regular inclusion.  The
 results of Section~\ref{GpReIn} are refinements and improvements of
 the results contained in Section~8 of our
 preprint~\cite{PittsStReIn}.

\textit{Acknowledgments:}  We  thank  Jon Brown, Allan Donsig, Ruy Exel,
Adam Fuller, and Vrej Zarikian for numerous helpful conversations.

\section{General Preliminaries} \label{GenPrelim} Throughout the
paper, unless otherwise stated, all \cstaralg s will be assumed
unital, and a \cstar-subalgebra $\A$ of the \cstaralg\ $\B$ will
usually be assumed to contain the identity of $\B$.  We will often use the
notation $(\B,\A)$ to indicate that $\A$ is a unital \cstar-subalgebra
of the \cstaralg\ $\B$.  For any \cstaralg\ $\A$, we let $\U(\A)$ be
the unitary group of $\A$.

We recall some terminology and notation.
For a Banach space $A$, we use $\dual{A}$ for the Banach space dual;
likewise when $u: A\rightarrow B$ is a bounded linear mapping between
Banach spaces $A$ and $B$, $\dual u: \dual B\rightarrow \dual A$ is
the usual Banach space adjoint of $u$.

Let $\A$ be a \cstaralg.  For a state $f$ on $\A$, $L_f$ denotes its
left kernel, \[L_f:=\{x\in \A: f(x^*x)=0\}.\] A pair $(\B,\alpha)$
consisting of a \cstaralg\ $\B$ and a $*$-monomorphism
$\alpha: \A\rightarrow \B$ is an \textit{extension} of $\A$.  The
extension $(\B,\alpha)$ is a \textit{\cstar-essential extension}, or
more simply an \textit{essential extension}, of $\A$ when the
following condition holds: if $J\subseteq \B$ is a (closed, two-sided)
ideal with $J\cap \alpha(\A)=(0)$, then $J=(0)$.   When
$\A\subseteq \B$ and $\alpha$ is the inclusion mapping, we will
sometimes say that the pair $(\B,\A)$ is essential or that the
extension $(\B,\subseteq)$ is essential. If  
$\A\subseteq \B$,  some
authors   say  $(\B,\A)$ has the \textit{ideal intersection
  property} when $(\B,\A)$ is an essential inclusion.

The commutative setting will play a role in the sequel.  When $X$ and $Y$  are compact 
Hausdorff spaces and $h: Y\twoheadrightarrow X$ is a continuous
surjection, the pair $(Y,h)$ is called a \textit{cover} for $X$.  The
cover $(Y,h)$ is an \textit{essential cover} if $Y$ is the only closed
subset $F$ of $Y$ such that $h(F)=X$.   The following fact is left to
the reader.
\begin{lemma}\label{ess<->ess}  Suppose $X$ and $Y$ are compact
  Hausdorff spaces and $(C(Y),\alpha)$ is an extension of $C(X)$.
  Then $(C(Y),\alpha)$ is an essential extension of $C(X)$ if and only
  if $(Y,\dual\alpha|_Y)$ is an essential cover for $X$.
\end{lemma}
\begin{remark}{Remark}  In the setting of Lemma~\ref{ess<->ess},  we
  will usually  abuse notation and write $\dual\alpha$ instead of
  $\dual\alpha|_Y$.
\end{remark}

For a topological space $X$ and a continuous function $f:X\rightarrow
\bbC$, we shall break with convention and write
\[\supp(f):=\{x\in X: f(x)\neq 0\}\] for the \textit{open} support of
$f$.    When $\D$ is an abelian \cstaralg\ and $d\in \D$,
we will sometimes write $\supp(d)$ instead of $\supp(\hat d)$.

A \textit{partial homeomorphism} of a topological space
  $X$ is a homeomorphism between two open subsets of $X$. If $s_1$ and
  $s_2$ are partial homeomorphisms, their product $s_1s_2$ has domain
  $s_2^{-1}(\dom(s_1)\cap \ran(s_2))$ and for $x\in X$,
  $(s_1s_2)(x)=s_1(s_2(x))$.   We shall use the symbol $\Inv(X)$ for
  the inverse semigroup of all partial homeomorphisms of $X$.  

Dually, if $\D$ is an abelian \cstaralg, a \text{partial automorphism} is a
$*$-isomorphism between two closed ideals of $\D$.  If $\alpha_1$ and
$\alpha_2$ are partial automorphisms, their product 
$\alpha_1\alpha_2$ has domain
\[\alpha_2^{-1}(\dom(\alpha_1)\cap\ran(\alpha_2))\] and for
$d\in\dom(\alpha_1\alpha_2)$,
$(\alpha_1\alpha_2)(d)=\alpha_1(\alpha_2(d))$.  We use the symbol
$\paut(\D)$ for the inverse semigroup of all partial automorphisms of
$\D$.   The semigroups $\paut(\D)$ and $\Inv(\hat\D)$ are
isomorphic via the map $\paut(\D)\ni \tau\mapsto \dual{\tau^{-1}{}}|_{\hat\D}$.
  
An 
 \textit{inclusion} is a pair $(\C,\D)$ of unital \cstaralg s
 (with the same unit) with $\D$ abelian, and $\D\subseteq \C$.  Let
 \[\D^c:=\{x\in \C: xd=dx \, \, \forall\,\, d\in\D\}\]
be the relative commutant of $\D$ in $\C$.  
When $\D$ is a MASA in $\C$, we refer to $(\C,\D)$ as a
\textit{ MASA inclusion}.  For any inclusion, an element of the set
\[\N(\C,\D):=\{v\in \C: v^*\D v\cup v\D v^*\subseteq \D\}\] is called
a \textit{normalizer}.  The inclusion $(\C,\D)$ is a \textit{regular
  inclusion} if 
$\N(\C,\D)$ has dense
span in $\C$.  
Observe that when $(\C,\D)$ is an inclusion, $\U(\D^c)\subseteq
\N(\C,\D)$.  
Clearly $\N(\D^c,\D)\subseteq \N(\C,\D)$, and a
routine argument shows
\begin{equation}\label{NDcdes}
\N(\D^c,\D)=\{v\in \D^c: v^*v= vv^*\in \D\}.
\end{equation} 

Given two inclusions $(\C_1,\D_1)$ and $(\C_2,\D_2)$, a
$*$-homomorphism $\alpha: \C_1\rightarrow \C_2$ is a \textit{regular
  homomorphism} if
\[\alpha(\N(\C_1,\D_1))\subseteq \N(\C_2,\D_2).\]  We will sometimes
write $\alpha: (\C_1,\D_1)\rightarrow (\C_2,\D_2)$ to indicate $\alpha$ is a
regular $*$-homomorphism.

The following will be used so frequently that we state it as a formal definition.
\begin{definition}\label{CartanDef}
A regular MASA inclusion $(\C,\D)$
is a \textit{Cartan inclusion} if there is a faithful conditional expectation
$\bbE:\C\rightarrow \D$. 
\end{definition}

The following notions will play an important role in the sequel.

\begin{definition}
Let $(\C,\D)$ be an inclusion and let $(I(\D),\iota)$ be an
injective envelope for $\D$.
\begin{enumerate}
  \item A \textit{pseudo-expectation relative
to $(I(\D),\iota)$} is a unital completely positive map
$E:\C\rightarrow I(\D)$ such that $E|_\D=\iota$.  When the choice of injective
envelope is clear from the context, we will use the simpler term,
pseudo-expectation.
\end{enumerate}
The 
injectivity of $I(\D)$ ensures pseudo-expectations always exist, but in general, there
are many pseudo-expectations.
\begin{enumerate}
  \setcounter{enumi}{1}
\item When there is a unique pseudo
expectation, we will say $(\C,\D)$ has the \textit{unique pseudo-expectation
  property}.
\item If $(\C,\D)$ has the unique pseudo-expectation
property and the pseudo-expectation is faithful, $(\C,\D)$ is said to
have the \textit{faithful unique pseudo-expectation property}.
\end{enumerate}
\end{definition}

Every regular MASA inclusion has the unique pseudo-expectation
property~\cite[Theorem~3.5]{PittsStReInI}.  When $(\C,\D)$ is a Cartan
inclusion, and $\bbE:\C\rightarrow \D$ is the conditional expectation,
then $\iota\circ\bbE$ is the unique pseudo-expectation for $(\C,\D)$,
so $(\C,\D)$ has the faithful unique pseudo-expectation property.
Interestingly, the 
faithful unique pseudo-expectation property can be used to
characterize Cartan inclusions as we shall see in
Proposition~\ref{!fps->vc} below.

For
Cartan inclusions, we
will sometimes abuse terminology and say that $\bbE$ is the
pseudo-expectation.

Every $v\in\N(\C,\D)$ determines a partial homeomorphism $\beta_v$ of the
Gelfand space $\hat\D$ of $\D$ in the following way:
$\dom\beta_v=\{\sigma\in\hat\D: \sigma(v^*v)>0\}$,
$\ran(\beta_v)=\{\sigma\in\hat\D: \sigma(vv^*)>0\}$, and for
$\sigma\in\dom(\beta_v)$ and $d\in\D$, 
\[\beta_v(\sigma)(d)=\frac{\sigma(v^*dv)}{\sigma(v^*v)}.\]
Then $\beta_v^{-1}=\beta_{v^*}$.   The map $\N(\C,\D)\ni v\mapsto \beta_v$ is
multiplicative on $\N(\C,\D)$ and the collection
\[\W(\C,\D):=\{\beta_v: v\in \N(\C,\D)\}\] is an inverse semigroup of
partial homeomorphisms of $\hat\D$.  The semi-group $\W(\C,\D)$ is
sometimes called the \textit{Weyl pseudo-group} or \textit{Weyl
  semi-group} for the inclusion (see~\cite{RenaultCaSuC*Al}).

Given $v\in\N(\C,\D)$, we now describe a partial isomorphism
$\theta_v$ of $\D$ which is dual to the homeomorphism $\beta_v$ of
$\hat\D$ discussed above.  The map $\D vv^*\ni vv^*h\mapsto v^*hv$
extends uniquely to a $*$-isomorphism $\theta_v$ of
$\overline{\D vv^*}$ onto $\overline{\D v^*v}$,
see~\cite[Lemma~2.1]{PittsStReInI}.  Note that for every
$\sigma\in\dom\beta_v$,
\[\sigma\circ\theta_v=\beta_v(\sigma).\]

For any inclusion $(\C,\D)$ and $\sigma\in\hat\D$,
let  \[\Mod(\C,\D)=\{\rho\in
\text{States}(\C):
\rho|_\D\in\hat\D\}\] be the set of extensions of pure states on $\D$
to $\C$.    Equipped with the relative weak-$*$
topology, $\Mod(\C,\D)$ is a compact set.  Elements of $\Mod(\C,\D)$
 have a property reminiscent of module homomorphisms:
for $\rho\in\Mod(\C,\D)$, the Cauchy-Schwartz inequality shows
that for $h,k\in \D$ and $x\in \C$,
\begin{equation}\label{CSMod}
\rho(hxk)=\rho(h)\rho(x)\rho(k).
\end{equation}  

For any $v\in \N(\C,\D)$ the
map $\beta_v$ extends to a partial homeomorphism $\tilde\beta_v$ of
$\Mod(\C,\D)$:  if $\rho\in \Mod(\C,\D)$ satisfies $\rho(v^*v)>0$, put
$\displaystyle\tilde\beta_v(\rho)(x)= \frac{\rho(v^*xv)}{\rho(v^*v)}$. If $r: \Mod(\C,\D)\rightarrow \hat\D$ is the restriction
map, $r(\rho)=\rho|_\D$, then  
\[\beta_v\circ r=r\circ\tilde\beta_v.\]   
\begin{definition}
A subset $X\subseteq
\Mod(\C,\D)$ \textit{covers $\hat\D$} if $r(X)=\hat\D$ and is \textit{$\N(\C,\D)$-invariant} if for every
$v\in\N(\C,\D)$, $\tilde\beta_v(X\cap \dom\tilde\beta_v)\subseteq X$.
\end{definition}

\begin{proposition}[{\cite[Proposition~2.7]{PittsStReInI}}] \label{invideal}  Let $(\C,\D)$ be a regular
  inclusion and suppose that $F\subseteq \Mod(\C,\D)$ is 
  $\N(\C,\D)$-invariant.  Then the set 
\[\K_F:=\{x\in\C: \rho(x^*x)=0\text{ for all } \rho\in F\}\] is a
  closed, two-sided ideal in $\C$.  Moreover, if 
  $\{\rho|_\D: \rho\in F\}$ is weak-$*$ dense in $\hat{\D}$, 
then $\K_F\cap \D=(0).$   
\end{proposition}

In~\cite{PittsStReInI}, we
introduced the
\textit{compatible states}, defined by
\[\fS(\C,\D):=\{\rho\in\text{States}(\C): \text{ for every } v\in \N(\C,\D),
  |\rho(v)|^2\in \{0,\rho(v^*v)\}\}.\]   Then $\fS(\C,\D)\subseteq
\Mod(\C,\D)$ and $\fS(\C,\D)$ is a weak-*
closed, $\N(\C,\D)$-invariant set of states.  However, it is not in
general the case that $\fS(\C,\D)$ covers $\hat\D$; in fact, there
exist regular inclusions for which 
$\fS(\C,\D)$ is empty, see \cite[Theorem~4.13]{PittsStReInI}.

\begin{definition}\label{covering}  A regular inclusion $(\C,\D)$ is
  \textit{a covering inclusion} if 
$\fS(\C,\D)$ covers $\hat\D$.  A subset $F\subseteq \fS(\C,\D)$ is a \textit{compatible cover} for $\hat\D$
if $F$ is weak-$*$ closed, $\N(\C,\D)$-invariant, and covers $\hat\D$.
\end{definition}

\begin{remark}{Examples}\label{sufficientcov}
\begin{enumerate}
\item Theorem~\ref{upse=>cov} below shows that any regular
  inclusion with the unique pseudo-expectation property is a covering
  inclusion.
  In particular, it follows from~\cite[Lemma~2.10 and Proposition~4.6]{PittsStReInI}
  that every regular inclusion $(\C,\D)$ for which $\D^c$ is abelian
  is a covering inclusion.

  Example~\ref{narelco} gives a construction of a regular inclusion
  $(\C,\D)$ with the unique pseudo-expectation property such that
  $\D^c$ is non-abelian.

\item Here is an elementary example of a covering inclusion $(\C,\D)$ such
  that $\D^c$ is non-abelian and which does not have the unique
  pseudo-expectation property.  Let  $\C:=M_2(\bbC)\oplus\bbC$ and
  $\D:=\bbC I_\C$.  Then $(\C,\D)$ is a regular inclusion and
  $M_2(\bbC)\oplus \bbC\ni x\oplus \lambda\mapsto \lambda$ is a
  compatible state.  Thus $(\C,\D)$ is a covering inclusion, yet
  $\D^c=\C$, is not abelian.  Any state on $\C$
  is a pseudo-expectation.

\item By definition, no regular inclusion $(\C,\D)$ with
  $\fS(\C,\D)=\emptyset$ is a covering
  inclusion.   When $\C$ is
simple \cstaralg, then $(\C,\bbC I)$ is a regular inclusion with 
  $\fS(\C,\bbC I)=\emptyset$ (see \cite[Theorem~4.13]{PittsStReInI}).  
  \end{enumerate}
\end{remark}

Associated to an ideal $J$ of a unital abelian \cstaralg\ $\D$ are the
two ideals,
\[J^\perp:=\{d\in \D: dJ=0\}\dstext{and} J^\dperp:= (J^\perp)^\perp.\]
When $J=J^\dperp$, 
$J$ called a \textit{regular ideal}.  Also for $d\in\D$, we will
write $d^\perp$ for the ideal $\{h\in\D: dh=0\}$; $d^\dperp$ is defined
similarly.   
Let \[\supp(J):=\{\tau\in\hat\D:
\tau|_J\neq0\}.\]  Then $J$ is a regular ideal if and only if
$\supp(J)$ is a regular open set.  Furthermore, for any ideal
$J\subseteq \D$,
\[\supp(J^\dperp)=\left(\overline{\supp J}\right)^\circ.\]
On the other hand, for an open set $G\subseteq \hat\D$, let
\[\ideal(G):=\{d\in\D: \supp(\hat d)\subseteq G\}.\]  Then $G$ is a
regular open set if and only if $\ideal(G)$ is a regular ideal.

Let $\rideal(\D)$ and $\ropen(\hat\D)$ be the Boolean
algebras of regular ideals of $\D$ and regular open sets in $\hat\D$
respectively.  These are isomorphic Boolean algebras under the map
$\rideal(\D)\ni J\mapsto \supp(J)$.

Essential extensions of abelian \cstaralg s have  isomorphic
Boolean algebras of regular ideals (and hence isomorphic Boolean
algebras of regular open sets).
The proof of the following fact is not trivial, but due to length
considerations, we leave the proof to the reader.
\begin{lemma}\label{isolattice}  Suppose $\A$ and $\B$ are abelian
  \cstaralg s, $(\B,\alpha)$ is an essential
  extension of $\A$, and  $r:\hat\B\rightarrow \hat\A$ is the
  continuous surjection,
  $\rho\in\hat\B\mapsto \rho\circ \alpha$.  Then the maps
  \[\rideal(\B)\ni J\mapsto \alpha^{-1}(J) \dstext{and}
    \ropen(\hat\A) \ni G\mapsto r^{-1}(G)\] are Boolean algebra
  isomorphisms of $\rideal(\B)$ onto $\rideal(\A)$ and
  $\ropen(\hat\A)$ onto $\ropen(\hat\B)$ respectively.  The inverses
  of these maps are
  \[\rideal(\A)\ni K\mapsto \alpha(K)^\dperp\dstext{and}
    \ropen(\hat\B)\ni H\mapsto (r(\overline{H}))^\circ\]
  respectively.  Furthermore, for $J\in\rideal(\B)$ and $G\in \ropen(\hat\A)$, 
  \[\supp(\alpha^{-1}(J))=
    (r(\overline{\supp(J)})^\circ\dstext{and}\ideal(r^{-1}(G))=\alpha(\ideal(G))^\dperp.\]
\end{lemma}

Let $(\C,\D)$ be an inclusion and $v\in\N(\C,\D)$.
In~\cite[Definition~2.13]{PittsStReInI}, we introduced the notion of a
\textit{Frol\' ik family of ideals for $v$}.  This is a set of five
regular ideals $\{K_i\}_{k=0}^4$ of $\D$, with the property that for
$i=1,2,3$, $K_i\theta_v(K_i)=0$, $K_4=(vv^*)^\dperp$, and $K_0$ is the
\textit{fixed point ideal for $v$}.  The following describes the
fixed point ideal.
\begin{lemma}[{\cite[Lemma~2.15]{PittsStReInI}}]\label{fixedideal} Let $(\C,\D)$ be an 
  inclusion and let $v\in \N(\C,\D)$.  Then
  \[K_0=\{d\in (vv^*\D)^\dperp: vd=dv\in\D^c\}=\{d\in (v^*v\D)^\dperp:
  vd=dv\in\D^c\}. \]
\end{lemma}
For $v\in \N(\C,\D)$, $\dom \beta_v=\supp(\overline{v^*v\D})$.  In
general, $\supp(K_0)$ need not be contained in $\dom\beta_v$, but they
are intimately related.  Indeed, since
$\dom\beta_v\cap \supp(K_0)=\supp(K_0\cap \overline{v^*v\D})$, the
next lemma shows that
$\dom\beta_v\cap \supp(K_0)=(\fix\beta_v)^\circ$, and furthermore,
$\supp(K_0)$ is the interior of the closure of $(\fix\beta_v)^\circ$.
This technical fact  will play a useful role in Section~\ref{CEHWG},
but it seems convenient to place it here.
\begin{lemma}\label{fixptchar}  Let $(\C,\D)$ be an
  inclusion and $v\in\N(\C,\D)$.
  Then
  \begin{equation}\label{fixedideal0}
    \supp(K_0\cap \overline{v^*v\D})=(\fix\beta_v)^\circ\dstext{and}
    K_0=(K_0\cap \overline{v^*v\D})^\dperp.
  \end{equation}
\end{lemma}
\begin{proof}    Let $J:=K_0\cap \overline{v^*v\D}$. Suppose $\sigma\in\hat\D$ and $\sigma|_J\neq 0$.  Let
  $d\in J$ be such that $\sigma(d)\neq 0$.  Suppose $\rho\in\hat\D$
  and $\rho(d)\neq 0$.  Then for any $h\in\D$, $dh\in K_0$, so
  $dhv=vdh\in\D^c$.  Thus, 
  \[\beta_v(\rho)(h)=\frac{\beta_v(\rho)(hd)}{\rho(d)}=\frac{\rho(hd)}{\rho(d)}=\rho(h),\]
  whence $\rho\in\fix\beta_v$.  As this holds for every such $\rho$,
  $\sigma\in (\fix\beta_v)^\circ$.  Therefore,
  $\supp(J)\subseteq (\fix\beta_v)^\circ$.

For the converse, suppose $d\in\D$ satisfies $\supp\hat d\subseteq
(\fix\beta_v)^\circ$.   We first show that
for every $\rho\in\hat{\D}$, 
\begin{equation}\label{acom}
\rho(v^*dv)=\rho(v^*vd).
\end{equation}
Let $\rho\in\hat{\D}$.  There are three cases.  First suppose
$\rho(v^*v)=0$. The Cauchy-Schwartz inequality gives,
\[|\rho(v^*dv)|\leq
\rho(v^*d^*dv)^{1/2}\rho(v^*v)^{1/2}=0,\] so \eqref{acom} holds when
$\rho(v^*v)=0$.  

Suppose next that $\rho(v^*v)>0$ and $\beta_v(\rho)(d)\neq 0$.
Then $\beta_v(\rho)\in \supp(\hat{d})$, so $\beta_v(\rho)\in
\fix{\beta_v}=\fix{\beta_{v^*}}$.  Thus, we get $\beta_v(\rho)=
\beta_{v^*}(\beta_v(\rho))=\rho$,
and hence $\rho(v^*dv)=\rho(v^*v)\rho(d)=\rho(v^*vd)$.

Finally suppose that $\rho(v^*v)>0$ and $\beta_v(\rho)(d)=0$.  Then
$\rho(v^*dv)=0$.  We shall show that $\rho(d)=0$.  If not,  the hypothesis
on $d$ shows that $\rho\in\fix{\beta_v}$.  Hence, 
$0\neq\rho(d)=\beta_v(\rho)(d)=\frac{\rho(v^*dv)}{\rho(v^*v)}=0$,
which is absurd.  So $\rho(d)=0$, and~\eqref{acom} holds in
this case also.    Thus we have established \eqref{acom} in all cases.

Thus $v^*dv=v^*vd$.   So for every $n\in \bbN$, 
\[0=v^*dv-v^*vd=v^*(dv-vd)=vv^*(dv-vd)=(vv^*)^n(dv-vd).\]  It follows
that for every polynomial $p$ with $p(0)=0$, 
$p(vv^*)(dv-vd)=0$.  Therefore, for every $n\in\bbN$,
$$0=(vv^*)^{1/n}(dv-vd)=d(vv^*)^{1/n}v-(vv^*)^{1/n}vd.$$  Since
$\lim_{n\rightarrow \infty} (vv^*)^{1/n}v=v$, we have $vd=dv$.
Clearly if $h\in\D$, then
$\supp(\widehat{dh})\subseteq (\fix\beta_v)^\circ$, so that
$vdh=hdv=hvd$, whence $vd\in\D^c$.  The first equality in~\eqref{fixedideal0} now
follows.

To prove the second equality in~\eqref{fixedideal0}, we will show that
\begin{equation}\label{fixptchar00}
  K_0\cap J^\perp=(0).
\end{equation}
So suppose  that $d\in K_0$ and $d\in J^\perp$.  As $d\in K_0$, $d\in
\{v^*v\}^\dperp$, so \begin{equation}\label{fixptchar1}
  \supp
\hat d\subseteq 
\overline{\supp(\widehat{v^*v})}.
\end{equation}
We claim that
\begin{equation}\label{fixptchar2}
  \supp\hat
  d\cap \supp(\widehat{v^*v})\subseteq (\fix\beta_v)^\circ.
\end{equation}
Notice that for  $h\in \D$,
\begin{equation}\label{fixptchar3}
  (v^*hv-v^*vh)d=v^*hvd-v^*vhd=0
\end{equation}
because $vd\in\D^c$.  Let
$\rho\in \supp\hat d\cap \supp \widehat{v^*v}$.  Since
$\rho(d)\rho(v^*v)\neq 0$, applying $\rho$ to~\eqref{fixptchar3} gives
$\beta_v(\rho)(h)=\rho(h)$ for every $h\in\D$.  The inclusion
\eqref{fixptchar2} follows.

If $\rho\in( \fix\beta_v)^\circ$, we may choose $h\in J$ such
that $\rho(h)\neq 0$.  Since $d\in J^\perp$, $dh=0$, so $\rho(d)=0$.
Thus $\hat d$ vanishes on $(\fix\beta_v)^\circ$.   
Using~\eqref{fixptchar1} and~\eqref{fixptchar2},
we conclude that \[\supp\hat d\subseteq
\overline{\supp(\widehat{v^*v})}\setminus \supp(\widehat{v^*v}),\]
which is a set with empty interior.  Thus $d=0$.  
\end{proof}

\section{Additional Preliminaries: Twists and their \cstaralg s}\label{twc*al}
In this section, we collect  some generalities on twists and the (reduced)
\cstaralg s associated to them for use in Sections~\ref{CEHWG} and~\ref{GpReIn}.  Much of this material
can be found in~\cite{RenaultCaSuC*Al} or~\cite{SimsHaEtGrThC*Al}.
The description of groupoids is standard, and we include it for
notational purposes.  Associated to a twist are a line bundle and its
conjugate, and a well-known construction associates a reduced
\cstaralg\ to each.  These \cstaralg s are anti-isomorphic.  This
unsurprising fact is doubtless known, but because we have not found a
proof of this fact in the literature, we provide a sketch of one in
Proposition~\ref{kpm} below.

There is inconsistency in the literature regarding which line bundle
to choose when constructing the \cstaralg\ associated to a twist---for
example, the line bundle described in~\cite[
Definition~$4^\circ$]{KumjianOnC*Di} is the conjugate of the line
bundle used in~\cite[Section~4, p.\ 39]{RenaultCaSuC*Al}.  However,
Kumjian uses homogeneous functions to define the convolution algebra
while Renault uses conjugate-homogeneous functions, so (as also noted
in Proposition~\ref{kpm}) the reduced \cstaralg s discussed by Kumjian
and Renault are the same.  For the Weyl twist (described in
Example~\ref{Wtw}), it seems more natural to use Renault's choice of
line bundle because the action of $\bbT$ arises from scalar
multiplication on $\N(\C,\D)$.  On the other hand, in
Section~\ref{GpReIn} we construct twists from certain families of
linear functionals on $\C$.  For these twists, it seems natural to use
Kumjian's choice of line bundle because the action of $\bbT$ arises
from scalar multiplication on linear functionals.

Let $G$ be a groupoid.  We use $\unit{G}:=\{g\in G: g=g^{-1}g\}$ for the
unit space of $G$.  We often use $r$ and $s$ for the range and
source maps:  $r(g)=gg^{-1}$ and $s(g)=g^{-1}g$.  Also if $x\in
\unit{G}$,
\[G^x:=\{g\in G: r(g)=x\}\dstext{and} G_x:=\{g\in G: s(g)=x\}.\]   We
will frequently write $xG$ (resp. $Gx$) instead of $G^x$ (resp $G_x$). 

\begin{definition} A locally compact topological groupoid $G$ is an
  \textit{\' etale groupoid} if the range map (or equivalently the source
  map) is a local
  homeomorphism.  A subset $S\subseteq G$ is a \textit{bisection} if
  there is an open set $V\subseteq G$ with $S\subseteq V$ such that
  $r|_V$ and $s|_V$ are both homeomorphisms onto open subsets of
  $\unit{G}$.  An open bisection is sometimes called a \textit{slice}
  (see~\cite[Section~3]{ExelInSeCoC*Al}).
\end{definition}
\'Etale groupoids may be regarded as the analog of discrete
groups.  They have a number of pleasant properties, including:
\begin{itemize}
\item the unit space of an \'etale groupoid is open;
\item for an \'etale groupoid $G$, if $x\in\unit{G}$, then $G_x$ and
  $G^x$ are discrete sets.
\end{itemize}

A twist is the analog of a central extension of a discrete group by
the circle $\bbT$.  Here is the formal definition.
\begin{definition}[{\cite[Definition~5.1.1]{SimsHaEtGrThC*Al}}] \label{twistdef}
  Let
  $\Sigma$ and $G$ be (not necessarily
  Hausdorff) locally compact groupoids with $G$ \'etale, and let
  $\bbT\times \unit{G}$ be the product groupoid.  (That is,
  $(z_1,e_1)(z_2,e_2)$ is defined if and only if $e_1=e_2$, in which
  case the product is $(z_1z_2,e_1)$, inversion is given by $(z,e)^{-1}=(z^{-1},e)$, and the
  topology is the product topology.)  Note that the unit space of
  $\bbT\times\unit{G}$ is $\{1\}\times \unit{G}$.  A \textit{twist}  is a sequence
  \begin{equation}\label{twistD}
    \bbT\times\unit{G}\overset{\iota}\hookrightarrow
    \Sigma\overset{\fq}\twoheadrightarrow G
  \end{equation}
  where
  \begin{enumerate}
    \item\label{twdef1} $\iota$ and $\fq$ are continuous groupoid homomorphisms with
      $\iota$ one-to-one and $\fq$ onto;
    \item \label{twdef3} $\fq^{-1}(\unit{G})=\iota(\bbT\times
        \unit{G})$;
    \item\label{twdef2} $\iota|_{\{1\}\times \unit{G}}$ and $\fq|_{\unit{\Sigma}}$
        are homeomorphisms onto $\unit{\Sigma}$ and $\unit{G}$
        respectively (thus we may, and do, identify $\unit{\Sigma}$ and $\unit{G}$ using $\fq$);
      \item \label{twdef4}  for every $\gamma\in \Sigma$ and $z\in\bbT$,
         $\iota(z,r(\gamma)) \, \gamma =\gamma\, \iota(z,s(\gamma))$; and
       \item\label{twdef5} for every $g\in G$ there is an open
         bisection $U\subseteq G$ with
         $g\in U$ and a continuous function $j_U: U\rightarrow \Sigma$
         such that $\fq\circ j_U=\id|_U$ and the map
         $\bbT\times U\ni (z\times h)\mapsto \iota(z,r(h))\, j_U(h)$
         is a homeomorphism of $\bbT\times U$ onto $\fq^{-1}(U)$.
       \end{enumerate}
For $\gamma\in \Sigma$, we will often denote $q(\gamma)$ by
$\dot\gamma$; indeed, we will usually use  the name $\dot\gamma$
for an arbitrary element of $G$.

  When $G$ is Hausdorff, we shall say  the twist $
    \bbT\times\unit{G}\overset{\iota}\hookrightarrow
    \Sigma\overset{\fq}\twoheadrightarrow G$ is a \textit{Hausdorff twist}.
         For $z\in\bbT$ and $\gamma\in\Sigma$ we will write
       \[z\cdot\gamma:=\iota(z,r(\gamma))\gamma\dstext{and} \gamma\cdot
         z:=\gamma\iota(z,s(\gamma)).\]   This action of $\bbT$ on
       $\Sigma$ is free.

       The map
  $q:\gamma\rightarrow G$ is necessarily a quotient map (see~\cite[Exercise~9K(3)]{WillardGeTo}).
     \end{definition}

\begin{definition}  The twist $\bbT\times
  \unit{G_2}\overset{\iota_2}\hookrightarrow
  \Sigma_2\overset{\fq_2}\twoheadrightarrow G_2$ is an
  \textit{extension} of the twist  $\bbT\times
  \unit{G_1}\overset{\iota_1}\hookrightarrow
  \Sigma_1\overset{\fq_1}\twoheadrightarrow G_1$ if there   are
  are continuous groupoid
  monomorphisms $\theta: G_1\hookrightarrow G_2$ and $\alpha:
  \Sigma_1\hookrightarrow \Sigma_2$ such that $\theta(G_1)$ is closed
  in $G_2$ and 
\[\fq_2\circ \alpha= \theta\circ \fq_1\dstext{and}
  \iota_2\circ (\text{id}|_\bbT\times \theta|_{\unit{G_1}})=\alpha\circ\iota_1.\]
When $\theta$ and $\alpha$ are homeomorphisms and
groupoid isomorphisms,  these 
  twists are \textit{equivalent} or \textit{isomorphic}.   Finally, if
  $\theta$ and $\alpha$ are inclusion maps,  $\bbT\times
  \unit{G_1}\overset{\iota_1}\hookrightarrow
  \Sigma_1\overset{\fq_1}\twoheadrightarrow G_1$ is a
  \textit{subtwist} of $\bbT\times
  \unit{G_2}\overset{\iota_2}\hookrightarrow
  \Sigma_2\overset{\fq_2}\twoheadrightarrow G_2$.
\end{definition}

\begin{remark}{Notation}
       We will denote a twist in several ways: i) by explicitly writing
       a sequence such as~\eqref{twistD}; ii) writing $(\Sigma, G,\iota,
       \fq)$, or iii) when the maps $\iota$ and $\fq$ are understood,
       simply by $(\Sigma,G)$. 
     \end{remark}

\begin{definition}  Given a twist $(\Sigma, G, \iota,\fq)$,  define $\overline\iota:
  \bbT\times \unit{G}\hookrightarrow \Sigma$ by
  $\overline{\iota}(z,e)=\iota(\overline z, e)$.  The  \textit{conjugate
    twist} to $(\Sigma,G,\iota,\fq)$ is the twist $(\Sigma,
  G,\overline\iota, \fq)$.   We will  sometimes denote it by
  $(\overline\Sigma, G,\overline\iota,\fq)$.
\end{definition}

\begin{remark}{Example}
  \label{Wtw} 
  An important example of a twist associated to a regular inclusion
  $(\C,\D)$ is the \textit{Weyl twist}, which we briefly describe
  here.  The construction is due to Kumjian and Renault
  (see~\cite{KumjianOnC*Di} and \cite{RenaultCaSuC*Al}).  Kumjian and
  Renault use different descriptions of the equivalence relation
  determining $G$, but it is known that their descriptions agree for
  regular MASA inclusions, a fact which also follows from
  Lemma~\ref{chareqrel} below.  We use Kumjian's description because
  it yields a twist for every regular inclusion, whereas Renault's
  description need not.

Given the regular   inclusion $(\C,\D)$, set 
\[\fX:=\{(\sigma_2, v,\sigma_1)\in \hat\D\times
  \N(\C,\D)\times\hat\D: \sigma_2\in \ran\beta_v, \sigma_1\in
  \dom\beta_v\text{ and }\sigma_2=\beta_v(\sigma_1)\}.\]  Define
\[\fR_1:=\{((\sigma_2,v',\sigma_1),(\sigma_2,v,\sigma_1))\in
  \fX\times\fX: \exists \, d, d'\in \D \text{ with } \sigma_1(d)>0,
  \sigma_1(d')>0 \text{ and } v'd'=vd \}\] and 
\[\fR_\bbT:=\{((\sigma_2,v',\sigma_1),(\sigma_2,v,\sigma_1))\in
  \fX\times\fX: \exists \, d, d'\in \D \text{ with } \sigma_1(d)\neq
  0, \sigma_1(d')\neq 0 \text{ and } v'd'=vd \}.\] (The notation
$\fR_1$ and $\fR_\bbT$ reflect that fact that the polar parts of
$\sigma(d)$ and $\sigma(d')$ belong to the group $\{1\}$ or $\bbT$.)
Then $\fR_1$ and $\fR_\bbT$ are equivalence relations on $\fX$ with
$\fR_1\subseteq \fR_\bbT$.  Denote the equivalence class of
$(\sigma_2,v,\sigma_1)\in \fX$ relative to $\fR_1$ and $\fR_\bbT$ by
$[\sigma_2,v,\sigma_1]_1$ and $[\sigma_2, v,\sigma_1]_\bbT$
respectively.  Use
\[\Sigma:=\fX/\fR_1\dstext{and} G:=\fX/\fR_\bbT\] to denote the
sets of equivalence classes for $\fR_1$ and $\fR_\bbT$.  Then $\Sigma$ and
$G$ become groupoids with operations defined as follows.  For $\varkappa\in
\{1,\bbT\}$, the pair 
$([\sigma_2,v,\sigma_1]_\varkappa,[\sigma_2',v',\sigma_1']_\varkappa)\in (\fX/\fR_\varkappa)^2$ is composable
 iff $\sigma_1=\sigma_2'$ in which case, the product is $[\sigma_2,
vv',\sigma_1']_\varkappa$; inverses are defined by $[\sigma_2, v, \sigma_1]^{-1}_\varkappa:=[\sigma_1, v^*,
\sigma_2]_\varkappa$.  The unit spaces are:
\[\unit{G}=\{[\sigma, d, \sigma]_\bbT: d\in \D, \sigma(d)\neq 0\}
  \dstext{and}\unit{\Sigma}=\{[\sigma, d, \sigma]_1: d\in \D,
  \sigma(d)>0\}.\] Define
$\iota: \bbT\times \unit{G}\rightarrow \Sigma$ by
$\iota(z, [\sigma, d, \sigma]_\bbT)= [\sigma, z d^*d,
\sigma]_1 $.  Note that $\iota$ maps ${\{1\}\times \unit{G}}$
bijectively onto $\unit{\Sigma}$.  The projection map
$\fq:\Sigma\rightarrow G$ given by
$[\sigma_2,v,\sigma_1]_1\mapsto [\sigma_2, v,\sigma_1]_\bbT$ is a
groupoid homomorphism with $\fq|_{\unit{\Sigma}}$ a bijection of
$\unit{\Sigma}$ onto $\unit{G}$.

Suppose
$[\sigma_2, v, \sigma_1]_1$ and $[\sigma_2', v',
\sigma_1']_1$ satisfy $\fq([\sigma_2, v, \sigma_1]_1)=\fq([\sigma_2', v', \sigma_1']_1 ')$.   Then for $i=1,2$,
$\sigma_i=\sigma_i'$ and there are $d, d'\in\D$ with $\sigma_1(d)$,
$\sigma_1(d')$ both non-zero such that $vd=v'd'$.  Choosing $z,
z'\in\bbT$ so that $z\sigma_1(d)>0$ and $z'\sigma_1(d') > 0$ we obtain
$[\sigma_2,  v', \sigma_1]_1=[\sigma_2, \overline{z} z' v,
\sigma_1]_1=\iota(\overline{z}z', [\sigma_2, I, \sigma_2]_\bbT)\, [\sigma_2,
v, \sigma_1]_1$.    In particular, if $[\sigma_2, v,\sigma_1]_\bbT=
[\sigma_2, v',\sigma_1]_\bbT$, then there exists a unique
$\lambda\in\bbT$ such that $[\sigma_2, v',\sigma_1]_1=[\sigma_2,
\lambda v, \sigma_1]_1$. 
Thus, we obtain the sequence,
\begin{equation}\label{WeylTw}\bbT\times \unit{G}\overset{\iota}\hookrightarrow
  \Sigma\overset{\fq}\twoheadrightarrow G
\end{equation}
satisfying conditions~\eqref{twdef3}
and~\eqref{twdef4} of Definition~\ref{twistdef}.

The next task is to describe the topologies on $G$ and $\Sigma$. 
Let $v\in \N(\C,\D)$ and set 
\[N_\bbT(v):=\{g\in G: \exists \, \sigma_1, \sigma_2\in\hat\D \text{
    such that } g=[\sigma_2, v, \sigma_1]_\bbT\}.\]
The collection of such sets forms a base
for a topology on $G$; also for $g=[\sigma_2,v,\sigma_1]_\bbT\in G$,
\[\{N_\bbT(vh): h\in\D \text{ and }\sigma_1(h)\neq 0\}\] is a local
base at $g$.  With this topology, $G$ becomes an \'etale topological
groupoid, but in general, it need not be Hausdorff.  In fact, we show
in Theorem~\ref{CET2} below that when $(\C,\D)$ is a regular MASA
inclusion,  $G$ is Hausdorff if and only if there
exists a conditional expectation $E:\C\rightarrow \D$.

To describe the topology on $\Sigma$, observe that for
$v\in \N(\C,\D)$, $\bbT\times N_\bbT(v)$ is homeomorphic to
$\fq^{-1}(N_\bbT(v))$ via the map
$\bbT\times N_\bbT(v)\ni (z, [\sigma_2, w, \sigma_1]_\bbT) \mapsto
[\sigma_2, z\, vw^*w, \sigma_1]_1$; use this map to identify
$\bbT\times N_\bbT(v)$ with $\fq^{-1}(N_\bbT(v))$.  A base for a
topology on $\Sigma$ is the collection of subsets of $\Sigma$ of the
form $\O\times N_\bbT(v)$ where $\O\subseteq \bbT$ is open.
With these topologies, ~\eqref{WeylTw} becomes a twist, which is
the \textit{Weyl twist} for the regular  inclusion $(\C,\D)$.
\end{remark}

\subsection*{Line Bundles Over Twists}
Associated to a twist $(\Sigma,G,\iota,\fq)$ and an integer $k\in \{-1,1\}$ is a
line bundle $L_k$ over $G$, which can then be used to construct
convolution algebras, which in turn produce \cstaralg s of interest.
This process has been previously studied by various authors
(e.g.~\cite{RenaultCaSuC*Al,SimsHaEtGrThC*Al})  and a
description of $L_{-1}$ has also been described
in~\cite{BrownFullerPittsReznikoffGrC*AlTwGpC*Al}.   For convenience,
we give a brief description here.

Fix a twist $(\Sigma,G,\iota,\fg)$ and $k\in\{-1,1\}$.  The group
$\bbT$  acts  freely on $\bbC\times \Sigma$:   for $z\in\bbT$, send
$(\lambda,\gamma)$ to
$(\lambda z^k, z\cdot \gamma)$.  Let $L_k$ be  the
quotient of $\bbC\times \Sigma$ by the equivalence relation determined
by this action.  (To be explicit, this equivalence relation is given
by  $(\lambda_1,\gamma_1)\sim_k
(\lambda_2,\gamma_2)$ if and only if there exists $z\in\bbT$ such that
$(\lambda_2,\gamma_2)= (\lambda_1z^k, z\cdot\gamma_1)$.)

Let
$[\lambda,\gamma]_k$ denote the equivalence class of
$(\lambda,\gamma)$ and equip $L_k$ with the quotient topology.  When
$k$ is clear from the context, we shall simplify notation and write
$[\lambda, \gamma]$ instead of $[\lambda,\gamma]_k$.

Notice that for $z\in \bbT$ and $[\lambda, \gamma]\in L_k$,
\begin{equation}\label{Lkact}
  [\lambda, z\cdot\gamma]=[\lambda \overline{z}^k,\gamma].
\end{equation}
The
map $P:L_k\rightarrow G$ given by $P([\lambda,\gamma])=\dot\gamma$ is
a continuous surjection whose fibers are homeomorphic to $\bbC$.
For $\dot\gamma\in G$, there is no canonical choice of $\gamma\in
P^{-1}(\dot\gamma)$, 
however, when  $x\in \unit{G}$,  $P^{-1}(x)\cap \unit{\Sigma}$ is a singleton
set, and, recalling that we have previously identified $\unit{G}$ and
$\unit{\Sigma}$,  we will usually identify $P^{-1}(x)$ with $\bbC$ using the
map $\lambda\mapsto [\lambda, x]$.
When given the following operations, $L_k$ becomes a Fell line bundle
over $G$ (\cite[Definition~2.1]{KumjianFeBuOvGr}):
\begin{description}
  \item[product] the product 
    $[\lambda_1,\gamma_1][\lambda_2,\gamma_2]$ is defined when the
    product $\gamma_1\gamma_2$ is defined in $\gamma$, in which case,
    $[\lambda_1,\gamma_1][\lambda_2,\gamma_2]:=[\lambda_1\lambda_2,\gamma_1\gamma_2]$;
    \item[conjugation]
      $\overline{[\lambda,\gamma]}:=[\overline\lambda, \gamma^{-1}]$;
      \item[scalar multiplication] for $\mu\in\bbC$,
        $\mu[\lambda,\gamma]:=[\mu,r(\gamma)][\lambda,\gamma]=[\mu\lambda,\gamma]$;
        and 
        \item[addition] addition of $[\lambda_1,\gamma_1]$ and
          $[\lambda,\gamma_2]$ is defined when
          $\dot\gamma_1=\dot\gamma_2$, in which case,
          $[\lambda_1,\gamma_1]+[\lambda_2,\gamma_2]=[\lambda_1+\overline{z}^k\lambda_2,
          \gamma_1]$, where $z$ is the (necessarily unique) element of
          $\bbT$  such that $\gamma_2=z\cdot
          \gamma_1$.
        \end{description}
Note that since $G$ is locally trivial, so is $L_k$
(see~\cite[Section~2.2]{BrownFullerPittsReznikoffGrC*AlTwGpC*Al}).   
        
Finally, there is
a  continuous map $\varpi: L\rightarrow [0,\infty)$ given by \[\varpi([\lambda,
\gamma]):=|\lambda|.\] 
When $f: G\rightarrow L$ is a section and $\dot\gamma\in G$, we will often write
$|f(\dot\gamma)|$ instead of $\varpi(f(\dot\gamma))$.

\subsection*{\cstaralg s associated to Hausdorff twists}
While it is possible to define \cstaralg s associated to non-Hausdorff
groupoids or twists (a discussion of this process may be
found in~\cite{ExelPittsChGrC*AlNoHaEtGr}), we will not need \cstaralg s
arising from non-Hausdorff groupoids in the sequel.  Thus, for the
remainder of this section, \textit{we will assume $G$ is Hausdorff.}

Fix a twist $(\Sigma, G,\iota,\fq)$ and $k\in\{-1,1\}$.  
        The \textit{open support} of a  continuous section $f:
        G\rightarrow L_k$ is
        \[\supp(f):=\{\dot\gamma\in G: \varpi(f(\dot\gamma))\neq
          0\},\] and $f$ is 
        \textit{compactly supported} if the closure of $\supp(f)$ is a
        compact subset of $G$.   Let $C_c(\Sigma,G,k)$ denote the linear space of all
        compactly supported continuous sections of $L_k$.  For $f,g\in
        C_c(\Sigma,G,k)$, and $\dot\gamma\in G$, define 
        \begin{equation}\label{Lkops}
          (f\star g)(\dot\gamma):=\sum_{\substack{\dot\gamma_1,\dot\gamma_2\in G,\\
              \dot\gamma_1\dot\gamma_2=\dot\gamma}}
          f(\dot\gamma_1)g(\dot\gamma_2) \dstext{and}
          f^*(\dot\gamma):=\overline{f(\dot\gamma^{-1})}. 
        \end{equation}
        These operations make $C_c(\Sigma,G,k)$ into a $*$-algebra.
        In addition, given $f\in C_c(\Sigma, G, k)$ we also write
        $\overline f$ for the function
        $\dot\gamma\mapsto \overline{f(\dot\gamma)}$.

        As described in \cite[Section~2]{BrownFullerPittsReznikoffGrC*AlTwGpC*Al}, we shall sometimes find it convenient to view elements of
        $C_c(\Sigma, G,k)$ as compactly supported scalar-valued
        functions on $\Sigma$.    Here is a brief outline of how this
        is done.  If
        $f\in C_c(\Sigma,G,k)$, then given $\dot\gamma\in G$, we may
        choose $\gamma\in P^{-1}(\dot\gamma)$ and a scalar
        $\tilde f(\gamma)$ so that
        $f(\dot\gamma)=[\tilde{f}(\gamma), \gamma]$.  For $z\in\bbT$,
        we may replace $\gamma$ with $z\cdot \gamma$, so we also find
        $f(\dot\gamma)=[\tilde{f}(z\cdot \gamma),z\cdot \gamma]$.
        Using~\eqref{Lkact}, we see that $\tilde f$ is $k$-equivariant
        in the sense that for every $\gamma\in \gamma$ and $z\in\bbT$,
\[\tilde f(z\cdot \gamma)=z^k \tilde f(\gamma).\]  On the other hand,
if $\tilde f$ is a compactly supported continuous $k$-equivariant
scalar-valued function on $\Sigma$, then defining
$f(\dot\gamma):=[\tilde f(\gamma), \gamma]_k$ gives an element of
$C_c(\Sigma,G,k)$.  The map $f\mapsto \tilde f$ is a linear bijection
between $C_c(\Sigma,G,k)$ and the space of compactly supported
continuous $k$-equivariant functions on $\Sigma$.

When viewed as
functions on $\Sigma$, the operations of
addition and scalar multiplication are pointwise, involution
becomes $f^*(\gamma)=\overline{f(\gamma^{-1})}$ and the 
convolution multiplication is
\begin{equation}\label{covmul}
  (f\star g)(\gamma)=\sum_{\substack{\dot\gamma_1 \in G \\
r(\dot\gamma_1)=r(\dot\gamma)}} f(\gamma_1)g(\gamma_1\inv \gamma),
\end{equation} where for each $\dot\gamma_1$ with
$r(\dot\gamma_1)=r(\dot\gamma)$, only one representative $\gamma_1$ of
$\dot\gamma_1$ is chosen.  Note that~\eqref{covmul}
gives a well-defined product.

For $x\in \unit{G}$ and $f\in C_c(\Sigma,G,k)$, let $\eps_{x,k}(f):=\tilde
f(x)$.  Then $\eps_{x,k}$ is a positive linear functional in the sense
that $\eps_{x,k}(f^*\star f)\geq 0$ for every $f\in C_c(\Sigma,G,k)$.
The GNS construction produces a $*$-representation $(\pi_{x,k},\H_{x,k})$ of
$C_c(\Sigma,G,k)$, and $C^*_r(\Sigma,G,k)$ is defined to be the
completion of $C_c(\Sigma,G,k)$ with respect to the norm,
\[\norm{f}:=\sup_{x\in \unit{G}} \norm{\pi_{x,k}(f)},\qquad f\in C_c(\Sigma,G,k).\]  We will generally
write the product in $C^*_r(\Sigma,G,k)$ using concatenation rather
than using the symbol $\star$, as in the definition of $C_c(\Sigma,G,k)$. 

We now show  the \cstaralg s $C^*_r(\Sigma,G,k)$ and
$C^*_r(\Sigma,G,-k)$ are anti-isomorphic.  To begin, 
let $\fc: \bbC\times \Sigma\rightarrow \bbC\times \Sigma$ be defined
by  \[\fc(\lambda, \gamma)=(\overline\lambda, \gamma).\]  Note that
$(\lambda_1, \gamma_1)\sim_k
(\lambda_2, \gamma_2)$ if and only if
$\fc(\lambda_1,\gamma_1)\sim_{-k} \fc(\lambda_2, \gamma_2)$.  In
particular, $\fc$ induces a homeomorphism, again called $\fc$,  from
$L_k$ to $L_{-k}$ satisfying $\fc^2=\text{id}|_{L_k}$.  
For $k\in\{-1,1\}$, $i\in \{1,2\}$, $\gamma_i\in \Sigma$,
$\lambda_i\in\bbC$  and $\mu\in \bbC$, calculations yield,
\begin{itemize}
\item $\fc (\mu [\lambda_1, \gamma_1]_{k})=\overline\mu \fc
  ([\lambda_1,\gamma_1]_k)$;
\item when $\dot\gamma_1=\dot\gamma_2$, $\fc
  ([\lambda_1,\gamma_1]_k+[\lambda_2, \gamma_1]_k)=
 \fc
  ([\lambda_1,\gamma_1]_k)+\fc([\lambda_2, \gamma_1]_k)$;
\item when $\gamma_1\gamma_2$ is defined,
$\fc([\lambda_1,
\gamma_1]_k[\lambda_2,\gamma_2]_k)=\fc([\lambda_1,\gamma_1]_k)
\fc([\lambda_2, \gamma_2]_k)$; and
\item $\fc\left(\overline{[\lambda_1, \gamma_1]}_k\right)=\overline{\fc([\lambda_1,\gamma_1]_k)}$.
\end{itemize}
In this sense, the bundles $L_k$ and $L_{-k}$ are conjugate.

Next define  $V: C_c(\Sigma, G,k)\rightarrow C_c(\Sigma, G, -k)$
by
\[(Vf)(\dot \gamma)= \fc(f(\dot\gamma)).\]   Clearly $V^2$ is the
identity mapping on $C_c(\Sigma, G, k)$.

Define the \textit{transpose map}, $\tau: C_c(\Sigma, G, k)\rightarrow
C_c(\Sigma, G, -k)$ by
\[\tau(f):=V(f^*).\]  Then $\tau$ is a linear, 
anti-isomorphism of $C_c(\Sigma, G, k)$ onto $C_c(\Sigma, G, -k)$.
For $f\in C_c(\Sigma,G,k)$ and $\dot\gamma\in G$, a calculation shows 
\[\tau(f)^*(\dot\gamma)=\fc(f(\dot\gamma))=\tau(f^*)(\dot\gamma).\]
Thus $\tau$ is also adjoint-preserving.

For $x\in \unit{G}$, let $\eta_{x,k}: C_c(\Sigma, G, k)\rightarrow \H_{x,k}$
be defined by $\eta_{x,k}(f)=f+\N_{x,k}$, where $\N_{x,k}$ is the left kernel of
$\eps_{x,k}$.
Calculations show that for $f, g\in C_c(\Sigma, G, k)$,
$\eps_{x,k}(g^*\star f)=\eps_{x,-k}((Vf)^*\star (Vg))$, that is,
\[\innerprod{\eta_{x,k}(f),\eta_{x,k}(g)}_{\H_{x,k}}
  =\innerprod{\eta_{-k}(Vg),\eta_{-k}(Vf)}_{\H_{x,-k}}.\] Thus $V$
induces a surjective conjugate-linear isometry
$W_x:\H_{x,k}\rightarrow\H_{x,-k}$ given by
$\eta_{x,k}(f)\mapsto \eta_{x,-k}(Vf)$.  Therefore, $\H_{x,-k}$ is the
conjugate Hilbert space of $\H_{x,k}$.

Now we observe that the transpose map is isometric.  For
$x\in \unit{G}$ and $f, g\in C_c(\Sigma, G, k)$, a calculation yields
\[W_x\pi_{x,k}(f^*)\eta_{x,k}(g)=\pi_{x,-k}(\tau(f)) W_x \eta_{x,k}(g).\]  Therefore, for any
$f\in C_c(\Sigma, G,k)$,
\begin{equation}\label{tflip}
  W_x\pi_{x,k}(f^*)W_x^{-1}=\pi_{x,-k}(\tau(f)).
\end{equation}
Write $\norm{\cdot}_k$ for the norm in $C_r^*(\Sigma, G, k)$.  For
$f\in C_c(\Sigma, G, k)$, \eqref{tflip} implies
\[\norm{f}_k=\norm{f^*}_k=\norm{\tau(f)}_{-k}.\]   

These considerations yield the first part of the following  observation relating the \cstaralg s of a
twist and its conjugate.
\begin{proposition}\label{kpm}  Let $(\Sigma, G, \iota, \fq)$ be a
 (Hausdorff) twist  with conjugate twist $(\overline\Sigma, G, \overline\iota,\fq)$
  and let $k\in \{-1,1\}$.  The following statements hold.
  \begin{enumerate}
    \item The transpose map
  $\tau: C_c(\Sigma, G,k)\rightarrow C_c(\Sigma, G, -k)$
  extends
  to a $*$-preserving anti-isomorphism of $C_r^*(\Sigma,G, k)$ onto
  $C^*_r(\Sigma, G, -k)$.
\item Let $L_k(\Sigma)$ and $L_k(\overline\Sigma)$ denote the Fell
  bundles over $G$ associated to $(\Sigma, G, \iota, \fq)$ and
  $(\overline\Sigma, G, \overline\iota, \fq)$ respectively.  Then
  $L_k(\Sigma)$ and $L_{-k}(\overline\Sigma)$ are the same; thus
  $C_r^*(\Sigma,G, k)=C_r^*(\overline\Sigma, G,-k)$.
\end{enumerate}

\end{proposition}
\begin{proof}  We have already outlined the proof of (a) above.
  
  (b) For $\lambda\in \bbC$, $z\in\bbT$, $\gamma\in \Sigma$ and $e=r(\gamma)$,
\[(\overline{z}\lambda,
  \iota(\overline{z},e)\gamma)=(\overline{z}\lambda,
  \overline{\iota}(z,e)\gamma),\dstext{so}
\{(z\lambda, \iota(z,e)\gamma): z\in \bbT\}=\{(\overline z,
  \overline\iota(z,e) \gamma): z\in \bbT\}.\] In other words, the
$L_k(\Sigma)$ equivalence class of $(\lambda, \gamma)$ coincides with
the $L_{-k}(\overline\Sigma)$ equivalence class of
$(\lambda, \gamma)$.  As the identity map
$\text{id}: L_k(\Sigma)\rightarrow L_{-k}(\overline\Sigma)$ preserves
the Fell bundle operations,
\[L_k(\Sigma)=L_{-k}(\overline\Sigma).\] 
Thus $C_c(\Sigma, G,k)=C_c(\overline\Sigma, G, -k)$, so that
$C^*_r(\Sigma, G, k)=C^*_r(\overline\Sigma, G, -k)$.
\end{proof}

\begin{remark}{Remark} \label{2C*algs} While anti-isomorphic, the
  \cstaralg s $C^*_r(\Sigma, G,1)$ and $C^*_r(\Sigma, G, -1)$ need not
  be isomorphic.
\end{remark}
\begin{remark}{Notation}\label{nok}
In the sequel, when we write
  $C^*_r(\Sigma, G)$, the reader is to assume that $k\in
  \{-1,1\}$ has
  been fixed, and that the validity of the result or discussion does
  not depend upon the choice of $k$.   However, when it is necessary to specify
  $k$, we will always write $C^*_r(\Sigma, G,k)$.
\end{remark}

As observed in the remarks
following~\cite[Proposition~4.1]{RenaultCaSuC*Al} (and with more
detail in
\cite[Propostion~2.21]{BrownFullerPittsReznikoffGrC*AlTwGpC*Al}),
elements of $C^*_r(\Sigma,G, k)$ may be regarded as $k$-equivariant
continuous functions on $\Sigma$, and the formulas defining the
product and involution on $C_c(\Sigma,G,k)$ remain valid for elements
of $C^*_r(\Sigma,G,k)$.  Also, as
in~\cite[Proposition~4.1]{RenaultCaSuC*Al} and \cite[Propostion~2.21]{BrownFullerPittsReznikoffGrC*AlTwGpC*Al}, for $\gamma\in\Sigma$ and
$f\in C^*_r(\Sigma,G,k)$,
\[|f(\gamma)|\leq \norm{f}.\] Thus, point evaluations are continuous
linear functionals on $C^*_r(\Sigma,G,k)$.   
\begin{definition}
We shall call the smallest topology on $C^*_r(\Sigma,G)$ such that for
every $\gamma\in \Sigma$, the point evaluation functional,
$C^*_r(\Sigma,G)\ni f\mapsto f(\gamma)$ is 
continuous, the \textit{$\Sigma$-pointwise topology} on $C^*_r(\Sigma,G)$.
Clearly this topology is Hausdorff.
\end{definition}

We will say that an element $f\in C^*_r(\Sigma,G)$ is
\textit{supported in the slice $S$} if $\supp(f)\subseteq S$.
Notice that 
$C_0(\unit{G})$ may be identified with \[\{f\in C^*_r(\Sigma,G):
\supp(f)\subseteq \unit{G}\}.\]  We will often tacitly make this identification.

In order to remain within the unital context, we now assume that the
unit space of $G$ is compact.  In this case $C^*_r(\Sigma,G)$ is
unital, and $C(\unit{G})\subseteq C^*_r(\Sigma,G)$, so that
$(C^*_r(\Sigma,G), C(\unit{G}))$ is an inclusion.  We next observe
it is a regular inclusion.
If $f\in C^*_r(\Sigma,G)$ is supported in a slice $U$, then a
computation (see~\cite[Proposition~4.8]{RenaultCaSuC*Al}) shows that
$f\in \N(C^*_r(\Sigma,G),C(\unit{G}))$, and, because the collection
of slices forms a basis for the topology of $G$
(\cite[Proposition~3.5]{ExelInSeCoC*Al}), it follows (as
in~\cite[Corollary~4.9]{RenaultCaSuC*Al}) that $(C^*_r(\Sigma,G),
C(\unit{G}))$ is a regular inclusion.

For $f\in C_c(\Sigma,G)$, define
\[E(f)(\dot\gamma):=\begin{cases} 0&\text{if
      $\dot\gamma\notin\unit{G}$;}\\
    f(\dot\gamma)& \text{if $\dot\gamma\in \unit{G}$.}
  \end{cases}
\]
As 
in~\cite[Proposition~4.3]{RenaultCaSuC*Al} or
\cite[Proposition~II.4.8]{RenaultGrApC*Al}, $E$ extends to a
faithful conditional expectation of $C^*_r(\Sigma,G)$ onto $C(\unit{G})$.  

\begin{proposition}\label{trivrad4twists} Let $(\Sigma,G,\iota,\fq)$
  be a Hausdorff twist and assume $\unit{G}$ is compact.  Then there
  is a faithful conditional expectation
  $E:C^*_r(\Sigma,G)\rightarrow C(\unit{G})$, and the inclusion
  $(C^*_r(\Sigma,G),C(\unit{G}))$ is regular.  Thus, when
  $C(\unit{G})$ is a MASA in $C^*_r(\Sigma,G)$,
  $(C^*_r(\Sigma,G),C(\unit{G}))$ is a Cartan inclusion.
\end{proposition}

\begin{proof}
We have already
observed that the inclusion is regular and there is a faithful
conditional expectation.  
When $C(\unit{G})$ is a MASA in $C^*_r(\Sigma,G)$,
$(C^*_r(\Sigma,G),C(\unit{G}))$ is a Cartan inclusion by definition.
\end{proof}

We conclude this section with a pair of technical results which we apply in
Section~\ref{GpReIn}.  The first of these will be used in the proof of
Theorem~\ref{inctoid}.  As we are regarding
$C_c(\Sigma, G)\subseteq C^*_r(\Sigma, G)$, we will write products via
concatenation.
\begin{lemma}\label{ceno}   Let $(\Sigma, G)$ be a twist.  For $i=1,2$ suppose $U_i$ are open
  bisections and $f_i\in C_c(\Sigma, G)$ satisfy
  $\overline{\supp}(f_i)\subseteq U_i$.  Then for $\dot\gamma\in G$,
  \[(E(f_1f_2^*)f_2)(\dot\gamma)=\begin{cases} 0 & \dot\gamma\not\in
      \supp(f_1)\cap \supp(f_2),\\  (f_1 f_2^*f_2)(\dot\gamma) &
      \dot\gamma\in \supp(f_1)\cap \supp(f_2).
    \end{cases}
  \]
\end{lemma}
\begin{proof}
A computation shows that when $h\in C_c(\Sigma,G)$ has support
in $\unit{G}$, then
\[(f_ih)(\dot\gamma)=f_i(\dot\gamma)h(s(\dot\gamma))\dstext{and}
  (hf_i)(\dot\gamma)=h(r(\dot\gamma))f_i(\dot\gamma).\] Thus, for
$\dot\gamma\in G$,
\begin{equation}\label{ceno1}
  (E(f_1f_2^*)f_2)(\dot\gamma)=E(f_1f_2^*)(r(\dot\gamma))\,\,
  f_2(\dot\gamma)\dstext{and}
  (f_1f_2^*f_2)(\dot\gamma)=f_1(\dot\gamma)\,\,(f_2^*f_2)(s(\dot\gamma)).
\end{equation}
Now
\begin{align*}
  E(f_1f_2^*)(r(\dot\gamma))&=(f_1f_2^*)(r(\dot\gamma))=
                             \sum_{\dot\gamma_1\dot\gamma_2=r(\dot\gamma)}
                              f_1(\dot\gamma_1)\overline{f_2(\dot\gamma_2^{-1})}\\
  &=\sum_{\dot\gamma_1\in r(\dot\gamma)G}
    f_1(\dot\gamma_1)\overline{f_2(\dot\gamma_1)}=f_1(\dot\gamma)\overline{f_2(\dot\gamma)};
\end{align*}
the last equality holding because $f_i$ are supported in bisections,
so that
$(r(\dot\gamma)G) \cap \supp(f_i)\subseteq \{\dot\gamma\}$.  Also, 
$(f_2^*f_2)(s(\dot\gamma))=\overline{f_2(\dot\gamma)}f_2(\dot\gamma)$
and thus,
\begin{align*}E(f_1f_2^*)(r(\dot\gamma))\,\,
  f_2(\dot\gamma)&=f_1(\dot\gamma)\overline{f_2(\dot\gamma)}f_2(\dot\gamma)\\
                 &=f_1(\dot\gamma)\,\, (f_2^*f_2)(s(\dot\gamma)).
\end{align*}
Combining this equality with~\eqref{ceno1} gives the result.
\end{proof}

The following lemma will be used when proving
Proposition~\ref{twistprop}.
\begin{lemma}\label{twistquot}  Suppose $(\Sigma, G)$ is a Hausdorff
  twist and $H\subseteq G$ is a closed subgroupoid satisfying the
  following factorization property:  if $\dot\gamma\in H$
factors as $\dot\gamma=\dot\gamma_1\dot\gamma_2$ where
$\dot\gamma_1, \dot\gamma_2 \in G$, then $\dot\gamma_1$ and
$\dot\gamma_2$ both belong to $H$.
Set $\Sigma_H:=\fq^{-1}(H)$.

Then $(\Sigma_H,H)$ is a subtwist of $(\Sigma, G)$ and the restriction map $P:
C_c(\Sigma,G)\rightarrow C_c(\Sigma_H,H)$ given by
$f\mapsto f|_H$  extends to a $*$-epimorphism of $C^*_r(\Sigma, G)$
onto $C^*_r(\Sigma_H, H)$.    
\end{lemma}
\begin{proof}
  We shall only sketch the proof.
  That $(\Sigma_H, H)$ is a subtwist follows from the definitions of
  twist and subtwist.

  The linearity of $P$ is clear and the factorization property for $H$ implies $P$
  is a $*$-homomorphism.   If $f\in C_c(\Sigma_H,H)$ is supported in a
  slice $U\subseteq H$, then we may find an open set $V\subseteq G$
  such that $U=V\cap H$, so extending $f$ by $0$ to $V$ produces an
  element $f_G\in C_c(\Sigma, G)$ such that $Pf_G=f$.  As $H$ is
  Hausdorff, the linear span of functions supported in slices of $H$ is all
  of $C_c(\Sigma_H, H)$; it follows that $P$ is onto.

  Let $Y=\unit{H}$, so $Y\subseteq \unit{G}$.  For $y\in Y$, we may
  consider the evaluation mappings $\eps_{y,H}$ (respectively $\eps_{y,G}$) given
  by $f\mapsto f(y)$, where $f$ is a continuous section of the line
  bundle for $H$ (respectively for $G$).  Let $(\pi_{y,H}, \H_{y,H})$ be the GNS
  representation of $C_c(\Sigma_H, H)$ arising from $\eps_{y,H}$ and
  let $\eta_{y,H}:C_c(\Sigma_H,H)\rightarrow \H_{y,H}$ be the map
  $f\mapsto f+N_{\eps_{y,H}}$; use similar notation for $\eps_{y,G}$.
  A computation (again using the factorization property) shows that
  for $f\in C_c(\Sigma,G)$, the map
  $\eta_{y,G}(f)\mapsto \eta_{y,H}(Pf)$ is a well-defined isometry
  which extends to
  a unitary operator $U_y:\H_{y,G}\rightarrow \H_{y,H}$.  Further, for
  $f\in C_c(\Sigma,G)$, 
  \begin{equation}\label{twistquot1}  U_y\pi_{y,G}(f)=\pi_{y,H}(Pf)
    U_y.
  \end{equation}
  Therefore,
  \[\norm{Pf}_{C^*_r(\Sigma_H,H)}=\sup_{y\in
      Y}\norm{\pi_{y,H}(Pf)}= \sup_{y\in Y}\norm{\pi_{y,G}(f)}\leq
    \sup_{x\in\unit{G}} \norm{\pi_x(f)} = \norm{f}_{C^*_r(\Sigma,G)}.\]
  The lemma follows.
\end{proof}

\section{Conditional Expectations and Hausdorff Weyl Groupoids}
\label{CEHWG}  The purpose of this section is to  
establish Theorem~\ref{CET2}, which shows that the groupoid of germs
$G$ 
for the Weyl semigroup associated to the regular MASA
inclusion  $(\C,\D)$ is Hausdorff if and
only if there exists a (necessarily unique) conditional expectation
$E:\C\rightarrow \D$.   The fact that $G$ is Hausdorff in the presence
of a conditional expectation was shown by Renault
in~\cite{RenaultCaSuC*Al}.  We include a sketch of
an alternate argument 
establishing this fact  in the proof of
Theorem~\ref{CET2}.  To our knowledge, the converse is new and is an
interesting application of the pseudo-expectation on a regular MASA inclusion.

We begin with two lemmas, the first of which will be used in the proof
of the second.
\begin{lemma}\label{v*E(v)}  Suppose $(\C,\D)$ is an inclusion and
  there exists a conditional expectation $E:\C\rightarrow\D$.  If
  $v\in \N(\C,\D)$, then $v^*E(v)\in\D^c$.
\end{lemma}
\begin{proof}
  Let $J=\overline{E(v)\D}$.  If
  $\tau\in\hat\D$ and $\tau|_J\neq 0$, we have $0< |\tau(E(v))|^2=
  \tau(E(v)^*E(v))\leq \tau(v^*v)$, so $\tau\in \dom\beta_v$.  For any
  $h\in J$ and $d\in\D$,
  \[v^*E(vh)d=v^*E(\theta_{v^*}(hd)v)=v^*\theta_{v^*}(hd)E(v)=hdv^*E(v)=dhv^*E(v)=dv^*E(vh),\]
  so $v^*E(vh)=v^*E(v)h\in\D$.  Taking $h$ to be from an approximate unit for $J$
  we obtain $v^*E(v)\in\D^c$.  
\end{proof}

The following lemma yields useful characterizations of the
equivalence relation used in the definition of Weyl groupoid for a
skeletal MASA inclusion.  However, we state the lemma for general
inclusions to make clear which hypothesis are needed for the
equivalences.  
\begin{lemma}\label{chareqrel}  Let $(\C,\D)$ be an inclusion and
  suppose $\M$ is a skeleton for $(\C,\D)$.   For
  $i=1, 2$, let 
  $v_i\in\M$ and suppose $\sigma\in \dom(\beta_{v_1})\cap
  \dom(\beta_{v_2})$.   Consider the following statements.
  \begin{enumerate}
\item There exist $h, k\in\D$ with $\sigma(h)\sigma(k)\neq 0$
        such that $v_1h=v_2k$.
  \item $\beta_{v_1}$ and $\beta_{v_2}$ have the same germ at
      $\sigma$.
            \end{enumerate}
Then (a)$\Rightarrow$(b).   If $\M$ is a MASA skeleton for $(\C,\D)$, then
(b)$\Rightarrow$(a).

      Furthermore, if $\M$ is a MASA skeleton for $(\C,\D)$ and there exists a
      conditional expectation $E:\C\rightarrow \D$, then (a) and (b)
      are  equivalent to:
      \begin{enumerate}\setcounter{enumi}{2}
      \item $\sigma(E(v_2^*v_1))\neq 0$.
      \end{enumerate}
      
    \end{lemma}
    \begin{proof}
      Suppose (a) holds.   Consider the open neighborhood of $\sigma$, 
      \[H:=\{\rho\in\hat\D: |  \rho(v_1^*v_1 v_2^*v_2 hk)| >
|\sigma(v_1^*v_1 v_2^*v_2 hk)|/2\}.\]  Choose $\rho\in H$.  By
hypothesis, $\sigma(v_1^*v_1 v_2^*v_2 hk)\neq 0$, so 
$\rho\in \dom\beta_{v_1h}\cap \dom\beta_{v_2k}$.   Thus for $d\in\D$,
\[\beta_{v_1}(\rho)(d)=\beta_{v_1h}(\rho)(d)=\beta_{v_2k}(\rho)(d)=\beta_{v_2}(\rho)(d),\]
so $\beta_{v_1}$ and $\beta_{v_2}$ have the same germ at $\sigma$.

Suppose $\M$ is a MASA skeleton for $(\C,\D)$ and (b) holds.  Since
$\M$ is a skeleton, $\D\subseteq \spn \M$, so letting $\M_1:=\spn
\M\cap \N(\C,\D)$, we see that $\M_1$ is a skeleton containing $\D$.   Let
$H\subseteq \dom(\beta_{v_1})\cap \dom(\beta_{v_2})$ be open in
$\hat\D$ with $\sigma\in H$ and $\beta_{v_1}|_H=\beta_{v_2}|_H$.
Recall that an \textit{intertwiner} for $(\C,\D)$ is an element
$w\in\C$ such that $w\D=\D w$ (no closures).  The proof of
\cite[Proposition~3.4]{DonsigPittsCoSyBoIs} shows that we may choose
$h_1, k_1\in \D$ such that $\sigma(h_1)\sigma(k_1)\neq 0$ and both
$v_1h_1$ and $v_2k_2$ are intertwiners.   Since $\beta_{v_1},
\beta_{v_2}, \beta_{v_1h_1}$ and
$\beta_{v_2k_1}$ have the same germ at $\sigma$,  we may assume without
loss of generality that  $v_1$ and $v_2$ are intertwiners.

Let $w=v_2^*v_1$.  Then $\beta_w|_H=\id_H$, so
$\sigma\in(\fix\beta_w)^\circ$.  Therefore there exists $d\in\D$ such
that $\supp\hat d\subseteq H$ and $\sigma(d)\neq 0$.
Lemmas~\ref{fixedideal} and~\ref{fixptchar} give
$wd=dw\in\D^c\cap \M_1=\D$.   Put $k:=wd$.  Since $v_1$ is an intertwiner,
there exists $a\in\D$ such that $(v_2v_2^*)v_1=v_1a$.  Set $h=ad$.  Then
\[v_2k=v_2v_2^*v_1d=v_1ad=v_1h.\]  Next,
$|\sigma(k)|^2=\sigma(d^*w^*wd)=|\sigma(d)|^2\sigma(w^*w)\neq 0$.
Since
\[|\sigma(h)|^2\sigma(v_1^*v_1)=\sigma(h^*v_1^*v_1h)=\sigma(k^*v_2^*v_2k)\neq
0,\] we obtain $\sigma(h)\neq 0$, so (a) holds.

For the remainder of the proof, suppose $\M$ is a MASA skeleton for
$(\C,\D)$ and there exists a
conditional expectation $E:\C\rightarrow \D$.  If (a) holds, then
\[\sigma(k^*)\sigma(h)
\sigma(E(v_2^*v_1))=\sigma(E(k^*v_2^*v_1h))=\sigma(E(h^*v_1^*v_1h))\neq
0.\] Thus (c) holds.

Finally, suppose (c) holds and again put $w=v_2^*v_1$.   By Lemma~\ref{v*E(v)},
$w^*E(w)\in \D$, so $\beta_{v_2}^{-1}\beta_{v_1}=\beta_w$ and $\id$ have the same germ at
$\sigma$.  Therefore, $\beta_{v_1}$ and $\beta_{v_2}$ have the same
germ at $\sigma$. 
\end{proof}

\newcommand{\tres}{r}

Let $(\C,\D)$ be a regular MASA inclusion, let 
$(I(\D),\iota)$ be an injective envelope for $\D$ and let $E:
\C\rightarrow I(\D)$ be the pseudo-expectation.   Recall that a state
on $\C$ is a
\textit{strongly compatible state} if it belongs to the set 
\begin{equation}\label{defscs}
  \fS_s(\C,\D):=\{\rho\circ E: \rho\in\widehat{I(\D)}\},
\end{equation}
and that $\fS_s(\C,\D)\subseteq \fS(\C,\D)$
(\cite[Proposition~4.6]{PittsStReInI} or Theorem~\ref{upse=>cov} below).
Let
$\tres:\fS_s(\C,\D)\rightarrow \hat\D$ be the restriction map,
$\tres(\tau)=\tau|_\D$.   Since $\iota=E|_\D$, $\tres$ is a continuous
surjection.

\begin{theorem} \label{CET2} For a regular MASA inclusion $(\C,\D)$
  the following statements are equivalent.
\begin{enumerate}
\item The restriction map $\tres$ is one-to-one.
\item There is a conditional expectation of
  $\C$ onto $\D$.
\item The Weyl groupoid $G$ associated to $(\C,\D)$  is Hausdorff.
\end{enumerate}
\end{theorem}
\begin{proof}
Suppose $\tres$ is one-to-one.  Then $\tres$ is a homeomorphism with inverse
$\tres^{-1}(\sigma)=\rho\circ E$ where $\rho\in\widehat{I(\D)}$ is any
choice so that $\sigma=\rho\circ E|_\D$, equivalently, $\rho\circ\iota
=\sigma$.  For
$x\in\C$, the function $\hat\D\ni\sigma\mapsto \tres^{-1}(\sigma)(x)$ is continuous
and hence determines an element $\Delta(x)\in \D$ whose Gelfand transform
is $\widehat{\Delta(x)}(\sigma):=
\tres^{-1}(\sigma)(x)$.  Then $\Delta$ is completely positive, unital, and
for every $d\in \D$ and $\sigma\in\hat\D$,
\[\sigma(\Delta(d))=\rho(E(d))=\rho(\iota(d))=\sigma(d),\] so $\Delta|_\D=\text{Id}_\D$.
Thus, $\Delta$ is a conditional expectation and (b) holds.

As noted earlier, the implication (b)$\Rightarrow$(c) was established
by Renault, see~\cite[Proposition~5.7]{RenaultCaSuC*Al}.  Here is a sketch
of an argument somewhat different from Renault's.  Let
$\X_1:=\{(v,\sigma_1): (\sigma_2,v,\sigma_1)\in\fX\}$.
For each $(v,\sigma)\in \X_1$, consider the function on $\C$ given by 
\begin{equation}\label{aGdef}
  |[v,\sigma]|(x)=\frac{|\sigma(E(v^*x))|}{\sigma(v^*v)^{1/2}}
  \quad\text{($x\in \C$)},
\end{equation}
and
  let $\fG:=\{|[v,\sigma]|: (v,\sigma)\in \X_1\}$.  Put the topology
  of pointwise convergence on $\fG$.  Then $\fG$ is a Hausdorff space.
   
It follows from Lemma~\ref{chareqrel} that the map $U$ given by
$[\sigma_2,v,\sigma_1]_\bbT\mapsto |[v,\sigma_1]|$ is a well-defined
map of the Weyl groupoid $G$ into $\fG$.  We shall show $U$
is continuous and one-to-one.  Indeed if
$[\rho_\lambda,v_\lambda,\sigma_\lambda]_\bbT$ is a net in $G$
converging to $[\rho,v,\sigma]_\bbT$, continuity of the source map
gives $\sigma_\lambda\rightarrow \sigma$.  Let $d\in\D$ with
$\sigma(d)\neq 0$.  Then $N(vd)$ is a basic neighborhood of
$[\rho,v,\sigma]$, so for large $\lambda$,
$[\rho_\lambda,v_\lambda, \sigma_\lambda]\in N(vd)$.  Thus for large
enough $\lambda$, $\beta_{v_\lambda}$ and $\beta_v$ have the same germ
at $\sigma_\lambda$.  Therefore there exist
$h_\lambda, k_\lambda\in\D$ such that
$\sigma_\lambda(k_\lambda)\sigma_\lambda(h_\lambda)\neq 0$ and
\[v_\lambda k_\lambda=vh_\lambda.\]   Then for $x\in\C$ computations yield,
\[
  | [v_\lambda,\sigma_\lambda]|=|[v_\lambda k_\lambda,\sigma_\lambda]|
  =|[vh_\lambda, \sigma_\lambda]|=|[v,\sigma_\lambda]|\rightarrow
  |[v,\sigma]|.\]  Thus $U$ is continuous.   

Suppose now that $g=[\sigma_2,v,\sigma_1]_\bbT$ and
$g'=[\sigma_2',v',\sigma_1']_\bbT$ are elements of $G$ with
$U(g)=U(g')$, that is, $|[v,\sigma_1]|=|[v',\sigma_1']|$.  Then for
any $h\in\D$, taking $x=vh$ in~\eqref{aGdef} above 
shows
\[\sigma_1(E(v^*vh))=\sigma_1(E(v^*v))\sigma_1(h)=0 \Leftrightarrow
\sigma_1'(E(v'^*vh))=\sigma_1'(E(v'^*v))\sigma_1'(h)=0.\]
It follows that for any $h\in \D$, $\sigma_1(h)=0$ if and only if
$\sigma_1'(h)=0$, so $\sigma_1=\sigma_1'$.    
Taking $h=I$ gives $\sigma_1(E(v'^*v))\neq 0$.  Thus  $\beta_v$ and $\beta_{v'}$ have the
same germ at $\sigma_1$ by Lemma~\ref{chareqrel}.  Therefore $g=g'$,
so $U$ is one-to-one.

Since $U$ is a one-to-one and continuous mapping of $G$ into
the Hausdorff space $\fG$, it follows $G$ is also Hausdorff.

We prove the contrapositive of (c)$\Rightarrow$(a).  Suppose $\tres$ is
not one-to-one.  We may then find $\rho_1, \rho_2\in \widehat{I(\D)}$
such that $\rho_1\circ E\neq \rho_2\circ E$ yet $\rho_1\circ
\iota=\rho_2\circ \iota$.  Put $\sigma:=\rho_i\circ \iota$.  

Since $\rho_1\circ E\neq \rho_2\circ E$, regularity yields the existence of $v\in
\N(\C,\D)$ such that $\rho_1(E(v))\neq \rho_2(E(v))$.  Let
$K_0\subseteq \D$ be the fixed point ideal for $v$.   

The Cauchy-Schwartz inequality gives
\[|\rho_i(E(v))|^2\leq\rho(\iota(v^*v))=\sigma(v^*v).\]  Since
$\rho_i(E(v))$ cannot both vanish, 
$\sigma(v^*v)\neq 0$.

We claim that $\sigma|_{K_0}=0$.  For $d\in K_0$, $vd=dv\in\D^c=\D$, so 
\[\rho_1(E(v))\sigma(d)=\rho_1(E(vd))=\sigma(vd)=\rho_2(E(vd))=\rho_2(E(v))\sigma(d).\]
Since $\rho_1(E(v))\neq \rho_2(E(v))$, $\sigma(d)=0$, as desired.

Next we show that $\sigma|_{K_0^\perp}=0$.  Choose $d\in K_0^\perp$
and let $R_0\in I(\D)$ be the support projection of $K_0$ in $I(\D)$
(see~\cite[Lemma~1.9]{PittsStReInI}).  Then $\iota(d)R_0=0$, and
by~\cite[Theorem~3.5]{PittsStReInI}, 
$E(v)=E(v)R_0$.  Therefore,
\[\rho_1(E(v))\sigma(d)= \rho_1(E(v))\rho_1(\iota(d))=\rho_1(
  E(v)R_0\,\iota(d))=0.\] Similarly, $\rho_2(E(v))\sigma(d)=0$,  so that
$\sigma(d)=0$.

By construction, $K_0$ is a regular ideal in $\D$.  Thus,
\[\supp(K_0)^\perp=\supp(K_0^\perp)\in \ropen(\hat\D).\]  Also, since $K_0\vee
K_0^\perp=\D$,
\[\supp(K_0)\cup \supp(K_0^\perp)\] is a dense open set in $\hat\D$.
Since $\sigma$ annihilates $K_0$ and $K_0^\perp$, we have
\[\sigma\in\overline{\supp(K_0)}\cap \overline{\supp(K_0^\perp)}.\]

Let $J:=\overline{v^*v\D}\cap K_0$.  By Lemma~\ref{fixptchar},
$J^\dperp =K_0$.  Therefore
$\overline{\supp(J)}=\overline{\supp(K_0)}$.
But $\supp(J)=(\fix\beta_v)^\circ$, so
$\sigma\in \overline{(\fix\beta_v)^\circ}$.
Note  $\sigma\in \fix(\beta_v)$ because $\fix\beta_v$ is relatively
closed in $\dom\beta_v$.

Consider the ideal $\L:=K_0^\perp\cap \overline{v^*v\D}$.  Fix
$\tau\in \supp(\L)$.  We claim that whenever $H\subseteq\supp(\L)$ is
an open neighborhood of $\tau$, then there exists $\tau_1\in H$ such
that $\beta_v(\tau_1)\neq \tau_1$.  Indeed, if otherwise, then there
exists an open neighborhood $H\subseteq\supp(\L)$ of $\tau$ such that
$\beta_v(\tau_1)=\tau_1$ for every $\tau_1\in H$.  But then
$\tau\in (\fix\beta_v)^\circ=\supp(J)\subseteq \supp(K_0)$.  But
$\supp(K_0)$ and $\supp(K_0^\perp)$ are disjoint, and
$\supp(\L)\subseteq \supp(K_0^\perp)$.  This contradiction establishes
the claim.

Since $\sigma\in\overline{\supp(K_0^\perp)}$ and
$\supp(\L)=\supp(K_0^\perp)\cap \dom\beta_v$, every neighborhood of
$\sigma$ contains an element of $\supp(\L)$.  Thus, the preceding
discussion shows that every open neighborhood 
of $\sigma$ has non-empty intersection with both $(\fix\beta_v)^\circ$
and the set $\{\tau\in\dom\beta_v: \beta_v(\tau)\neq\tau\}$.  Therefore
\[[\sigma,I,\sigma]_\bbT\neq [\sigma,v,\sigma]_\bbT.\]

Suppose now that $V_1$ and $V_2$ are open neighborhoods of
$[\sigma, I, \sigma]_\bbT$ and $[\sigma,v, \sigma]_\bbT$ respectively.  We
may choose $d_1, d_2\in \overline{v^*v\D}$ such that
$\sigma(d_i)=1$ so that $N_\bbT(d_1)\subseteq V_1$ and
$N_\bbT(vd_2)\subseteq V_2$; recall $N_\bbT(d_1)$ and $N_\bbT(vd_2)$
are basic open neighborhoods of $[\sigma, I, \sigma]_\bbT$ and
$[\sigma, v, \sigma]_\bbT$ respectively.  We shall show that
$N_\bbT(d_1)\cap N_\bbT(vd_2)\neq \emptyset$.  Let $d=d_1d_2$.  Since
$\sigma\in\overline{\supp(J)}$ and $\sigma(d)=1$, we may
find $\tau\in\supp(J)$ such that $|\tau(d)|>1/2$.  Then
$\tau((vd)^*(vd))\neq 0$, and $\tau\in \fix(\beta_{vd})^\circ$.  Hence 
$[\tau, vd, \tau]_\bbT=[\tau, d, \tau]_\bbT\in N_\bbT(d_1)\cap
N_\bbT(vd_2)\subseteq V_1\cap V_2$.
Therefore, $G$ is not Hausdorff.
  
\end{proof}

\section{Cartan Envelopes}\label{ccce}
It follows from~\cite[Theorem~5.7]{PittsStReInI} that a regular
inclusion $(\C,\D)$ regularly embeds into a Cartan pair $(\C_1,\D_1)$
precisely when the ideal
$\rad(\C,\D)=\{x\in \C: \rho(x^*x)=0 \,\, \forall \, \rho\in
\fS(\C,\D)\}$ vanishes.  In general, the construction given in the
proof of~\cite[Theorem~5.7]{PittsStReInI} produces a Cartan pair
$(\C_1,\D_1)$ having little connection with the original pair
$(\C,\D)$.   An example of this behavior is  the inclusion $(C[0,1],
\bbC I)$, where the Cartan pair  into which $(C[0,1],
\bbC I)$ embeds is $(C[0,1], C[0,1])$.
However in some cases, the image of $\C$ under the embedding generates
$\C_1$ as a $\D_1$-bimodule and $(\C_1,\D_1)$ is minimal in a sense
made precise below.  When this occurs, we call such a minimal pair
$(\C_1,\D_1)$ a Cartan envelope for $(\C,\D)$ (in analogy with the
$C^*$-envelope of an operator system).

The purpose of this section is
to establish a main result, Theorem~\ref{ccequiv}, which characterizes
the existence and uniqueness of the Cartan envelope for a regular
inclusion in terms of the ideal intersection property and also in
terms of the unique faithful pseudo-expectation property.  We begin
with definitions.

\begin{definition} \label{Cenvdef} Let $(\C,\D)$ be a regular inclusion.
    \begin{enumerate}
    \item An \textit{extension} of $(\C,\D)$ is a triple
      $(\C_1,\D_1,\alpha)$ consisting of the regular inclusion
      $(\C_1,\D_1)$ and a regular $*$-monomorphism
      $\alpha:(\C,\D)\rightarrow (\C_1,\D_1)$.  In addition, if $(\C_1,\D_1)$ is a
      Cartan pair, we say $(\C_1,\D_1,\alpha)$ is a \textit{Cartan
        extension} of $(\C,\D)$.
    \item A \textit{\cover}
      for $(\C,\D)$ is an extension
  $(\C_1,\D_1,\alpha)$ such that there exists a faithful conditional
  expectation $\bbE_1: \C_1\rightarrow \D_1$ and the image of
  $(\C,\D)$ under $\alpha$ generates $(\C_1,\D_1)$ in the sense that 
  \[  \C_1=C^*(\alpha(\C)\cup
        \bbE_1(\alpha(\C)))\dstext{and}
        \D_1=C^*(\bbE_1(\C)).\]
      \item The \cover\  $(\C_1,\D_1,\alpha)$ is a \textit{\Ccover} when
    $(\C_1,\D_1)$ is a Cartan pair (the additional restriction is that
    $\D_1$ is a MASA in $\C_1$). 
    \item An \textit{\envelope} for $(\C,\D)$ is a \cover\ 
      $(\C_1,\D_1,\alpha)$ for 
      $(\C,\D)$ such that $(\D_1,\alpha|_\D)$ is an essential
      extension of $\D$.   An \envelope\ $(\C_1,\D_1,\alpha)$ is a
      \textit{Cartan envelope} when $(\C_1,\D_1)$ is a Cartan pair.
    \end{enumerate}
    Two extensions   $(\C_i,\D_i,\alpha_i)$  ($i=1,2$) of
$(\C,\D)$ are 
 \textit{equivalent} if there
  is a regular $*$-isomorphism $\psi:\C_1\rightarrow \C_2$ such that
  $\psi\circ\alpha_1=\alpha_2$. 
\end{definition}

Not every regular inclusion $(\C,\D)$ has a Cartan extension.
    Indeed, if $\D^c$ is not abelian,~\cite[Theorem~5.4]{PittsStReInI}
    shows that $(\C,\D)$ cannot have a Cartan extension.
 However, as noted above, 
    \cite[Theorem~5.7]{PittsStReInI} characterizes when $(\C,\D)$ has a Cartan
    extension.

We can now state the main result of this section.
\begin{theorem}\label{ccequiv}  Let $(\C,\D)$ be a regular inclusion
  and let $\D^c$ be the relative commutant of $\D$ in $\C$.
  The following statements are equivalent:
  \begin{enumerate}
  \item $(\C,\D)$ has a Cartan envelope;
    \item $(\C,\D)$ has the faithful unique pseudo-expectation
      property; 
      \item $(\D^c,\D)$ and $(\C,\D^c)$ are essential inclusions and
        $\D^c$ is abelian.
      \end{enumerate} 
When $(\C,\D)$ satisfies any of  conditions (a)--(c), the following
statements hold.
\begin{description}
\setcounter{enumi}{3}
\item[\rm Uniqueness] If for $j=1,2$, $(\C_j,\D_j,\alpha_j)$ are Cartan envelopes for
    $(\C,\D)$, there exists a unique regular $*$-isomorphism
    $\psi: \C_1\rightarrow \C_2$ such that $\psi\circ\alpha_1=\alpha_2$.
   
\item[\rm Minimality] If $(\C_1,\D_1,\alpha)$ is a Cartan  
      \cover\ for $(\C,\D)$, there is an ideal $\fJ\subseteq \C_1$ such
      that $\fJ\cap \alpha(\C)=(0)$ and, letting
      $q:\C_1\rightarrow \C_1/\fJ$ denote the quotient map, 
      $(\C_1/\fJ, \D_1/(\fJ\cap \D_1), q\circ\alpha)$ is a Cartan
      envelope for $(\C,\D)$.
    \end{description}
  \end{theorem}

 \begin{remark}{Remark}  It is possible to construct a regular MASA
   inclusion whose pseudo-expectation is not faithful, so the
   condition in part (c) that $(\C,\D^c)$ is essential is needed.

We shall give a groupoid description of the Cartan envelope for a
regular inclusion with the faithful unique pseudo-expectation
property in 
Section~\ref{GpReIn} (see Theorem~\ref{inctoid}).

 \end{remark}

 The proof of Theorem~\ref{ccequiv} will be accomplished in
          several steps.  We begin with a lemma on essential
          inclusions for abelian \cstaralg s.  It is possible to give
          a proof of the lemma from the definitions, but we prefer to use
          properties of pseudo-expectations.

          \begin{lemma}\label{abeless}  For $i=1,2,3$, let $\D_i$ be abelian
  \cstaralg s with $\D_1\subseteq \D_2\subseteq \D_3$.  Then
  $(\D_3,\D_1)$ is an essential inclusion if and only if both
  $(\D_3,\D_2)$ and $(\D_2,\D_1)$ are essential inclusions.
\end{lemma}
\begin{proof}
  Suppose $(\D_3,\D_1)$ is an essential inclusion.  That $(\D_3,\D_2)$
  is an essential inclusion follows readily from the definition of
  essential inclusion.
  By~\cite[Corollary~3.22]{PittsZarikianUnPsExC*In}, $(\D_3,\D_1)$ has
  the faithful unique pseudo-expectation property,
  so~\cite[Proposition~2.6]{PittsZarikianUnPsExC*In}, shows $(\D_2,\D_1)$
 also  has the faithful unique pseudo-expectation property.  Then
  $(\D_2,\D_1)$ is an essential inclusion by
  \cite[Corollary~3.22]{PittsZarikianUnPsExC*In}.

  The converse is left to the reader.
\end{proof}

         The  following  gives  the equivalence  of parts (b) and (c) of
          Theorem~\ref{ccequiv}.  

\begin{proposition}\label{!fps->vc}  Suppose $(\C,\D)$ is a regular
  inclusion and $(I(\D),\iota)$ is an injective envelope for $\D$.
  The following statements hold.
  \begin{enumerate}
\item   $(\C,\D)$ has the faithful unique pseudo-expectation
      property if and only if   the
  the following conditions hold:
  \begin{enumerate}
\item[(i)]    the relative commutant of $\D$ in $\C$ is abelian; and 
\item[(ii)] both  $(\D^c,\D)$ and $(\C,\D^c)$ are essential inclusions.
\end{enumerate}
\item $(\C,\D)$ is a Cartan inclusion
  if and only if there is a faithful conditional expectation
  $E:\C\rightarrow \D$ and $\iota\circ E$ is the only
  pseudo-expectation for $(\C,\D)$.
\end{enumerate}
\end{proposition}
\begin{remark}{Remark}  An inclusion with a unique conditional
  expectation need not have the unique pseudo-expectation property; an
  example of this behavior is given by Zarikian in~\cite{ZarikianUnCoExAbC*In}.
\end{remark}
\begin{proof}  (a)  Suppose $(\C,\D)$ has a unique pseudo expectation
  $E:\C\rightarrow I(\D)$ which is faithful.
  By~\cite[Corollary~3.14]{PittsZarikianUnPsExC*In}, $\D^c$ is
  abelian, and~\cite[Proposition~2.6]{PittsZarikianUnPsExC*In} shows
  that $E|_{\D^c}$ is the unique pseudo-expectation for $(\D^c,\D)$.
  Then \cite[Corollary~3.22]{PittsZarikianUnPsExC*In} shows that
  $(\D^c,\D)$ is essential and $E|_{\D^c}$ is a $*$-monomorphism which
  is the unique pseudo-expectation for $(\D^c,\D)$.  As
  $(I(\D), \iota(\D))$ is an essential inclusion and
  $\iota(\D)\subseteq E(\D^c)\subseteq I(\D)$, Lemma~\ref{abeless}
  shows  the inclusion $(I(\D), E(\D^c))$ is
  also essential.  It follows from the ``Moreover'' portion
  of~\cite[Theorem~2.16]{HadwinPaulsenInPrAnTo} that $(I(\D), E|_{\D^c})$ is
  an injective envelope for $\D^c$.

By~\cite[Lemma~2.10]{PittsStReInI}, the identity mapping on $\C$ is a
regular $*$-monomorphism of $(\C,\D)$ into $(\C,\D^c)$, whence $(\C,\D^c)$
is a regular MASA inclusion.   
Note that $E$ is a  pseudo-expectation for $(\C,\D^c)$ (relative to
$(I(\D), E|_{\D^c})$).  By~\cite[Theorem~3.5 and
  Theorem~3.15]{PittsStReInI},  $\L(\C, \D^c)$ is the left kernel of
  $E$ and is the unique ideal of $\C$ maximal with respect to having
  trivial intersection with $\D^c$.   
Since $E$ is faithful, $\L(\C,\D^c)=(0)$, and it follows that
$(\C,\D^c)$ is an essential inclusion.

For the converse, suppose $(\C,\D)$ is a regular inclusion satisfying
conditions (i) and (ii).  By~\cite[Corollary~3.7]{PittsStReInI},
$(\C,\D)$ has a unique pseudo-expectation $E:\C\rightarrow I(\D)$.  As
$(\D^c,\D)$ is an essential
inclusion,~\cite[Corollary~3.22]{PittsZarikianUnPsExC*In} shows
$E|_{\D^c}$ is multiplicative and is the unique pseudo-expectation for
$(\D^c,\D)$.  It follows that $(I(\D), E|_{\D^c})$ is an injective
envelope for $\D^c$, whence  $E$ is also the pseudo-expectation for
$(\C,\D^c)$.  As $(\C,\D^c)$ is essential and
$\L(\C,\D^c)\cap \D^c=(0)$, we obtain $\L(\C,\D^c) =(0)$.   Thus,  $E$ is
faithful.

(b) If $(\C,\D)$ is a Cartan inclusion, then by definition, it is a regular MASA
inclusion with a faithful conditional expectation $E:\C\rightarrow
\D$.  By~\cite[Theorem~3.5]{PittsStReInI}, $\iota\circ E$ is
the unique pseudo-expectation for $(\C,\D)$.
 
Suppose now that $(\C,\D)$ has a faithful conditional expectation
$E:\C\rightarrow \D$ and $\iota\circ E$ is the only
pseudo-expectation.  As $\iota\circ E$ is faithful, part
(a) shows that $\D^c$ is abelian and $(\D^c,\D)$ is an essential
extension.  An application
of~\cite[Corollary~3.22]{PittsZarikianUnPsExC*In} shows that
there exists a unique and faithful pseudo-expectation
$u:\D^c\rightarrow I(\D)$ which is multiplicative.   Then $\iota\circ
E|_{\D^c}=u$, so $E|_{\D^c}$ is a homomorphism of $\D^c$ onto $\D$.
As $E$ is faithful, $E|_{\D^c}$ is an isomorphism, whence $\D^c=\D$.
So $(\C,\D)$ is a regular MASA inclusion with a faithful conditional
expectation, that is, $(\C,\D)$ is a Cartan inclusion.

\end{proof}

 Next we prove the implication (a)$\Rightarrow$(b) of
 Theorem~\ref{ccequiv}.  

\begin{proposition}\label{CE-->!Fpse}  Suppose $(\C,\D)$ is a regular
  inclusion and $(\C_1,\D_1,\alpha)$ is a Cartan envelope for
  $(\C,\D)$.  Then $(\C,\D)$ has the faithful unique pseudo-expectation property.
\end{proposition}
\begin{proof}
 By~\cite[Proposition~5.3(ii)]{PittsStReInI}, 
the relative commutant, $\D^c$, of $\D$ in $\C$ is abelian and
$\alpha(\D^c)\subseteq \D_1$.    Since $(\D_1,\alpha|_\D)$ is an essential
extension of $\D$, Lemma~\ref{abeless} implies
$(\alpha(\D^c),\alpha|_\D)$ is also an essential extension.

Let $(I(\D),\iota)$ be an injective envelope for $\D$ and use $\bbE_1$
to denote the (faithful) conditional expectation of $\C_1$ onto
$\D_1$. 
As $\alpha$ is one-to-one, an application
of~\cite[Corollary~3.7]{PittsStReInI} shows that $(\C,\D)$ has a
unique pseudo-expectation $E$.     Injectivity shows there is a
$*$-homomorphism $u: \D_1\rightarrow I(\D)$ such that $\iota= u\circ
\alpha|_\D$ (see~diagram~\ref{!CD}).  Since $(\D_1,\alpha|_\D)$ is an
essential extension of $\D$, $u$ is faithful.   As $u\circ
\bbE_1\circ \alpha:\C\rightarrow I(\D)$ satisfies
$\iota=u\circ\bbE_1\circ \alpha|_\D$, it is a pseudo-expectation.  Then
$E=u\circ\bbE_1\circ \alpha$ is a composition of faithful completely
positive maps.  Thus $E$ is faithful. 
\end{proof}

The converse of Proposition~\ref{CE-->!Fpse} will require more
effort.  To begin, we observe that the faithful unique
pseudo-expectation property implies the existence of Cartan
extensions.
\begin{lemma} \label{fusp-->VC}  Suppose $(\C,\D)$ is a regular
  inclusion with the faithful unique pseudo-expectation property.   Then $(\C,\D)$ has a Cartan extension.
\end{lemma}
\begin{proof}
  Proposition~\ref{!fps->vc} implies $(\C,\D^c)$ is a regular MASA
  inclusion and $(\D^c,\D)$ is an essential inclusion.  Let
  $(I(\D),\iota)$ be an injective envelope for $\D$ and let
  $E:\C\rightarrow I(\D)$ be the pseudo-expectation.  As $(\D^c,\D)$
  is essential, $E|_{\D^c}$ is a faithful $*$-monomorphism of $\D^c$
  into $I(\D)$ (see~\cite[Corollary~3.22]{PittsZarikianUnPsExC*In}).
  In particular, $(I(\D),E|_{\D^c})$ is an injective envelope for
  $\D^c$.  It follows that $E$ is also the unique pseudo-expectation for
  $(\C,\D^c)$ (relative to $(I(\D),E|_{\D^c})$).  Since $E$ is
  faithful, the ideal  $\L(\C,\D^c)$ is trivial.  As
  $\rad(\C,\D^c)\subseteq \L(\C,\D^c)$,
  \cite[Theorem~5.7]{PittsStReInI} shows  $(\C,\D^c)$
  regularly embeds into a \cstardiag.   
  By~\cite[Lemma~2.10]{PittsStReInI}, the identity map on $\C$ is
  regular when viewed as a map of $(\C,\D)$ into $(\C,\D^c)$.
As the composition of regular maps is again regular, we conclude that
$(\C,\D)$ regularly embeds into a \cstardiag.
But every \cstardiag\ is a
  Cartan inclusion, so we are done.
\end{proof}

\begin{lemma}\label{!D1Dpe}  Suppose $(\C,\D)$ is a regular inclusion
  with the unique pseudo-expectation property, let 
  $(I(\D),\iota)$ be an injective envelope for $\D$, and let
  $E:\C\rightarrow I(\D)$ be the pseudo-expectation.
  Suppose $(\C_1,\D_1,\alpha)$ is a \cover. Then there exists a
  unique $*$-homomorphism $u: \D_1\rightarrow I(\D)$ such that $u\circ
  \alpha|_\D=\iota$.
\end{lemma}
\begin{proof}  Let $\bbE_1:\C_1\rightarrow \D_1$ be a faithful
  conditional expectation such that $\D_1=C^*(\bbE_1(\alpha(\C))$ and $\C_1$ is generated by
  $\alpha(\C)$ and $\D_1$.  
Injectivity gives the existence of a $*$-homomorphism
$u:\D_1\rightarrow I(\D)$ such that $u\circ\alpha|_\D=\iota$.  
  We thus have the following diagram
(the vertical upward-pointing arrows are
the inclusion maps).
\begin{equation}\label{!CD}\xymatrix{
    \C\ar[rr]^\alpha\ar[dr]^E& &\C_1\ar@/^2pc/[dd]^{\bbE_1}\\
    &I(\D)& \\
    \D\ar[uu]^\subseteq\ar[rr]_{\alpha|_\D}\ar[ur]^\iota&&
    \D_1\ar[uu]_{\subseteq}\ar[ul]_u}
\end{equation}  
Observe that $u\circ\bbE_1\circ \alpha$ is a pseudo-expectation for
$(\C,\D)$ because $u\circ\alpha|_\D=\iota$.   The uniqueness
hypothesis on $E$ then yields, \begin{equation}\label{!CDE}
  E=u\circ \bbE_1\circ
  \alpha.
\end{equation}
This
equality also holds for any other 
 $*$-homomorphism $u':\D_1\rightarrow I(\D)$ satisfying
$u'\circ\alpha|_\D=\iota$.  Thus
$u|_{\bbE_1(\alpha(\C))}=u'|_{\bbE_1(\alpha(\C))}$.  Since 
$\D_1$ is generated by $\bbE_1(\alpha(\C))$, we obtain $u=u'$, as desired.
\end{proof}

In the presence of the unique pseudo-expectation property, every
\envelope\ is a Cartan envelope.

\begin{proposition}\label{CEenv=Cenv}  Suppose $(\C,\D)$ is a regular
  inclusion with the unique pseudo-expectation property and suppose  
  $(\C_1,\D_1,\alpha)$ is a \cover.  The following are equivalent.
  \begin{enumerate}
  \item $(\C_1,\D_1,\alpha)$ is an \envelope.
  \item $(\D_1,\alpha|_\D)$ is an essential extension of $\D$.
  \item $(\C,\D)$ has the faithful unique pseudo-expectation property.
  \item $(\C_1,\D_1,\alpha)$ is a Cartan envelope.
  \end{enumerate}
\end{proposition}
\begin{proof} Let $(I(\D),\iota)$ be an injective envelope for
  $(\C,\D)$, let $E:\C\rightarrow \D$ be the pseudo-expectation and
  let $u: \D_1\rightarrow I(\D)$ be the $*$-homomorphism obtained from
  Lemma~\ref{!D1Dpe}.  Also, let $\bbE_1:\C_1\rightarrow \D_1$ be a
  faithful conditional expectation.  Diagram~\eqref{!CD}
  and equation~\eqref{!CDE} show $E$ is faithful if and only if $u$ is
  faithful.

The equivalence of (a) and (b) is the definition of \envelope.

Assume (b) holds.   Then $u$ faithful, so $E$ is faithful, which is  (c).   

Next suppose $E$ is faithful.  Then $u$ is faithful, so
Lemma~\ref{abeless} implies that  $(I(\D), u)$ 
is an essential extension of $\D_1$.
Therefore, $(I(\D),u)$ is an injective envelope for $\D_1$.  Let
$\Delta: \C_1\rightarrow I(\D)$ be a pseudo-expectation for
$(\C_1,\D_1)$ relative to $(I(\D),u)$.  Then \[\Delta|_{\D_1}=u=(u\circ
\bbE_1)|_{\D_1}.\]   Note that $(\Delta\circ \alpha)|_\D= (u\circ
\alpha)|_\D=\iota$, so $\Delta\circ \alpha$ is a pseudo-expectation
for $(\C,\D)$.    Since $(\C,\D)$ has the unique pseudo-expectation
property,
\[\Delta\circ \alpha=E=u\circ\bbE_1\circ \alpha.\]
By Choi's Lemma (see~\cite[Corollary~3.19]{PaulsenCoBoHoOpAl}), both
$\Delta$ and $u\circ \bbE_1$ are $\D_1$-bimodule maps.  As
$(\C_1,\D_1,\alpha)$ is a \cover, we conclude 
\[\Delta=u\circ\bbE_1.\]   Lemma~\ref{!fps->vc}(b) shows
$(\C_1,\D_1)$ is a Cartan inclusion.  Thus $(\C_1,\D_1,\alpha)$ is a
Cartan envelope.

As every Cartan envelope is an \envelope, the proof is complete.\end{proof}

Next we construct a \cover\ from a Cartan extension. 

\begin{lemma}\label{littleCartan}  Suppose $(\C,\D)$ is a regular
  inclusion.  If $(\C,\D)$ has a Cartan extension, then it has a
  \cover.  More specifically, if  $(\C_1,\D_1,\alpha)$ is a Cartan extension for
  $(\C,\D)$ with conditional expectation $\bbE_1:\C_1\rightarrow
  \D_1$, put 
  \[\C_\alpha:=C^*(\alpha(\C)\cup \bbE_1(\alpha(\C))\dstext{and}
    \D_\alpha:=C^*(\bbE_1(\alpha(\C))).\]
  Then $(\C_\alpha,\D_\alpha,\alpha)$ is a \cover\ for $(\C,\D)$.
\end{lemma}
\begin{proof}
Thinking of $\D_1$ as a subalgebra of $I(\D_1)$, we may view 
 $\bbE_1$ as the
pseudo-expectation for $(\C_1,\D_1)$.  Then \cite[Proposition~3.14]{PittsStReInI}
implies that for any $w\in \N(\C_1,\D_1)$ and $z\in \C_1$,
\[w^*\bbE_1(z)w=\bbE_1(w^*zw).\]

  Fix $v\in
\N(\C,\D)$.   We claim that for $n\in\bbN$ and any $x_1,\dots, x_n\in \C$,
\begin{equation}\label{IC0}
  \alpha(v)\left(\prod_{j=1}^n \bbE_1(\alpha(x_j))\right) \alpha(v^*)\in \D_\alpha.
\end{equation}
When $n=1$,  as $\alpha(v)\in \N(\C_1,\D_1)$, we
find
\begin{equation*}\label{lC1}
  \alpha(v)\bbE_1(\alpha(x_1))\alpha(v)^*=\bbE_1(\alpha(vx_1v^*))\in\D_\alpha.
\end{equation*}
Now suppose \eqref{IC0} holds for some $n$, let
$\{x_j\}_{j=1}^{n+1}\subseteq \C$ and put
\[y=\prod_{j=1}^n \bbE_1(\alpha(x_j)).\]
For $d\in \D$, 
\[\alpha(v)\alpha(v^*dv)
  y\bbE_1(\alpha(x_{n+1}))\alpha(v)^*
  =\alpha(v)y\alpha(v)^*\,\, \alpha(d)\,\,\alpha(v)\bbE_1(\alpha(x_{n+1}))\alpha(v)^*\in
  \D_\alpha. \]   Noting that $\overline{v^*\D v}=\overline{v^*v\D}$,
continuity shows  that for any $h\in
\overline{v^*v\D}$,
  \[\alpha(v)\alpha(h)y\bbE_1(\alpha(x_{n+1}))\alpha(v)^*\in
    \D_\alpha.\] Replacing $h$ with  $(v^*v)^{1/n}$ and then taking
  the limit as $n\rightarrow \infty$ yields
  \begin{equation}\label{lC2}
    \alpha(v)y\bbE_1(\alpha(x_{n+1}))\alpha(v)^*\in
    \D_\alpha.
  \end{equation}  Induction completes the proof of~\eqref{IC0}. 
  Since $\D_\alpha$ is generated by $\bbE_1(\alpha(\C))$,~\eqref{IC0} yields $\alpha(v)\in
\N(\C_\alpha,\D_\alpha)$.  It follows that $\alpha(\N(\C,\D)\subseteq
\N(\C_\alpha, \D_\alpha)$.   

 Since $\alpha(\N(\C,\D))\cup\D_\alpha$
generates $\C_\alpha$, we conclude that $(\C_\alpha,\D_\alpha)$ is a
regular inclusion and $\alpha:(\C,\D)\rightarrow
(\C_\alpha,\D_\alpha)$ is a regular homomorphism.  

As $\bbE_1|_{\C_\alpha}$ is a faithful conditional
expectation, $(\C_\alpha,\D_\alpha, \alpha)$ is a \cover\ 
for $(\C,\D)$.
\end{proof}

\begin{remark}{Remark} \label{?MASA} It would simply our arguments
  if we could show that  $(\C_\alpha,\D_\alpha)$ is a
  MASA inclusion, for it would then follow that $(\C_\alpha,\D_\alpha,\alpha)$ is a
  \Ccover.   We have been unable to do this is because
  we do not know whether $I(\D)$ and
  $I(\D_\alpha)$ agree, so   we do not know that $u\circ \bbE_\alpha$ is the unique
  faithful pseudo-expectation for $(\C_\alpha,\D_\alpha)$.  Thus we cannot apply
  Proposition~\ref{!fps->vc}.     
\end{remark}

\begin{proposition}\label{CE-->CC}  Suppose $(\C,\D)$ is a regular
  inclusion with the faithful unique pseudo-expectation property.
  Then $(\C,\D)$ has a Cartan envelope.
\end{proposition}
\begin{proof} Lemma~\ref{fusp-->VC} shows the existence of a
  Cartan extension $(\C_2,\D_2,\alpha)$ for $(\C,\D)$.  Let
  $\bbE_2:\C_2\rightarrow \D_2$ be the conditional expectation.
  Put
\[\D_1:=C^*(\bbE_2(\alpha(\C))), \quad \C_1:=C^*(\alpha(\C),\,
  \D_1),\dstext{and} \bbE_1:=\bbE_2|_{\C_1}.\] Then
$\bbE_1:\C_1\rightarrow \D_1$ is a faithful conditional expectation
and Lemma~\ref{littleCartan} shows $(\C_1,\D_1,\alpha)$ is a \cover\
for $(\C,\D)$.  Furthermore, the regularity of $\alpha$ and the
definition of $\bbE_1$ show that for any $x\in \C_1$ and
$v\in\N(\C,\D)$,
  \begin{equation}\label{bbE1inv}
    \alpha(v)\bbE_1(x)\alpha(v)^*=\bbE_1(\alpha(v)x\alpha(v)^*).
  \end{equation}
  
  Let
  $u:\D_1\rightarrow I(\D)$ be the (unique) $*$-homomorphism with
  $\iota=u\circ\alpha|_\D$ obtained from Lemma~\ref{!D1Dpe}.  Set
  \[\fJ:=\{x\in \C_1: \bbE_1(x^*x)\in\ker u\}.\]   We shall show that the following
  statements hold.
\begin{enumerate}
  \item $\fJ$ is a
  closed, two-sided ideal of $\C_1$ such that $\alpha(\C)\cap
  \fJ=(0)$ and $\D_1\cap \fJ=\ker u$.
  \item If $\tilde\alpha:\C\rightarrow \C_1/\fJ$ is the map $x\mapsto
    \alpha(x)+\fJ$, then $(\C_1/\fJ, \D_1/\ker u, \tilde\alpha)$ is a
    Cartan envelope for $(\C,\D)$.
  \end{enumerate}

To this end, first note that $\fJ$ is a closed left ideal of $\C_1$:  indeed,
for
$x\in \fJ$ and $y\in \C_1$, $\bbE_1(x^*y^*yx)\leq
\norm{y}^2\bbE(x^*x)=0$.  To show that $\fJ$ is a right ideal, we will
require  the following invariance property of $\ker u$:
  for every $v\in \N(\C,\D)$,
  \begin{equation}\label{ninv} \alpha(v)^*(\ker u)\alpha(v)\subseteq
    \ker u.
  \end{equation}
For this, we first establish some properties of $\ker u$.

Consider the inclusion $(\D_1,\alpha(\D))$.  The uniqueness of $u$
allows us to
apply~\cite[Corollary~3.21]{PittsZarikianUnPsExC*In} to conclude that 
  $\ker u$ is the unique ideal of $\D_1$ maximal  with respect to having
  trivial intersection with $\alpha(\D)$.  

Let $S:=\{\sigma_1\in \hat\D_1: \sigma_1\text{ annhiliates } \ker u\}$.
Then $S$ is closed and 
\begin{equation}\label{keruc}
  \ker u=\{h\in \D_1: \hat h \text{ vanishes on } S\}.
\end{equation}
Using~\cite[Corollary~3.21]{PittsZarikianUnPsExC*In} again, $S$ is the
unique minimal closed subset of $\hat \D_1$ such that
$\dual{\alpha}(S)=\hat\D$.  Since
$\dual{\iota}=\dual{\alpha}\circ \dual{u}$, it follows that
$\dual{\alpha}(\dual{u}(\widehat{I(\D)}))=\hat\D$, so
$S\subseteq \dual{u}(\widehat{I(\D)})$.  On the other hand, if
$\sigma_1=\rho\circ u$ for some $\rho\in \widehat{I(\D)}$, then
$\sigma_1$ annihilates $\ker u$, whence $S\supseteq
\dual{u}(\widehat{I(\D)})$.  Thus,
\begin{equation}\label{CECC1} S= \dual{u}(\widehat{I(\D)}).
\end{equation}

Next, fix $v\in \N(\C,\D)$ and suppose $\sigma_1\in S$ satisfies
$\sigma_1(\alpha(v^*v))\neq 0$.   We will show
\begin{equation}\label{Sinv}
  \beta_{\alpha(v)}(\sigma_1)\in S.
\end{equation}
Recall
from~\cite[Proposition~1.11]{PittsStReInI} that the partial
automorphism $\theta_v$ of $\D$ (see~\cite[Lemma~2.1]{PittsStReInI})
extends uniquely to a partial automorphism $\tilde\theta_v$ of $I(\D)$
such that $\tilde\theta_v\circ\iota=\iota\circ\theta_v$.
By~\eqref{CECC1}, there exists $\rho\in \widehat{I(\D)}$ such that
\[\sigma_1=\rho\circ u.\]  Since
$\theta_v(vv^*)=v^*v$,
\[(\rho\circ\tilde\theta)(\iota(vv^*))=
  \rho(\iota(v^*v))=\rho(u(\alpha(v^*v)))=\sigma_1(\alpha(v^*v))\neq
  0.\] Thus, setting $\rho':=\rho\circ\tilde\theta_v$ we find
$\rho'\in \widehat{I(\D)}$.  Let $\sigma_1'=\rho'\circ u\in\hat\D_1$.
By construction, $\sigma_1'\in S$.  We shall show
that \begin{equation}\label{CECC2}
  \sigma_1'=\beta_{\alpha(v)}(\sigma_1).
\end{equation}
For $x\in\C$,
\begin{align*}
  \beta_{\alpha(v)}(\sigma_1)(\bbE_1(\alpha(x)))&=\frac{\sigma_1(\alpha(v)^*\bbE_1(\alpha(x))\alpha(v))}{\sigma_1(\alpha(v^*v))}
                                          =\frac{\sigma_1(\bbE_1(\alpha(v^*xv)))}{\sigma_1(\alpha(v^*v))}\\
                                        &=\frac{\rho(u(\bbE_1(\alpha(v^*xv))))}{\rho(u(\alpha(v^*v)))}
                                          =\frac{\rho(E(v^*xv))}{\rho(\iota(v^*v))},\\
  \intertext{and applying~\cite[Proposition~3.14]{PittsStReInI} to the
  numerator,}
  &=\frac{\rho(\tilde\theta_v(E(vv^*x))}{\rho'(\iota(vv^*))}=
    \frac{\rho'(\iota(vv^*)E(x))}{\rho'(\iota(vv^*))}=\rho'(E(x))\\
  &=\rho'(u(\bbE_1(\alpha(x))))=\sigma_1'(\bbE_1(\alpha(x))).
\end{align*}
As $\beta_{\alpha(v)}(\sigma_1)$ and $\sigma_1'$ belong to $\hat\D_1$
and $\D_1$ is generated by $\bbE_1(\alpha(\C))$,~\eqref{CECC2} holds.
This establishes~\eqref{Sinv}.

We are now prepared to establish~\eqref{ninv}.
Choose  $h\in \ker u$.
When  $\sigma_1\in S$ satisfies
$\sigma_1(\alpha(v^*v))\neq 0$,
\[\sigma_1(\alpha(v)^*h\alpha(v))=\beta_{\alpha(v)}(\sigma_1)(h)\,
\sigma_1(\alpha(v^*v))=\sigma'(h)\, \sigma_1(\alpha(v^*v))=0.\] 
On the other hand, when $\sigma_1(\alpha(v^*v))=0$, $\sigma_1(\alpha(v)^*h\alpha(v))=0$ because 
\[|\sigma_1(\alpha(v)^*h
  \alpha(v))|^2=\sigma_1(\alpha(v)^*h^*\alpha(vv^*)h\alpha(v))\leq
  \norm{\alpha(v)^*h}^2 \sigma_1(\alpha(v^*v))=0.\]   We conclude that
whenever $\sigma_1\in S$,  $\sigma_1(\alpha(v)^*h\alpha(v))=0$.
By~\eqref{keruc}, $\alpha(v)^*h\alpha(v)\in \ker u$, so~\eqref{ninv} holds.

Next we show $\fJ$ is a right ideal.  If $x\in \fJ$, $v\in \N(\C,\D)$,
and $h\in\D_1$, then $x\alpha(v)h\in \fJ$ because
\[\bbE_1(h^*\alpha(v)^*x^*x\alpha(v)h)=h^*\alpha(v^*)\bbE_1(x^*x)\alpha(v)h\in
  \ker u.\] As $\alpha(\N(\C,\D))\subseteq \N(\C_1,\D_1)$,
$\{\alpha(v)h: v\in \N(\C,\D), h\in \D_1\}$ is a $*$-semigroup whose
span is dense in $\C_1$.  It follows that $\fJ$ is a right ideal.  So
$\fJ$ is a closed, two-sided ideal in $\C_1$.

If $y\in \alpha(\C)\cap \fJ$, then there is some $x\in \C$ so that
$y=\alpha(x)$.  Then
$0=u(\bbE_1(\alpha(x^*x)))=E(x^*x)$, so $x=0$ because the
pseudo-expectation for $(\C,\D)$ is faithful.   Thus $\alpha(\C)\cap
\fJ=(0)$.   The fact that $\fJ\cap \D_1=\ker u$ is clear.  This
completes the proof of assertion (a).

We turn now to assertion (b).
Let \[\tilde\C_1:=\C_1/\fJ, \quad \tilde\D_1:=\D_1/\ker u, \dstext{and}  
\tilde u:\tilde\D_1\rightarrow I(\D)\] be the map
$\tilde\D_1\ni h+\ker u\mapsto u(h)$.  Then if $\fA$ is the
\cstar-subalgebra of $I(\D)$ generated by $E(\C)$, $\tilde u$ is a
$*$-isomorphism of $\tilde \D_1$ onto $\fA$.  Since
$\iota=\tilde u\circ \tilde \alpha|_\D$,
$(\tilde\D_1,\tilde\alpha|_\D)$ is an essential extension of $\D$ and
$(I(\D), \tilde u)$ is an injective envelope for $\tilde \D_1$.

Clearly, $\tilde\alpha$ is a
regular $*$-monomorphism of $(\C,\D)$ into $(\tilde\C_1, \tilde\D_1)$.
If $x\in \fJ$, then the operator inequality,
$\bbE_1(x)^*\bbE_1(x)\leq \bbE_1(x^*x)$, shows that
$\bbE_1(\fJ)\subseteq \fJ$.  Thus the map
$\tilde\bbE_1: \tilde\C_1\rightarrow \tilde \D_1$ given by
$x+\fJ\mapsto \bbE_1(x)+\ker u$ is well-defined and is a conditional
expectation.  If $\tilde\bbE_1(x^*x+\fJ)=0$, then
$\bbE_1(x^*x)\in \ker u$, so $x\in \fJ$.  It follows that
$\tilde\bbE_1$ is faithful.  Note also that $\tilde u\circ
\tilde\bbE_1$ is then a faithful pseudo-expectation for $(\tilde\C_1,\tilde\D_1)$.

Suppose $\Delta: \tilde \C_1\rightarrow I(\D)$ satisfies
$\Delta|_{\tilde\D_1}=\tilde u$.   Then
$\Delta\circ\tilde\alpha=\iota$, so $\Delta\circ\tilde\alpha$ is a
pseudo-expectation for $(\C,\D)$.  As 
 $(\C,\D)$ has the unique pseudo-expectation property, 
\[\Delta\circ\tilde\alpha=E=\tilde u\circ \tilde\bbE_1\circ
  \tilde\alpha.\] 
Since $\Delta|_{\tilde\D_1}=(\tilde u\circ\tilde
\bbE_1)|_{\tilde\D_1}$, and $\tilde\C_1$ is generated by
$\tilde\D_1\cup\tilde\alpha(\N(\C,\D))$, $\Delta=\tilde u\circ
\tilde\bbE_1$.  Thus
$\tilde u\circ \tilde \bbE_1$ is the unique pseudo-expectation for
$(\tilde\C_1,\tilde\D_1)$.   By Proposition~\ref{!fps->vc}(b),
$(\tilde\C_1,\tilde\D_1)$ is a Cartan inclusion.

Finally, as $\tilde\C_1$ is regular, it  is generated by $\tilde\alpha(\C)\cup
\tilde D_1$,  and by construction, $\tilde\D_1$ is generated by
$\tilde\bbE_1(\tilde\alpha(\C))$.  Thus 
$(\tilde\C_1,\tilde\D_1,\tilde\alpha)$ is a \Ccover\ for
$(\C,\D)$.  As $\tilde u$ is a faithful $*$-homomorphism of
$\tilde\D_1$ into $I(\D)$, it follows that $(\tilde\D_1,\tilde\alpha)$
is an essential extension for $\D$, so the proof of assertion (b), and
hence the proposition, is complete.

\end{proof}

Next we show uniqueness (up to equivalence) of  Cartan envelopes. 
\begin{proposition}\label{!CE}  
  Let $(\C,\D)$ be a regular inclusion, and assume that for $i=1,2$,
  $(\C_i,\D_i,\alpha_i)$ are Cartan envelopes for $(\C,\D)$.  Then
  there is a unique regular $*$-isomorphism $\psi: \C_1\rightarrow \C_2$ such
  that $\psi\circ\alpha_1=\alpha_2$.
  Furthermore, if $\bbE_i:\C_i\rightarrow \D_i$ are the conditional
  expectations, then for every $x\in\C$ and $v\in \N(\C,\D)$,
  \begin{equation}\label{g2g}
    \psi(\alpha_1(v) \bbE_1(\alpha_1(x)))=\alpha_2(v)
    \bbE_2(\alpha_2(x)).
  \end{equation}
\end{proposition}
\begin{proof}
  Let $(I(\D),\iota)$ be an
  injective envelope for $\D$.  Lemma~\ref{CE-->!Fpse} gives a 
  unique (and faithful)
  pseudo-expectation $E:\C\rightarrow I(\D)$ for $(\C,\D)$.
For $i=1,2$, let $\bbE_i: \C_i\rightarrow \D_i$ be the (necessarily
  unique) faithful conditional expectation. 

  By Lemma~\ref{!D1Dpe}, there exist unique $*$-homomorphisms
  $u_i: \D_i\rightarrow I(\D)$ such that
  $u_i\circ\alpha_i|_{\D}=\iota$.  By the definition of Cartan
  envelope, $(\D_i,\alpha_i|_\D)$ are essential extensions of $\D$, so
  $u_i$ are actually $*$-monomorphisms.

Examining a variant of the diagram~\eqref{!CD} and using the fact that
$E$ is the unique pseudo-expectation, observe that for every $x\in \C$,
\[u_i(\bbE_i(\alpha_i(x)))=E(x).\] Since $\D_i$ is generated by
$\{\bbE_i(\alpha_i(x)): x\in \C\}$, we conclude that the range of
$u_i$ is the \cstar-subalgebra $\fA\subseteq I(\D)$ generated by
$E(\C)$.  Thus $\psi:=u_2^{-1}\circ u_1$ is a $*$-isomorphism of
$\D_1$ onto $\D_2$ such that for every $x\in\C$,
\[\psi(\bbE_1(\alpha_1(x)))=\bbE_2(\alpha_2(x)).\]

Let
\[\M_i:=\{\alpha_i(v) h: v\in \N(\C,\D), \, h\in
  \overline{\alpha_i(v^*v)\D_i}
  \}. \] It follows from~\cite[Lemma~2.1]{PittsStReInI} that $\M_i$ is
a $*$-semigroup and, as $\C_i$ is generated by
$\alpha_i(\C)\cup \bbE_i(\alpha_i(\C))$, $\M_i$ is a MASA skeleton for
$(\C_i,\D_i)$ (see~\cite[Definitions~1.7 and~3.1]{PittsStReInI}).  Let
  $n\in\bbN$ and for $1\leq k\leq n$, suppose $v_k\in \N(\C,\D)$ and
  $h_k\in \overline{\alpha(v_k^*v_k)\D_1}$.  Then
\begin{align*} 0=\sum_{k=1}^n \alpha_1(v_k)h_k&
   \Leftrightarrow \bbE_1\left(\sum_{k, \ell=1}^n
                    h_k^*\alpha_1(v_k^*v_\ell) h_\ell\right)=0\\
                                          &       \Leftrightarrow \psi\left(\sum_{k, \ell=1}^n
                    h_k^*\bbE_1(\alpha_1(v_k^*v_\ell)) h_\ell\right)=0\\
                  &\Leftrightarrow \sum_{k, \ell=1}^n
                    \psi(h_k^*)\bbE_2(\alpha_2(v_k^*v_\ell)) \psi(h_\ell)=0\\
                  &\Leftrightarrow \bbE_2\left(\sum_{k, \ell=1}^n
                    \psi(h_k^*)\alpha_2(v_k^*v_\ell)
                    \psi(h_\ell)\right)=0\\
  &\Leftrightarrow \sum_{k=1}^n \alpha_2(v_k)\psi(h_k) =0.
\end{align*} Thus, we may extend $\psi$ uniquely to a linear
mapping, again called $\psi$, of $\spn \M_1$ onto $\spn \M_2$
determined by
\[\sum \alpha_1(v_k)h_k\mapsto \sum \alpha_2(v_k)\psi(h_k).\]  Since
$\M_i$ are $*$-semigroups, it follows that $\spn\M_i$ are $*$-algebras
and $\psi$ is a $*$-isomorphism of $\spn\M_1$ onto $\spn \M_2$.

Now define two \cstar-norms on $\spn\M_1$:
\[\nu_1(x)=\norm{x}_{\C_1}\dstext{and}\nu_2(x)=\norm{\psi(x)}_{\C_2}.\]
Since $(\C_1,\D_1)$ is a Cartan pair, $\L(\C_1,\D_1)=(0)$, 
so~\cite[Theorem~7.4]{PittsStReInI} implies that $\nu_1$ is the
minimal \cstar-norm on $\spn\M_1$.  Thus, for every $x\in \spn\M_1$,
$\norm{x}_{\C_1}\leq
\norm{\psi(x)}_{\C_2}$.  A symmetric argument yields
$\norm{\psi(x)}_{\C_2}\leq \norm{x}_{\C_1}$, so  $\psi$ is
an isometric $*$-isomorphism of  $\spn\M_1$ onto $\spn\M_2$.
Therefore, 
$\psi$ extends to a $*$-isomorphism  (once again
called $\psi$) of $\C_1$ onto
$\C_2$.   As $\psi(\D_1)=\D_2$, $\psi$ is regular, and by
construction, $\psi\circ\alpha_1=\alpha_2$.  Thus
$(\C_1,\D_1,\alpha_1)$ is equivalent to $(\C_2,\D_2,\alpha_2)$ via a
$*$-isomorphism satisfying~\eqref{g2g}.

Turning to the uniqueness statement, suppose
$\psi':(\C_1,\D_1)\rightarrow (\C_2,\D_2)$ is a regular
$*$-isomorphism such that $\psi'\circ\alpha_1=\alpha_2$.  The
regularity of $\psi'$ yields $\psi'(\D_1)\subseteq \D_2$.  We now show
equality.  If $y\in \C_2$ commutes with $\psi'(\D_1)$, then
$\psi'^{-1}(y)$ commutes with $\D_1$, whence $\psi'^{-1}(y)\in \D_1$.
Thus, $y\in \psi'(\D_1)$, so $\psi'(\D_1)$ is a MASA in $\C_2$.  This
gives $\psi'(\D_1)=\D_2$.

Note that $\psi'^{-1}\circ\bbE_2\circ \psi'$ is a faithful
conditional expectation of $\C_1$ onto $\D_1$.  By uniqueness of
$\bbE_1$, we obtain $\bbE_2\circ\psi'=\psi'\circ\bbE_1$.  As this
relation also holds for $\psi$, 
\[\psi\circ \bbE_1\circ \alpha_1=\bbE_2\circ\psi \circ \alpha_1 = \bbE_2\circ \alpha_2=\bbE_2\circ \psi'\circ
  \alpha_1=\psi'\circ\bbE_1\circ \alpha_1.\]  Therefore, 
$\psi|_{\D_1}=\psi'|_{\D_1}$ because $\D_1$ is generated
by $\bbE_1(\alpha_1(\C))$.   Since $\C_1$ is the closed
$\D_1$-bimodule generated by $\alpha_1(\C)$, we obtain $\psi'=\psi$,
completing the proof.
\end{proof}

We have now established most of Theorem~\ref{ccequiv}, and we now
finish its proof.
\begin{proof}[Proof of Theorem~\ref{ccequiv}]
  Proposition~\ref{!fps->vc} gives (b)$\Leftrightarrow$(c) and
  (a)$\Leftrightarrow$(b) is Proposition~\ref{CE-->!Fpse} combined
  with Proposition~\ref{CE-->CC}.  Uniqueness of the Cartan envelope
  is Proposition~\ref{!CE}.

Finally, suppose $(\C,\D)$ has the faithful unique pseudo-expectation
property and $(\C_1,\D_1,\alpha)$ is a \Ccover\ for
$(\C,\D)$.  The proof of Proposition~\ref{CE-->CC} establishes the existence
of an ideal $\fJ$ in $\C_1$ with the requisite properties.
          \end{proof}

Suppose $(\C,\D)$ has the unique faithful
unique pseudo-expectation property.    Then $(\C,\D^c)$ is a virtual
Cartan inclusion.  Our next goal is Proposition~\ref{sub} below, which shows that
$(\C,\D)$ and $(\C,\D^c)$ have the same Cartan envelope.   

  \begin{lemma}\label{esub}  Suppose $(\C,\D)$ is an inclusion with
    the faithful unique pseudo-expectation property, let $\D^c$ be the
    relative commutant of $\D$ in $\C$, and let $\D_0\subseteq \D$ be a
    \cstar-subalgebra such that $(\D,\D_0)$ is an essential inclusion.
    If $\C_0$ is a \cstar-subalgebra of $\C$ containing $\D_0$, then
    the following statements hold.
    \begin{enumerate}
   \item   The inclusions $(\C,\D_0)$ and
    $(\C_0,\D_0)$ both have the
    faithful unique pseudo-expectation property.
    \item The inclusion mapping is a regular
      homomorphism of $(\C_0,\D_0)$ into $(\C,\D^c)$; in other words,
      \[\N(\C_0,\D_0)\subseteq \N(\C,\D^c).\]
    \end{enumerate}
  \end{lemma}
  \begin{proof}
    We begin with some preliminary observations.  Let
    $(I(\D_0),\iota_0)$ be an injective envelope for $\D_0$.
    By~\cite[Corollary~3.22]{PittsZarikianUnPsExC*In}, there exists a
    unique pseudo-expectation $\iota:\D\rightarrow I(\D_0)$ for
    $(\D,\D_0)$; furthermore, $\iota$ is multiplicative and faithful.
    We claim that $(I(\D_0),\iota)$ is an injective envelope for $\D$.
    As $\iota|_{\D_0}=\iota_0$, \[I(\D_0)\supseteq\iota(\D)\supseteq \iota_0(\D_0).\]
Lemma~\ref{abeless} implies $(I(\D_0),\iota)$ is an essential extension of
    $\D$.  Thus as $I(\D_0)$ is injective, $(I(\D_0),\iota)$ is an
    injective envelope for $\D$.  Let $E:\C\rightarrow I(\D_0)$ be the
    pseudo-expectation for $(\C,\D)$ with respect to
    $(I(\D_0),\iota)$.

a) Let $\Delta: \C\rightarrow
    I(\D_0)$ be a pseudo-expectation for $(\C,\D_0)$ relative to
    $(I(\D_0),\iota_0)$.    Then $\Delta|_{\D}$ is a
    pseudo-expectation for $(\D,\D_0)$, so $\Delta|_\D=\iota$.
    Therefore, $\Delta$ is a pseudo-expectation for $(\C,\D)$ relative
    to $(I(\D_0),\iota)$.  Since $(\C,\D)$ has the unique
    pseudo-expectation property, $\Delta=E$.  As $E$ is faithful, $(\C,\D_0)$ has the
    faithful unique pseudo-expectation property.   That $(\C_0,\D_0)$
    also has the faithful unique pseudo-expectation property follows
    from~\cite[Proposition~2.6]{PittsZarikianUnPsExC*In}.

  b) Let $\D_0^c$ be the relative commutant of $\D_0$ in $\C$.
  By~\cite[Corollary~3.14]{PittsZarikianUnPsExC*In}, $\D^c$ and
 $\D_0^c$  are abelian.
  Therefore, $(\C,\D^c)$ is a MASA inclusion.  Since $\D^c\subseteq
  \D_0^c$, we obtain \begin{equation}\label{esub0}
    \D^c= \D_0^c.
  \end{equation}
  
  Now suppose that $v\in \N(\C_0,\D_0)$ and let $h\in\D^c$.   Let $\theta_{v^*}$ be the $*$-isomorphism
  of $\overline{v^*v\D_0}$ onto $\overline{vv^*\D_0}$ determined by
  $v^*vd\mapsto vdv^*$.  Then for every $k\in \overline{v^*v\D_0}$,
  $vk=\theta_{v^*}(k)v$.  Let $d\in\D_0$.  For every $u\in
  \overline{v^*v\D_0}$,
  \begin{equation}\label{esub1}
    v^*hv ud= v^*h\theta_{v^*}(du)v=v^*\theta_{v^*}(du)hv=du v^*hv.
  \end{equation}
  Taking $(u_\lambda)$ to be an approximate unit for
  $\overline{v^*v\D_0}$, we have $v=\lim vu_\lambda$ and
  $v^*=\lim u_\lambda v^*$.  Then~\eqref{esub1} gives
  $v^*hv \in \D_0^c$, so 
  $v^*hv\in\D^c$ by~\eqref{esub0}.    A similar argument shows that for every
  $h\in \D^c$, $vhv^*\in \D^c$.  Thus, $v\in \N(\C,\D^c)$.

\end{proof}

We now give the relationships between 
Cartan envelopes and Cartan extensions for $(\C,\D)$ and $(\C,\D^c)$.

\begin{proposition}\label{sub}  Suppose $(\C,\D)$ is a regular
  inclusion with the unique faithful pseudo-expectation property.  The
  following statements hold.
  \begin{enumerate}
    \item 
  $(\C_1,\D_1,\alpha)$ is a Cartan extension for $(\C,\D)$ if and only
  if $(\C_1,\D_1,\alpha)$ is a Cartan extension for $(\C,\D^c)$.
  \item 
 $(\C_1,\D_1,\alpha)$ is a  Cartan envelope for $(\C,\D)$ if and
 only if $(\C_1,\D_1,\alpha)$ is a Cartan envelope for
 $(\C,\D^c)$.
\end{enumerate}
\end{proposition}
\begin{proof}

If $(\C_1,\D_1,\alpha)$ is a Cartan extension for
$(\C,\D^c)$, it clearly is a Cartan extension for $(\C,\D)$.  On the other hand, if
$(\C_1,\D_1,\alpha)$ is a Cartan extension for $(\C,\D)$,~\cite[Proposition~5.3(b)]{PittsStReInI} shows that it is
also a Cartan extension for $(\C,\D^c)$.

Turning to part (b), suppose $(\C_1,\D_1,\alpha)$ is a Cartan envelope
for $(\C,\D)$.  Let $\bbE_1$ be the conditional expectation for
$(\C_1,\D_1)$. We have already observed that
$\alpha(\D^c)\subseteq \D_1$
(\cite[Proposition~5.3(b)]{PittsStReInI}).

  By definition, $\D_1$ is
  generated by $\bbE_1(\alpha(\C)))$ and $\C_1$ is generated by
  $\alpha(\C)\cup \D_1$.  As $(\D_1,\alpha(\D))$ is an essential
  extension, so is $(\D_1,\alpha(\D^c))$.   Applying
  Lemma~\ref{esub}(b) to $(\alpha(\C),\alpha(\D^c))$ yields 
  $\alpha(\N(\C,\D^c))\subseteq \N(\C_1,\D_1)$.  Thus,
  $\alpha:(\C,\D^c)\rightarrow (\C_1,\D_1)$ is a regular
  $*$-monomorphism.  We conclude that $(\C_1,\D_1,\alpha)$
  is Cartan envelope for $(\C,\D^c)$.

  For the converse, suppose $(\C_1,\D_1,\alpha)$ is a Cartan envelope
  for $(\C,\D^c)$.   Then $\alpha(\D)\subseteq \alpha(\D^c)\subseteq
  \D_1$.   As $(\D_1,\alpha(\D^c))$ and $(\D^c, \D)$ are both
  essential inclusions, Lemma~\ref{abeless} shows $(\D_1, \alpha|_\D)$
  is an essential extension of $\D$.    As $\D_1$ is generated by
  $\bbE_1(\alpha(\C))$, the proposition follows. 
\end{proof}

  Theorem~\ref{ccequiv} characterizes which inclusions have a Cartan
  envelope.  Of course, the process of passing from an inclusion to its
  Cartan envelope loses information:   for example, if $\D$ is an abelian
  \cstaralg, and $\D_0$ is any essential \cstar-subalgebra of $\D$, then
  $(\D,\D)$ is the Cartan envelope for $(\D,\D_0)$.   Here is a more
  interesting example of very different inclusions with the same
 Cartan envelope.

\begin{example}\label{noce1}
Let $A$ and $B$ be self-adjoint unitaries in $M_2(\bbC)$ such that 
\[AB=-BA;\] we have in mind the concrete examples, 
$A=\begin{bmatrix} 0&1\\ 1&0\end{bmatrix}$ and
 $B=\begin{bmatrix} 1&0\\ 0&-1\end{bmatrix}$.  The spectra of $A$ and
$B$ are both $\{-1,1\}$, and the set $\{A, B\}$ generates 
$M_2(\bbC)$.  For $k=0,1$, let $P_k$ be the projection onto the
$(-1)^k$-eigenspace for $A$, likewise let $Q_k$ be the projection onto the
$(-1)^k$-eigenspace for $B$.   
Write $C^*(A)$ and $C^*(B)$ for the \cstar-subalgebras of $M_2(\bbC)$
generated by $A$ and $B$.  The inclusions $(M_2(\bbC), C^*(A))$ and
$(M_2(\bbC), C^*(B))$ are \cstar-diagonals,  and we let $\bbE_A$ and $\bbE_B$
denote the conditional expectations of $M_2(\bbC)$ onto $C^*(A)$ and
$C^*(B)$ respectively.  

Let
\[\bbX:=[-2,-1]\cup [1,2] \dstext{and}
  \C:=C(\bbX)\otimes M_2(\bbC);\] regard $\C$ as continuous
$M_2(\bbC)$-valued functions on $\bbX$.  Set
\[\D:=\left\{f\in \C: f(t)\in C^*(A) \text{ if  $t<0$; } \, \, f(t)\in C^*(B)
    \text{ if $t>0$} \right\}.\]
Then $(\C,\D)$ is a Cartan inclusion.

We now define two inclusions whose Cartan envelope is $(\C,\D)$.
Let \[\D_0:=\{f\in \D: f(-1)=f(1)\}\dstext{and}
      \C_0:=\{f\in\C: f(-1)=f(1)\}.\]  Then $(\D,\D_0)$ is an
    essential inclusion and a computation shows  $(\C_0,\D_0)$ is a  MASA
    inclusion.
Also, \[\{f\otimes A: f\in
C(\bbX)\}\cup\{f\otimes B: f\in C(\bbX)\}\subseteq
\N(\C_0,\D_0).\]  It now follows readily that both $(\C_0,\D_0)$ and
$(\C,\D_0)$ are regular inclusions with the unique pseudo-expectation
property, but in both cases, the pseudo-expectation is not a
conditional expectation.

Note that unlike $(\C_0,\D_0)$,  $(\C,\D_0)$ is not    a
MASA inclusion.   Nevertheless, the Cartan
    envelopes of $(\C_0,\D_0)$ and $(\C,\D_0)$ are both $(\C,\D)$.

\end{example}

  Thus, the question
  of when two inclusions have isomorphic Cartan envelopes arises.   
The following result gives a method for constructing regular inclusions with
the faithful unique pseudo-expectation property as
sub-inclusions of a given  Cartan inclusion.  
This result also provides a 
  a partial answer to the question.

\begin{proposition}\label{subsCart}   Suppose the following:
  \begin{itemize}
    \item $(\C,\D)$ is a Cartan inclusion
      with conditional expectation $\bbE:\C\rightarrow\D$;
      \item
  $\D_0\subseteq \D$ is a \cstar-subalgebra such that $(\D,\D_0)$ is an
  essential inclusion; and
  \item $\M\subseteq \N(\C,\D_0)$ is a 
    $*$-monoid such that $\D_0\subseteq \overline{\spn}\,\M$.
  \end{itemize}
If $\C_0:=\overline{\spn}\,\M$, then
   $(\C_0,\D_0)$ is a regular inclusion with the faithful unique
  pseudo-expectation property.

  Furthermore, if
  \[\D_1=C^*(\bbE(\C_0)), \quad \C_1=C^*(\C_0 \cup \D_1),\] and $\alpha:\C_0\rightarrow \C_1$
  is the inclusion map, then:
  \begin{enumerate}
    \item $(\C_1,\D_1,\alpha)$ is the Cartan envelope
      for $(\C_0,\D_0)$; and 
      \item $(\C_0, \C_0\cap \D)$ is a virtual
        Cartan inclusion.
      \end{enumerate}
    \end{proposition}
\begin{proof}
By construction, $(\C_0,\D_0)$ is a regular inclusion and
Lemma~\ref{esub}(a) shows it has the faithful unique
pseudo-expectation property.

We next show $(\C_1,\D_1)$ is a Cartan inclusion.  Part (b) of
Lemma~\ref{esub} shows $\N(\C,\D_0)\subseteq \N(\C,\D)$, so since
$\C_1$ is generated by $\D_1\cup\N(\C,\D_0)$, $(\C_1,\D_1)$ is a
regular inclusion.  As $(\D,\D_0)$ is an essential inclusion, so is
$(\D,\D_1)$.  Therefore,
Lemma~\ref{esub}(a) shows $(\C_1,\D_1)$ has the faithful unique
pseudo-expectation property.  As $\bbE|_{\C_1}$ is a faithful
conditional expectation of $\C_1$ onto $\D_1$,
Proposition~\ref{!fps->vc}(b) shows $(\C_1,\D_1)$ is a Cartan
inclusion.

 Lemma~\ref{esub}(b) shows that $\N(\C_0,\D_0)\subseteq
 \N(\C_1,\D_1)$, so the inclusion map $\alpha:\C_0\rightarrow \C_1$ is a
regular homomorphism.   Thus, $(\C_1,\D_1,\alpha)$ is
the Cartan envelope for $(\C_0,\D_0)$.   This completes the proof of
(a).

We turn to (b).  Since 
$(\D,\D_0)$ is an essential inclusion, both
$(\C_0\cap \D,\D_0)$ and $(\D,\C_0\cap \D)$ are essential inclusions.
Lemma~\ref{esub} shows $\N(\C_0,\D_0)\subseteq \N(\C,\D)$, whence
$\N(\C_0,\D_0)\subseteq \N(\C_0,\C_0\cap \D)$.    It follows that
$(\C_0,\C_0\cap \D)$ is a regular inclusion.  Also, Lemma~\ref{esub}
shows $\alpha$ is a regular homomorphism of $(\C_0,\C_0\cap \D)$ into
$(\C,\D)$.   Let $\A$ be the
relative commutant of $\C_0\cap \D$ in $\C_0$.  
By~\cite[Proposition~5.3(b)]{PittsStReInI}, $\A\subseteq \D$.  Thus,
\[\C_0\cap \D\subseteq \A\subseteq \C_0\cap \D,\] so $(\C_0,\C_0\cap
\D)$ is a regular MASA inclusion.  Finally, as the faithful unique
pseudo-expectation property is hereditary from above and $(\C_0,\D_0)$
has that property, so does $(\C_0,\C_0\cap \D)$.  Thus $(\C_0,\C_0\cap
\D)$ is a virtual Cartan inclusion. 

\end{proof}

\section{Some Consequences of the Unique Pseudo-Expectation
  Property} \label{UPSEC}

This section has two main purposes.  One goal is to show that regular
inclusions with the unique pseudo-expectation property are covering
inclusions.  This gives a class of regular inclusions for which the
results of Section~\ref{GpReIn} can be used for descriptions of
\cover s via twists.
The second goal is to  record some consequences of the
unique pseudo-expectation property for regular inclusion with the hope
that they may be useful in obtaining a characterization of the unique
pseudo-expectation property.

Several of
the results of this section extend results of~\cite{PittsStReInI} from the setting of
regular MASA inclusions (or skeletal MASA inclusions) considered there
to simply assuming the unique pseudo-expectation property.  While
some of these are routine adaptations of proofs
in~\cite{PittsStReInI}, others are more complicated.

    When $(\C,\D)$ has the unique pseudo-expectation property and
$E:\C\rightarrow I(\D)$ is the pseudo-expectation, recall
from~\eqref{defscs} that
\[\fS_s(\C,\D)=\{\rho\circ E: \rho\in \widehat{I(\D)}\}.\]
We already have observed that if $(\C,\D)$ is a regular inclusion with
the \textit{faithful} unique pseudo-expectation property, then
$\fS_s(\C,\D)$ is a compatible cover for $\hat\D$.  We shall see that
this holds when the faithfulness hypothesis on the unique
pseudo-expectation is removed: Theorem~\ref{upse=>cov} shows that in
the presence of the unique pseudo-expectation property, $(\C,\D)$ is a
covering inclusion and that $\fS_s(\C,\D)$ is a compatible cover for
$\hat\D$ which is contained in all compatible covers.

The proof of the following statements are straightforward adaptations
of the proofs of corresponding results found in~\cite{PittsStReInI}. 
\begin{theorem} \label{2fP} Let $(\C,\D)$ be a regular inclusion with the unique pseudo-expectation property, let $(I(\D),\iota)$ be an injective envelope for $\D$ and let $E$ be the pseudo-expectation.
  The following statements hold.
  \begin{enumerate}
  \item The map $\dual{E}: \widehat{I(\D)}\rightarrow\Mod(\C,\D)$ is the unique continuous map of $\widehat{I(\D)}$ into $\Mod(\C,\D)$ such that for every $\rho\in \widehat{I(\D)}$, $\dual{E}(\rho)|_\D=\rho\circ\iota$.
  \item Suppose $F$ is a closed subset of $\Mod(\C,\D)$ such that $\hat\D=\{\rho|_\D: \rho\in  F\}$.  Then $\fS_s(\C,\D)\subseteq F$.
  \end{enumerate}
\end{theorem}
\begin{proof}
Part (a) follows as in the proof of~\cite[Theorem~3.9]{PittsStReInI}, and part (b) follows as in the proof of the first part of~\cite[Theorem~3.12]{PittsStReInI}.
\end{proof}

When $(\C,\D)$ is a regular MASA inclusion, it has the unique
pseudo-expectation property and the left kernel of the
pseudo-expectation is a two-sided ideal of $\C$ maximal with respect
to having trivial intersection with $\D$
(\cite[Theorem~3.15]{PittsStReInI}).  Our next goal is to show that
this result holds when the hypothesis that $\D$ is a MASA in $\C$ is
removed, that is, we extend this result to any inclusion with the
unique pseudo-expectation property.

To do this, we require the following result, which is the version
of~\cite[Proposition~3.14]{PittsStReInI} with the skeletal MASA hypothesis
removed.  For $v\in \N(\C,\D)$, we have already defined the partial automorphism, $\theta_v: \overline{vv^*\D}\rightarrow\overline{v^*v\D}$.  In the following, $\tilde\theta_v$ is the unique extension of $\theta_v$ to a partial automorphism of the regular ideals in $I(\D)$ generated by $vv^*$ and $v^*v$ respectively, see~\cite[Definition~2.13]{PittsStReInI}.

\begin{proposition}\label{ENinv}  Suppose $(\C,\D)$ is a regular inclusion with the unique pseudo-expectation property, let $(I(\D),\iota)$ be an injective envelope for $\D$ and let $E:\C\rightarrow I(\D)$ be the pseudo-expectation.  Then for every $v\in \N(\C,\D)$ and $x\in \C$,
  \begin{equation}\label{ENinv1} E(v^*xv)=\tilde\theta_v(E(vv^*x)).
  \end{equation}
\end{proposition}
\begin{proof}
  We begin by establishing~\eqref{ENinv1} for $x\in \D^c$. First suppose $x^*=x\in\D^c$ and put
  \[\S:=\{d\in \D_{s.a.}: d\geq x\}.\]  We claim that
  \begin{equation}\label{ENinv2}
    \tilde{\theta_v}(E(vv^*x))=\inf_{I(\D)_{s.a.}} \iota(v^*\S v) \geq E(v^*xv).
  \end{equation}
  Recall that for any $f\in \overline{vv^*\D}$,  $\iota(\theta_v(f))=
  \tilde\theta_v(\iota(f))$.  So
  \begin{align*}
    d\in \S &\Rightarrow  \iota(d)\geq E(x) \Rightarrow \iota(dvv^*)\geq E(xvv^*)\\
            &\Rightarrow
              \iota(\theta_v(dvv^*))=\tilde\theta_v(\iota(dvv^*))\geq
              \tilde\theta_v(E(xvv^*))\\
            &\Rightarrow \iota(v^*dv)\geq \tilde\theta_v(E(xvv^*)) \\
            &\Rightarrow \inf_{I(\D)_{s.a.}} v^*\S v\geq \tilde\theta_v(E(xvv^*)).
  \end{align*}

  Next, let $Q\in I(\D)$ be the support projection for
  $\overline{v^*v\D}$, that is,
  $Q=\sup_{I(\D)_{s.a.}} \iota(u_\lambda)$ where $(u_\lambda)$ is an
  approximate unit for $\overline{v^*v\D}$.  Suppose
  $y\in I(\D)_{s.a.}$ satisfies \[y\leq \iota(v^*dv)\] for every
  $d\in\S$.  Then $Qy\leq Q\iota(v^*dv)=\iota(v^*dv)$ for $d\in\S$.  Thus, for
  every $d\in \S$,
  \[\tilde\theta_v^{-1}(Qy)\leq \tilde\theta_v^{-1}(\iota(v^*dv))
                                       =\tilde\theta_v^{-1}(\tilde\theta_v(\iota(vv^*d)))=\iota(vv^*d).
                                     \] Therefore,
                                     \[\tilde\theta_v^{-1}(Qy)\leq \inf_{I(\D)_{s.a.}} (vv^*\S)\overset{(1)}{=}\iota(vv^*)\inf_{I(\D)_{s.a.}} \S\overset{(2)}{=} \iota(vv^*)E(x)=E(vv^*x),\] with equality~(1) following from~\cite[Lemma~1.9]{HamanaReEmCStAlMoCoCStAl} and equality~(2) by \cite[Theorem~3.16]{PittsZarikianUnPsExC*In}.   
Thus,  \[Qy\leq\tilde\theta_v(E(xvv^*)).\]
Observe that $Q^\perp y\leq 0$ because 
$Q\iota(v^*v)=\iota(v^*v)$ and $Q^\perp y$ is a lower bound for
$Q^\perp\iota(v^*\S v) = \{0\}$.   Therefore, $y\leq Qy$, so
$y\leq \tilde\theta_v(E(xvv^*))$.  Thus, $\inf_{I(\D)_{s.a.}}
\iota(v^*\S v)\leq \tilde\theta_v(E(xvv^*))$, which completes the proof  
of the equality in~\eqref{ENinv2}.

If $d\in \S$, then $v^*dv\geq v^*xv$, so $\iota(v^*dv)\geq E(v^*xv)$.  Thus
$E(v^*xv)$ is a lower bound for $\iota(v^*\S v)$, which gives the inequality in~\eqref{ENinv2}.
  
Replacing $x$ with $-x$ in~\eqref{ENinv2} yields
$\tilde\theta_v(E(vv^*x))\leq E(v^*xv)$, whence~\eqref{ENinv1} holds
for any $x\in (\D^c)_{s.a.}$.  It follows from linearity that~\eqref{ENinv1} holds for all $x\in\D^c$.

  The remainder of the proof is a modification of the proof of~\cite[Propostion~3.14]{PittsStReInI}.   Since $(\C,\D)$ is a regular inclusion, it suffices to show~\eqref{ENinv1} holds for any $w\in \N(\C,\D)$.

By Lemma~\cite[Lemma~3.3]{PittsStReInI}, it suffices to show that the ideal
\[H:=\{d\in\D: (E(v^*wv)-\tilde\theta_v(E(vv^*w))\iota(d)=0\}\] is an
essential ideal of $\D$. 

So let $w\in\N(\C,\D)$.  Let $\{K_{i}\}_{i=0}^4$ be a left Frol\' ik family
of ideals for $w$, let $J=\overline{vv^*\D}$ and let $P$ be the
support projection in $I(\D)$ for $J$.   Let 
\[A:=\theta_v(J)^\perp \cup \bigcup_{i=0}^4 \theta_v(J\cap K_i).\]
Then $A^\perp=\{0\}$, so 
$A$ generates an essential ideal of $\D$.  To show $H$ is an essential ideal, we show 
$A\subseteq H$.  

The following facts follow as in the corresponding facts in the proof of~\cite[Theorem~3.14]{PittsStReInI}:
\begin{itemize}
\item $\theta_v(J)^\perp\subseteq H$; and
\item for $1\leq i\leq 4$, $\theta_v(J\cap K_i)\subseteq H$.
\end{itemize}
We now show that $\theta_v(J\cap K_0)\subseteq H$.  Let $d\in J\cap K_0$.
Lemma~\ref{fixedideal} gives $wd=dw\in \D^c$.  Thus,
\begin{align*}
E(v^*wv)\iota(\theta_v(d))&=
                            E(v^*wdv)
                            =\tilde\theta_v(E(vv^*wd))=\tilde\theta_v(E(vv^*w))\iota(\theta_v(d)).
\end{align*}
Therefore, $\theta_v(J\cap K_0)\subseteq H$ as
well. 

We conclude that $A\subseteq H$, which completes the proof.
\end{proof}

With the previous proposition in hand, we have the following result which  extends~\cite[Theorem~3.15]{PittsStReInI}.  

\begin{theorem}\label{Eideal}  Let $(\C,\D)$ be a regular inclusion with the unique pseudo-expectation property.  Let $E$ be the pseudo-expectation and set $\L(\C,\D):=\{x\in \C: E(x^*x)=0\}$. Then 
$\L(\C,\D)$ is an ideal of $\C$ such that 
$\L(\C,\D)\cap \D=(0).$  

Moreover,  if $\K\subseteq \C$ is an ideal such
that $\K\cap \D=(0)$, then $\K\subseteq \L(\C,\D)$.

\end{theorem}
\begin{proof}   The proof is the same as the proof of~\cite[Theorem~3.15]{PittsStReInI}, except that one uses Proposition~\ref{ENinv} instead of~\cite[Proposition~3.14]{PittsStReInI}.
\end{proof}

\begin{remark}{Remark}  Without the regularity assumption on
  $(\C,\D)$, Theorem~\ref{Eideal} is false,
  see~\cite[Remark~3.11]{PittsZarikianUnPsExC*In}.
\end{remark}

By~\cite[Proposition~3.6]{PittsZarikianUnPsExC*In}, the quotient of an
inclusion with the unique pseudo-expectation property by a
$\D$-disjoint ideal also has the unique pseudo-expectation property.
The maximality of $\L(\C,\D)$ allows us to conclude that the quotient
by $\L(\C,\D)$ has the faithful unique pseudo-expectation property.

\begin{corollary}\label{qupse}  Suppose $(\C,\D)$ is a regular
  inclusion with the unique pseudo-expectation property.   Let
  $\C_q:=\C/\L(\C,\D)$, let 
  $q:\C\rightarrow \C_q$ be the quotient map, and $\D_q=q(\D)$.  Then
  $q|_{\D}$ is an isomorphism of $\D$ onto $\D_q$ and 
  $(\C_q,\D_q)$ has the faithful unique pseudo-expectation property.
\end{corollary}
\begin{proof}
  Since $\L(\C,\D)\cap \D=(0)$, $q|_\D$ is one-to-one, so $q|_\D$ is
  an isomorphism of $\D$ onto $\D_q$, and as noted above,
  $(\C_q,\D_q)$ has the unique
  pseudo-expectation property.
Consider
  the ideal $J:=q^{-1}(\L(\C_q,\D_q))$.  If $d\in \D\cap J$, then
  $q(d)\in \L(\C_q,\D_q)\cap \D_q=(0)$, hence $d=0$.  Therefore
  $J\subseteq \L(\C,\D)$, whence $\L(\C_q,\D_q)=(0)$.  It follows that
  $E_q$ is faithful, as desired.
  \end{proof}

Now we turn to showing that inclusions with the unique
pseudo-expectation property are covering inclusions.  We begin with a
special case. 
  \begin{lemma}\label{cenpse1}  Suppose $(\C,\D)$ is an inclusion with
    the unique pseudo-expectation property.  If $\D$ is contained in
    the center of $\C$, then the unique pseudo-expectation is
    multiplicative on $\C$ and the ideal $\L(\C,\D)$ contains the
    commutator ideal of $\C$.
  \end{lemma}
\begin{proof}
Put $J:=\L(\C,\D)$, 
  let $q: \C\rightarrow \C/J$ be the quotient map, $\C_q:=\C/J$,
  and $\D_q:=q(\D)$.  Corollary~\ref{qupse} shows that $\D_q\simeq \D$
  and $(\C_q,\D_q)$ has the faithful unique pseudo-expectation
  property.  Let $E_q$ be the unique pseudo-expectation for $(\C_q,\D_q)$.

Note that $\D_q$ is contained in the center of $\C_q$, so the relative commutant of $\D_q$ in $\C_q$ is all of $\C_q$.
By~\cite[Corollary~3.14]{PittsZarikianUnPsExC*In}, $\C_q$ is abelian,
so $\L(\C,\D)$ contains the commutator ideal of $\C$.
By~\cite[Corollary~3.21]{PittsZarikianUnPsExC*In}, $E_q$ is
multiplicative.  But $E_q\circ q$ is multiplicative and is the
pseudo-expectation for $(\C,\D)$.  
\end{proof}

We now are equipped to show that $\fS_s(\C,\D)$ is the minimal compatible cover for a regular inclusion with the unique pseudo-expectation property.

\begin{theorem}\label{upse=>cov} Suppose $(\C,\D)$ is a regular
  inclusion with the unique pseudo-expectation property.  Then
  $(\C,\D)$ is a covering inclusion and $\fS_s(\C,\D)$ is a compatible
  cover for $\hat\D$.  Furthermore, if $F$ is any closed subset of
  $\Mod(\C,\D)$ which covers $\hat\D$, then $\fS_s(\C,\D)\subseteq F$.
  \end{theorem}
  \begin{proof}
We claim that whenever $v\in\N(\C,\D)$ and $\rho_0\in
\widehat{I(\D)}$ satisfies $\rho_0(E(v))\neq 0$, then $\rho_0\circ
\iota\in (\fix\beta_v)^\circ$.
Denote by $r$ the ``restriction'' map, $\widehat{I(\D)} \ni
\rho\mapsto \rho\circ\iota\in \hat\D$ and 
  let 
 \[X:= \{\rho'\in \widehat{I(\D)} : |\rho'(E(v))|>|\rho_0(E(v))|/2\}.\] 
As $\widehat{I(\D)}$ is Stonean, $\overline X$ is clopen, so  $\overline X\in \ropen(\widehat{I(\D)})$.   Let $G:=(r(\overline X))^\circ$.  Lemma~\ref{isolattice} shows that $r^{-1}(G)=\overline X$.   In particular, $\rho_0\circ \iota \in G$.

 If $\sigma\in G$, then $\sigma=\rho'\circ\iota$ for some $\rho'\in\overline X$, and, by definition of $X$, $|\rho'(E(v))|\neq 0$.  By ~\cite[Lemma~2.5]{PittsStReInI}, $(\rho'\circ E)|_\D\in \fix\beta_v$.  But $(\rho'\circ E)_\D=\rho'\circ\iota$.  Thus,
 \[\rho_0\circ\iota\in G \subseteq (\fix\beta_v)^\circ.\] 
Therefore, the claim holds. 
 
 Let $\tau\in \fS_s(\C,\D)$ and suppose $\tau(v)\neq 0$ for some $v\in
 \N(\C,\D)$.   Write $\tau=\rho\circ E$ for some $\rho\in
 \widehat{I(\D)}$.   Set $\sigma=\tau|_\D$; note that
 $\sigma\in\hat\D$ and $\sigma(v^*v)>0$.  By the claim,  $\sigma\in (\fix\beta_v)^\circ$.  By Lemmas~\ref{fixptchar} and~\ref{fixedideal}, there exists $d\in \D$ such that $\sigma(d) =1$ and $\supp d\subseteq (\fix\beta_v)^\circ$.   Then $dv=vd\in \D^c$.    By Lemma~\ref{cenpse1}, $E$ is multiplicative on $\D^c$, so
\[ |\tau(v)|^2=|\tau(vd)|^2=\rho(E(d^*v^*))\rho(E(vd))=\rho(E(d^*v^*vd))=\tau(v^*v).\]
It follows that $\tau$ is a compatible state.

Let us show the invariance of $\fS_s(\C,\D)$.  Choose $\tau\in
\fS_s(\C,\D)$ and write $\tau=\rho\circ
E$ for some $\rho\in \widehat{I(\D)}$.  Suppose $v\in \N(\C,\D)$ is
such that $\rho(\iota(v^*v))\neq 0$ and let $P$ and $Q$ be the support
projections in $I(\D)$ for $\overline{vv^*\D}$ and $\overline{v^*v\D}$
respectively.    Then $\tilde\theta_v$ is a partial automorphism with
domain $PI(\D)$ and range $QI(\D)$.   Define $\tau'\in
\widehat{I(\D)}$ by \[\tau'(h)=\rho(\tilde\theta_v(Ph)), \qquad h\in
  I(\D).\]
  For $x\in \C$,
Proposition~\ref{ENinv} gives,
\begin{align*}
  \rho(E(v^*xv))&= \rho(\tilde\theta_v(E(vv^*x)))=
                  \rho(\tilde\theta_v(\iota(vv^*) PE(x)))\\
  &=\rho(\iota(v^*v))\rho(\tilde\theta_v(PE(x)))=
    \rho(\iota(v^*v))(\tau'\circ E)(x).
\end{align*}
Thus $\fS_s(\C,\D)$ is invariant.

If $\sigma\in\hat\D$, choose any $\rho\in \widehat{I(\D)}$ such that $\rho\circ\iota=\sigma$.  Then $\sigma=(\rho\circ E)|_\D$, so $\fS_s(\C,\D)$ covers $\hat\D$.  Thus, $\fS_s(\C,\D)$ is a compatible cover for $\hat\D$ and $(\C,\D)$ is a covering inclusion.

Finally, if $F\subset \Mod(\C,\D)$ is closed and covers $\hat\D$, then $\fS_s(\C,\D)\subseteq F $ by Theorem~\ref{2fP}(b).

\end{proof}

\begin{example}\label{narelco}
  As noted previously, when $(\C,\D)$ has the faithful unique
  pseudo-expectation property, $\D^c$ is abelian, but when $(\C,\D)$
  merely has the unique pseudo-expectation property, $\D^c$ need not
  be abelian.  We now outline a method for constructing examples of
  this behavior.  This also provides a negative answer
  to~\cite[Question~5]{PittsZarikianUnPsExC*In}.

Suppose $(\A, \B)$ is an inclusion with the unique pseudo-expectation
property and suppose $\A$ is abelian.   Let $J:=\L(\A,\B)$ and assume
$J\neq (0)$.   (See~\cite[Corollary~3.21]{PittsZarikianUnPsExC*In} for
a characterization of such inclusions.)   Define
\[\C:=\left\{\begin{bmatrix}b+j_{11}&j_{12}\\ j_{21}&
      b+j_{22}\end{bmatrix}: b\in \B, j_{mn}\in
    J\right\}\dstext{and} \D:=\left\{\begin{bmatrix}
    b&0\\0&b\end{bmatrix} : b\in \B\right\}\simeq \B.\]
Then $\D$ is contained in the center of $\C$, so $(\C,\D)$ is a
regular inclusion and  $\D^c=\C\supseteq M_2(J)$.  Thus $\D^c$ is not
abelian.   Let $E$ be the pseudo-expectation for $(\A,\B)$ and suppose
$\Delta$ is a pseudo-expectation for $(\C,\D)$.   Then $a\in \A
\mapsto \Delta(a\oplus a)$ is a pseudo-expectation for $(\A,\B)$.  So
for $j\in J=\ker E$, the
fact that $(\A,\B)$ has the unique-pseudo expectation property gives
$\Delta(j\oplus j)=0$.  Also, for $j_1, j_2\in J$, applying $\Delta$
to each operator in  the inequality,
\[ -(|j_1|+|j_2|)\oplus -(|j_1|+|j_2|) \leq
  j_1\oplus j_2\leq (|j_1|+|j_2|)\oplus
  (|j_1|+|j_2|)\] gives $\Delta(j_1\oplus j_2)=0$.  Note that
$\Delta\left(\begin{bmatrix} 0 &j_1\\
  j_2&0\end{bmatrix}\right)=0$ because
\[\Delta\left(\begin{bmatrix} 0& j_1\\ j_2&0\end{bmatrix}\right)^*
  \Delta\left(\begin{bmatrix} 0& j_1\\ j_2&0\end{bmatrix}\right) \leq \Delta\left( \begin{bmatrix} 0& j_1\\ j_2&0\end{bmatrix}^*
  \begin{bmatrix} 0& j_1\\ j_2&0\end{bmatrix}\right)=\Delta( |j_2|^2\oplus
  |j_1|^2)=0.\]   Thus $M_2(J)\subseteq \ker\Delta$. If $x\in
\C$, write $x=(b\oplus b) + y$ where $b\in \B$ and $y\in M_2(J)$.
Then  $\Delta(x)=E(b)$, and it follows that $(\C,\D)$ has the unique
pseudo-expectation property. 

\end{example}

Given an inclusion $(\C,\D)$ and an injective envelope $(I(\D),\iota)$
for $\D$, we use $\PsExp(\C,\D)$ for the collection
of all pseudo-expectations for $(\C,\D)$ relative to $(I(\D),\iota)$.

A consequence of the following result is that $(\C,\D)$ has the unique
pseudo-expectation property if and only if $(\D^c,\D)$ does. 
\begin{proposition}\label{bipseudo}  Let $(\C,\D)$ be a regular
  inclusion.  The map $\Phi: \PsExp(\C,\D)\rightarrow \PsExp(\D^c,\D)$
  given by $\Delta\mapsto \Delta|_{\D^c}$ is a bijection.
\end{proposition}
\begin{proof} Clearly when $\Delta\in \PsExp(\C,\D)$,
  $\Delta|_{\D^c}\in \PsExp(\D^c,\D)$.  The fact that $\Phi$ is onto
  follows from injectivity.  To show $\Psi$ is one-to-one, suppose
  $\Delta_1, \Delta_2\in \PsExp(\C,\D)$ and
  $\Delta_1|_{\D^c}=\Delta_2|_{\D^c}$.  We now argue as in the proof of~\cite[Proposition~3.4]{PittsStReInI}.  Here
  is an outline.  Let $v\in \N(\C,\D)$ and let
\[J:=\{d\in \D: (\Delta_1(v)-\Delta_2(v))\iota(d)=0\}.\]  To obtain
$\Delta_1=\Delta_2$, it suffices to show $J$ is an essential ideal of $\D$.
  Let $\{K_i\}_{i=0}^4$ be a left Fr\'olik family of
  ideals for $v$.  For $d\in K_0$, $vd=dv\in \D^c$, so
\[\Delta_1(v)\iota(d)=\Delta_1(vd)=\Delta_2(vd)=\Delta_2(v)\iota(d).\]
Thus, $K_0\subseteq J$.  Establishing $K_i\subseteq J$ for $1\leq i\leq
4$ is done exactly as in the proof
of~\cite[Proposition~3.4]{PittsStReInI}.  Thus $\cup_{i=0}^4
K_i\subseteq J$.   As $\cup_{i=0}^4 K_i$
generates an essential ideal of $\D$,  $J$ is essential.
By~\cite[Lemma~3.3]{PittsStReInI}, $\Delta_1(v)=\Delta_2(v)$.   As
$\spn\N(\C,\D)$ is dense in $\C$, $\Delta_1=\Delta_2$.  Thus $\Phi$
is one-to-one.   
\end{proof}

The following result extends~\cite[Theorem~3.5]{PittsStReInI} to the
setting of regular inclusions $(\C,\D)$ for which $\D^c$ is abelian.
\begin{corollary} \label{upsedc}  Let $(\C,\D)$ be a regular inclusion such that
  $\D^c$ is abelian.  Let $X:=\hat\D$, $Y:=\widehat{\D^c}$, and let
  $r:Y\twoheadrightarrow X$ be the restriction mapping.  The following
  statements are equivalent.
  \begin{enumerate}
  \item $(\C,\D)$ has the unique pseudo-expectation property.
    \item There exists a unique minimal closed set $F\subseteq Y$ such
      that $r(F)=X$.
      \item There exists a unique maximal $\D$-disjoint ideal of $\D^c$.
      \end{enumerate}
    \end{corollary}
\begin{proof}  Combine Proposition~\ref{bipseudo}
  with~\cite[Corollary~3.21]{PittsZarikianUnPsExC*In}.
\end{proof}

We conclude with the following conjecture regarding  characterizations
of the unique pseudo-expectation property.  
\begin{conjecture}\label{upseconj}  Let $(\C,\D)$ be a regular
  inclusion and let $\fM$ be the set of all multiplicative linear
  functionals on $\D^c$.   The following statements are equivalent.
  \begin{enumerate}
    \item $(\C,\D)$ has the unique pseudo-expectation
    property.
    \item $(\C,\D)$ is a covering inclusion and there exists a
      $\D$-disjoint ideal $J\subseteq \C$ with the following property:
      if $I\subseteq \C$ is a $\D$-disjoint ideal of $\C$, then
      $I\subseteq J$.
      \item $(\C,\D)$ is a covering inclusion and there exists a
        compatible cover $F$ for $\hat\D$ with the following property:
        if 
        $C\subseteq \Mod(\C,\D)$ is closed and covers $\hat\D$, then 
        $F\subseteq C$.
      \item $(\D^c,\D)$ has the unique pseudo-expectation property.
        \item Every multiplicative linear functional on $\D$ extends
          to an element of $\fM$ and there exists a 
          $\D$-disjoint ideal $J\subseteq \D^c$ with the following property:
      if $I\subseteq \D^c$ is a $\D$-disjoint ideal of $\D^c$, then
      $I\subseteq J$.
      \item Every multiplicative linear functional on $\D$ extends
          to an element of $\fM$ and there exists a closed subset
          $F\subseteq \fM$ with the following properties:  $F$ covers
          $\hat\D$ and if $C\subseteq \Mod(\D^c,\D)$ is closed and covers $\hat
          \D$, then $F\subseteq C$.
      \end{enumerate}
    \end{conjecture}
\begin{remark}{Remark}   Here are some comments regarding
  this conjecture.
    \begin{itemize}
\item (a)$\Rightarrow$(b) and (a)$\Rightarrow$(c) by Theorems~\ref{Eideal}
  and~\ref{upse=>cov}.
\item  Suppose (c) holds.  Then $\K_F:=\{x\in \C: \rho(x^*x)=0 \text{
    for all } \rho\in F\}$ is an ideal of $\C$ such that $\K_F\cap
  \D=(0)$.  We expect this ideal to have the property in (b).  Furthermore,
  $(F,r)$ is an essential cover for
  $\hat\D$, and if $(\C_q,\D_q):=(\C/\K_F, \D/(\K_F\cap \D))$, it
  seems reasonable to 
  expect $(\C_q,\D_q)$ has the faithful unique pseudo-expectation
  property (see Proposition~\ref{twistprop} below).  Because we have been unable to
  establish that every
  pseudo-expectation annihilates $\K_F$, it is not clear
 this implies $(\C,\D)$ has the unique
  pseudo-expectation property, 
\item (a)$\Leftrightarrow$(d) by Proposition~\ref{bipseudo}.
  \item The sets $\fS(\D^c,\D)$ and $\fM$ coincide
    by~\cite[Theorem~4.13]{PittsStReInI}.  In particular $(\D^c,\D)$ is a
    covering inclusion if and only if $\fM$ covers $\hat\D$.
  \end{itemize}
\end{remark}

\section{Twists Associated to a Regular Covering Inclusion}\label{GpReIn}

Let $(\C,\D)$ be an inclusion for which there exists a compatible
cover $F\subseteq \fS(\C,\D)$ for $\hat\D$.  The main result of this
section, Theorem~\ref{inctoid}, shows that given this data, there
exists a twist $(\Sigma, G)$ and a regular $*$-homomorphism $\theta$
of $(\C,\D)$ into the the inclusion $(C^*_r(\Sigma,G), C(\unit{G}))$
such that $F$ is identified with the unit space of $G$ and
$\theta(\C)$ is dense in $C^*_r(\Sigma)$ with respect to a pointwise
topology.  The kernel of $\theta$ is the ideal
$\K_F=\{x\in\C: \rho(x^*x)=0 \text{ for all } \rho\in F\}$, and when
this ideal vanishes, $(C^*_r(\Sigma,G), C(\unit{G}), \theta)$ is a
\cover\ for $(\C,\D)$.  The fact that elements of $F$ are
compatible states is the key ingredient in our construction of our
groupoids and $\theta$. 

Also, when $(\C,\D)$ has the faithful unique pseudo-expectation
property, and $F$ is taken to be the family of strongly compatible
states, this construction produces the Cartan envelope for $(\C,\D)$.
Theorem~\ref{inctoid} is a refinement of the embedding results of
Section~5 of \cite{PittsStReInI}.

Our construction is inspired by the constructions by Kumjian and
Renault, but the construction here differs from theirs in significant
ways.  Most significantly, in the Renault and Kumjian contexts, a conditional
expectation $E:\C\rightarrow \D$ is present, and since
$\{\rho\circ E: \rho\in\hat{\D}\}$ is homeomorphic to $\hat{\D}$,
Renault and Kumjian use $\hat{\D}$ as the unit space for the twists
they construct.  In our context, we need not have a conditional
expectation, so instead of using $\hat\D$ as the unit space, we use
the set $F$ instead.

In order that the coordinates determined by the twist $(\Sigma_F,G_F)$
reflect as many of the properties of the original covering inclusion
$(\C,\D)$ as
possible, it is desirable that the choice of 
$F\subseteq \fS(\C,\D)$ be made as small as  possible.
If $F_1\subseteq F_2$ are two $\N(\C,\D)$-invariant, closed and
$\D$-covering subsets of $\fS(\C,\D)$, 
$\fJ_{F_2}\subseteq \fJ_{F_1}$.  However, it is not clear whether the
triviality of the ideal $\fJ_{F_2}$ implies $\fJ_{F_1}$ is also
trivial.   Thus, it may be that $\C$ regularly embeds into
$C^*_r(\Sigma_{F_2},G_{F_2})$  but does not regularly embed into
$C^*_r(\Sigma_{F_1},G_{F_1})$.   This leads to the following
question.
\begin{question}\label{radJF}  Suppose $(\C,\D)$ is a regular inclusion and
  $F\subseteq \fS(\C,\D)$ is closed, invariant and covers $\hat\D$.
  Let $r:F\rightarrow \hat\D$ be the restriction map.  If
  $(\fS(\C,\D), r)$ is an essential cover for $\hat\D$, must
  $\rad(\C,\D)=\fJ_F$?  If $\rad(\C,\D)=(0)$, is $\fJ_F=(0)$ also?
\end{question}

\begin{example} \label{counradJF}  This example illustrates the need
  for the hypothesis that $(\fS(\C,\D), r)$ is an essential cover for
  $\hat\D$ in Question~\ref{radJF}.  Take $\C=C[0,1]$ and $\D=\bbC I$.  Then
  $\N(\C,\D)=\{\lambda U: \lambda\in \bbC, U\in \U(\C)\}$.
  By~\cite[Theorem~4.13]{PittsStReInI}, $\fS(\C,\D)$ is the set of all
  multiplicative linear functionals on $\C$.  Thus $\rad(\C,\D)=(0)$
  However, every non-empty closed subset $F\subseteq \fS(\C,\D)$
 is invariant and covers the singleton set $\hat\D$.  In particular,
 it is possible for $\fJ_F$ to be a maximal ideal.

 However, when $\C$ is abelian and $(\C,\D)$ is an essential
 inclusion, Question~\ref{radJF} has an affirmative answer.   As
 before, $\fS(\C,\D)=\hat\C$, and a closed set $F\subseteq \fS(\C,\D)$ is a
 compatible cover for $\hat\D$ when $\fJ_F\cap \D=(0)$.  As $(\C,\D)$ is essential,
 $\fJ_F=(0)=\rad(\C,\D)$.     
\end{example}

Recall (see~\cite{DonsigPittsCoSyBoIs}) that an \textit{eigenfunctional}
is a non-zero element $\phi\in\dual{\C}$ which is an eigenvector for both the
left and right actions of $\D$ on $\dual{\C}$; when this occurs, there
exist unique elements $\rho,\sigma\in\hat{\D}$ so that whenever
$d_1,d_2\in\D$ and $x\in \C$, we have
\begin{equation}\label{srei}
  \phi(d_1xd_2)=\rho(d_1)\phi(x)\sigma(d_2).
\end{equation}

As a simple example, consider the inclusion, $(M_n(\bbC), D_n)$, where
$D_n$ is  the set of  $n\times n$ diagonal matrices.  The eigenfunctionals for
this inclusion are non-zero scalar multiples of
functionals  of the form $\phi_{ij}: A\mapsto
\innerprod{Ae_j, e_i}$; here $\{e_i\}$ is the usual orthonormal  basis for
$\bbC^n$.
For $\phi_{ij}$ and $d\in D_n$,
$\rho(d)=\innerprod{de_i, e_i}$ and $\sigma(d)=\innerprod{de_j,e_j}$. 

For a general regular inclusion and eigenfunctional $\phi$, the
elements $\rho, \sigma\in \hat\D$ appearing in~\eqref{srei} are to be
regarded as range and source maps for $\phi$, and we write
$$s(\phi):=\sigma\dstext{and} r(\phi):=\rho.$$  It will be clear
from the context whether $r$ refers to the range of an
eigenfunctional or a restriction mapping (as
used earlier).
\begin{remark}{Remark} Notice that when $\phi$ is an eigenfunctional
  and $v\in \N(\C,\D)$ is such that $\phi(v)\neq 0$, then for every
  $d\in\D$, 
 \begin{equation}\label{sreig}
\frac{s(\phi)(v^*dv)}{s(\phi)(v^*v)}=r(\phi)(d).
\end{equation} Indeed,
$\phi(v)\, s(\phi)(v^*dv)
=\phi(vv^*dv)=\phi(dvv^*v)=r(\phi)(d)\,\phi(v)\, s(\phi)(v^*v)$. 
\end{remark}

\begin{definition} A \textit{compatible eigenfunctional} is a
  eigenfunctional $\phi$ such that for every $v\in
 \N(\C,\D)$,
\begin{equation}\label{cedef}
|\phi(v)|^2\in\{0,s(\phi)(v^*v)\}\dstext{(equivalently, $|\phi(v)|^2\in\{0,r(\phi)(vv^*)\}$).}
\end{equation}
\end{definition}

\begin{remark}{Notation}\label{eignotation} 
We use the following notation.
\begin{enumerate}
\item $\E(\C,\D)$ (respectively $\ce(\C,\D)$)  will denote the set consisting of the
  zero functional together with the set of all 
  eigenfunctionals (resp. the set of all compatible eigenfunctionals
  together with the zero functional).
\item  Let $\Eigone(\C,\D)$ be the set of all
  eigenfunctionals with unit norm.  Likewise 
  $\ceo(\C,\D)$ will denote the compatible eigenfunctionals of unit
  norm. 
\end{enumerate}
  Equip $\E(\C,\D)$, $\Eigone(\C,\D)$, $\ce(\C,\D)$
  and $\ceo(\C,\D)$ with the relative
  $\sigma(\dual{\C},\C)$ topology.
\begin{enumerate}
\setcounter{enumi}{2}
\item For $v\in\N(\C,\D)$ and $f\in \Mod(\C,\D)$
  such that $f(v^*v)>0$, let $[v,f]\in\dual{\C}$ be defined by
  \[[v,f](x):=\frac{f(v^*x)}{f(v^*v)^{1/2}}= 
\innerprod{x+L_f,\frac{v+L_f}{\norm{v+L_f}_{\H_f}}}_{\H_f}.\]
  (This notation is borrowed from Kumjian~\cite{KumjianOnC*Di}.
  There, Kumjian works in the context of \cstar-diagonals and uses
  states on $\C$ of the form $\sigma\circ E$ with $\sigma\in\hat{\D}$.)
\end{enumerate}
\end{remark}

A calculation shows that if $v\in\N(\C,\D)$ and $f\in\Mod(\C,\D)$ satisfy
$f(v^*v)\neq 0$, then 
\[\phi:=[v,f]\text{ belongs to } \Eigone(\C,\D), \,\, \, 
s(\phi)=f|_\D\dstext{and} r(\phi)=\beta_v(s(\phi)).\]   Our goal is to show
that in fact, all elements of $\Eigone(\C,\D)$ arise in this way, and
furthermore,  all elements of $\ceo(\C,\D)$ have the form $[v,f]$, where $f\in
\fS(\C,\D)$.  
We first show that associated with each $\phi\in \Eigone(\C,\D)$ is a  pair
$f,g\in \Mod(\C,\D)$ which extend $r(\phi)$ and $s(\phi)$ and describe
some properties of these preferred extensions.
Note that regularity of the inclusion
$(\C,\D)$ ensures the existence of $v\in \N(\C,\D)$ such that $\phi(v)>0$.

\begin{theorem}\label{cefform}  Let $(\C,\D)$ be a regular inclusion
  and let $\phi\in\Eigone(\C,\D)$.   
The following statements hold.
\begin{enumerate}
\item  There are unique elements 
 $\fs(\phi), \, \fr(\phi)\in \Mod(\C,\D)$ such that whenever $v\in\N(\C,\D)$ satisfies
$\phi(v)\neq 0$ and $x\in\C$,
\[\phi(vx)=\phi(v)\, \fs(\phi)(x), \quad \phi(xv)=\fr(\phi)(x)\,
\phi(v).\]
The functionals $\fs(\phi)$ and $\fr(\phi)$ satisfy, 
\begin{enumerate}
\item[(i)] $s(\phi)=\fs(\phi)|_\D$ and  $r(\phi)=\fr(\phi)|_\D;$
\item[(ii)]  if  
$v\in\N(\C,\D)$ and  $\phi(v)\neq 0$,  then for every $x\in \C$,
  \[\fs(\phi)(v^*xv)=\fs(\phi)(v^*v)\fr(\phi)(x)\dstext{and}
\fr(\phi)(vxv^*)=\fr(vv^*)\fs(\phi)(x).\]
\end{enumerate}
\item If $v\in \N(\C,\D)$ satisfies $\phi(v)>0$, then 
$\phi=[v,\fs(\phi)].$
\item If $\phi$ is represented in two ways, $\phi=[v,f]=[w,g]$
  ($f, g\in\Mod(\C,\D)$, $v, w\in \N(\C,\D)$), then $f=g=\fs(\phi)$
  and $\fs(\phi)(v^*w)>0$.
\item If $\fs(\phi)\in\fS(\C,\D)$, then $\phi\in \ceo(\C,\D)$.
\end{enumerate}
In addition, when 
$\phi\in\ceo(\C,\D)$ the following hold. 
\begin{enumerate}
\setcounter{enumi}{4}
\item Both $\fs(\phi)$ and $\fr(\phi)$ belong to $\fS(\C,\D)$.
\item Suppose $v\in\N(\C,\D)$, satisfies $\phi(v)>0$.  If
  $w\in\N(\C,\D)$ and $\fs(\phi)(v^*w)>0$, then $\phi=[w,\fs(\phi)]$.
\end{enumerate}
\end{theorem}

\begin{proof} 
Begin by fixing $v\in\N(\C,\D)$ such that $\phi(v)>0$.   Define linear
functionals on $\C$ by  
\[\fs(\phi)(x):=\frac{\phi(vx)}{\phi(v)}\dstext{and}\fr(\phi)(x):=\frac{\phi(xv)}{\phi(v)}.\]
  As $\phi\in\Eigone(\C,\D)$,  $\fr(\phi)|_\D=r(\phi)$ and $\fs(\phi)|_\D=s(\phi)$.  
  We next claim that $\norm{\fs(\phi)}=\norm{\fr(\phi)}=1.$
  For any $d\in\D$ with $s(\phi)(d)=1$, replacing $v$ by $vd$ in the
  definition of $\fr(\phi)$ does not
  change $\fr(\phi)$. Thus,  if $x\in\C$ and $\norm{x}\leq 1$, we have
  $|\fr(\phi)(x)|\leq \inf\left\{\frac{\norm{vd}}{\phi(v)}: d\in \D,
  s(\phi)(d)=1\right\}=1$ (because $d$ may be chosen so that
  $\norm{vd}=\norm{d^*v^*vd}^{1/2}$ is as close to
  $s(\phi)(v^*v)^{1/2}$ as desired).  This shows
  $\norm{\fr(\phi)}=1$. Likewise $\norm{\fs(\phi)}=1$.  As $\fr(\phi)(1)=\fs(\phi)(1)=1$, 
  both $\fr(\phi)$ and $\fs(\phi)$ are states
  on $\C$ and hence belong to $\Mod(\C,\D)$.  This gives the existence
  portion of (a) and also item (i) of part (a).

We continue with the choice of $v$ made above.  Note that
$s(\phi)(v^*v)>0$ because  \[0\neq \phi(v)=\lim
\phi(v(v^*v)^{1/n})=\lim \phi(v) (s(\phi)(v^*v))^{1/n}.\]  
 Thus $\psi:=[v,\fs(\phi)]$ is defined.
 A calculation shows that for any $x\in\C$,
\[\psi(x)=\frac{r(\phi)(vv^*)^{1/2}}{\phi(v)} \phi(x).\] Thus $\psi$
is a positive scalar multiple of $\phi$ and as
$\norm{\psi}=\norm{\phi}=1$, we obtain $\phi=\psi$.  This establishes
part (b).

To verify item (ii) of part (a), taking $x=v$ in the representation, $\phi=[v,\fs(\phi)]$, we find $r(\phi)(vv^*)^{1/2}=\phi(v)$.  Similarly,
$s(\phi)(v^*v)^{1/2}=\phi(v)$.  A calculation
now shows that for $x\in\C$,
 \[\fs(\phi)(v^*xv)=\fs(\phi)(v^*v)\fr(\phi)(x).\] A similar argument yields
$\fr(\phi)(vxv^*)=\fr(\phi)(vv^*)\fs(\phi)(x)$ for each $x\in\C$.

The uniqueness portion of part (a) will follow from part (c), so we
turn to part (c) now. 
  Suppose that
   $\phi=[v,f]=[w,g].$  Computations using
   Notation~\ref{eignotation}(c) show that for every $x\in\C$, 
  $$ f(x)=\frac{\phi(vx)}{\phi(v)}\dstext{and}
  g(x)=\frac{\phi(wx)}{\phi(w)}.$$ Since
  $\frac{g(w^*v)}{g(w^*w)^{1/2}}=\phi(v)=f(v^*v)^{1/2}$, we obtain
 $$ g(w^*v)=f(v^*v)^{1/2}g(w^*w)^{1/2}>0.$$  Likewise, $f(v^*w)>0.$   Also,
 \begin{equation}\label{sames}
   f(x)=\frac{\phi(vx)}{\phi(v)}=\frac{[w,g](vx)}{[v,f](v)}=
\frac{g(w^*vx)}{f(v^*v)^{1/2}g(w^*w)^{1/2}}=
\frac{g(w^*v)g(x)}{g(w^*v)}=g(x),
\end{equation}
where the fourth equality follows
from  \cite[Proposition~4.4]{PittsStReInI}.  Part (c) now follows.

For (d), suppose $\fs(\phi)\in\fS(\C,\D)$.  If $w\in \N(\C,\D)$ and $\phi(w)\neq 0$, then
  $$|\phi(w)|^2=\frac{|\fs(\phi)(v^*w)|^2}{\fs(\phi)(v^*v)}=\frac{\fs(\phi)(v^*ww^*v)}{\fs(\phi)(v^*v)}=
  \beta_v(s(\phi))(ww^*)=r(\phi)(ww^*),$$ so $\phi$ belongs to
  $\ceo(\C,\D)$ by~\eqref{cedef}.

Next we establish part (e).  Suppose $\phi\in\ceo(\C,\D)$.  
  If $w\in \N(\C,\D)$ and $\fr(\phi)(w)\neq 0$, we have (using~\eqref{sreig}) 
  $$|\fr(\phi)(w)|^2=
  \left|\frac{\phi(wv)^2}{\phi(v)}\right|^2=
  \frac{s(\phi)(v^*w^*wv)}{s(\phi)(v^*v)} =r(\phi)(w^*w)=f(w^*w),$$
  and it follows that $\fr(\phi)\in \fS(\C,\D)$.  Likewise, $\fs(\phi)\in
  \fS(\C,\D)$.

 Turning to (f), suppose
$v,w\in\N(\C,\D)$ satisfy $\fs(\phi)(v)>0$ and  
$\fs(\phi)(v^*w)>0$.  Since
 $\fs(\phi)\in\fS(\C,\D)$, \cite[Proposition~4.4]{PittsStReInI} shows that
$\fs(\phi)(v^*w)^2=\fs(\phi)(w^*w)\fs(\phi)(v^*v)$.  Thus in the GNS Hilbert space $\H_{\fs(\phi)}$,
we have $\innerprod{v+L_{\fs(\phi)},w+L_{\fs(\phi)}}=\norm{v+L_{\fs(\phi)}}\norm{w+L_{\fs(\phi)}}$.  By the
Cauchy-Schwartz inequality, there exists a positive real number $t$ so
that $v+L_{\fs(\phi)}=tw+L_{\fs(\phi)}$.  But then for any $x\in \C$,
\[[v,\fs(\phi)](x)=\frac{\innerprod{x+L_{\fs(\phi)},v+L_{\fs(\phi)}}}{\norm{v+L_{\fs(\phi)}}}=
\frac{\innerprod{x+L_{\fs(\phi)},tw+L_{\fs(\phi)}}}{\norm{tw+L_{\fs(\phi)}}}=[w,\fs(\phi)](x).\]

\end{proof}
The following is immediate.
\begin{corollary} \label{cefformcor} For a regular inclusion $(\C,\D)$,
\begin{align}\label{cefform1}\Eigone(\C,\D)&=\{[v,f]: v\in \N(\C,\D), f\in \Mod(\C,\D)
  \text{ and } f(v^*v)\neq 0\}\quad\text{and}\\
\label{cefformc1}
\ceo(\C,\D)&=\{[v,f]: v\in \N(\C,\D), f\in \fS(\C,\D)
  \text{ and } f(v^*v)\neq 0\}.
\end{align} 
\end{corollary}

\begin{remark}{Standing Assumption}  Unless stated otherwise, for the
  remainder of this section, $(\C,\D)$ will be a regular covering
  inclusion and $F\subseteq\fS(\C,\D)$ will be a compatible cover for
  $\hat\D$.
\end{remark}

By item(ii) of Theorem~\ref{cefform}(a), the $\N(\C,\D)$-invariance of
$F$ shows that for $\phi\in\ceo(\C,\D)$, we have $\fs(\phi)\in F$ if and
only if $\fr(\phi)\in F$.  
\begin{definition} 
  Let $\ceoF(\C,\D):=\{\phi\in\ceo(\C,\D): \fs(\phi)\in F\}$.  We
  shall call $\phi\in\ceoF(\C,\D)$ an \textit{$F$-compatible
    eigenfunctional}.  By Theorem~\ref{cefform},
$$\ceoF(\C,\D)=\{[v,f]: f\in F\text{ and } f(v^*v)\neq 0\}.$$
\end{definition}

The proof of the following fact is essentially the same 
as that of~\cite[Proposition~2.3]{DonsigPittsCoSyBoIs} (the continuity
of the range and source maps follows from their definition).
\begin{proposition}\label{loccmpt} The set 
  $\ceoF(\C,\D)\cup\{0\}$ is a weak-$*$ compact subset of
  $\dual{\C}$, and the maps $\fs,\fr:\ceoF(\C,\D)\rightarrow
  \fS(\C,\D)$ are weak-$*$--weak-$*$ continuous.
\end{proposition}

We now show that $\ceoF(\C,\D)$ forms
a topological groupoid.  The topology has already been defined, so we
need to define the source and range maps, composition and inverses.
The hypothesis that $F\subseteq \fS(\C,\D)$ in the following
definition ensures that the product on $\ceoF(\C,\D)$ is well-defined. 
\begin{definition} Given $\phi\in\ceoF(\C,\D)$, let $v\in\N(\C,\D)$ be
  such that $\phi(v)> 0$.  We make the following definitions.
\begin{enumerate} 
 \item We say that $\fs(\phi)$ and $\fr(\phi)$
  are  the
  \textit{source} and \textit{range} of $\phi$ respectively.
  
\item Define the \textit{inverse}, $\phi^{-1}$ by the formula,
$$\phi^{-1}(x):=\overline{\phi(x^*)}.$$

If $\phi\in\ceo(\C,\D)$ and $v\in\N(\C,\D)$ is such that
  $\phi(v)>0$, (so that $\phi=[v,\fs(\phi)]$), then a calculation
  shows that
  $\phi^{-1}=[v^*,\fr(\phi)].$  The fact that $F$ is
  $\N(\C,\D)$-invariant 
ensures that $\phi^{-1}\in \ceoF(\C,\D)$.  Thus, our definition of $\phi^{-1}$ is
  consistent with the definition of inverse in the definition of the
  twist of a \cstardiag\ 
  arising in \cite{KumjianOnC*Di}
  and the twist of a Cartan MASA from \cite{RenaultCaSuC*Al}.

\item For $i=1,2$, let $\phi_i\in\ceoF(\C,\D)$.  We say
  that the pair $(\phi_1,\phi_2)$ is a \textit{composable pair} if
  $\fs(\phi_1)=\fr(\phi_2)$.  As is customary, we write
  $\ceoF(\C,\D)^{(2)}$ for the set of composable pairs.   
 To define the composition, choose $v_i\in
  \N(\C,\D)$ with $\phi_i(v_i)>0$, so that
  $\phi_i=[v_i,\fs(\phi_i)]$.   By item (ii) of Theorem~\ref{cefform}(a), we
  have
  $$\fs(\phi_2)(v_2^*v_1^*v_1v_2)=\fr(\phi_2)(v_1^*v_1)\fs(\phi_2)(v_2^*v_2)
  = \fs(\phi_1)(v_1^*v_1)\fs(\phi_2)(v_2^*v_2)> 0,$$ so that
  $[v_1v_2,\fs(\phi_2)]$ is defined.
  The product is then defined by  
   \[\phi_1\phi_2:=[v_1v_2,\fs(\phi_2)].\]

   We show now that this product is well defined.  Suppose that
   $(\phi_1,\phi_2)\in\ceoF(\C,\D)^{(2)}$, $f=\fs(\phi_2),$
   $\fr(\phi_2)=g=\fs(\phi_1),$ and that for $i=1,2$,
   $v_i,w_i\in\N(\C,\D)$ are such that $\phi_1=[v_1,g]=[w_1,g]$ and
   $[v_2,f]=[w_2,f].$ Then using parts (c) and (f) of Theorem~\ref{cefform}, we have
   $g(w_1^*v_1)>0$ and $f(v_2^*w_2)>0$, so, as $f\in\fS(\C,\D)$, there
   exists a positive scalar $t$ such that $v_2+L_f=tw_2+L_f$.  Hence,
\begin{align*}
 f((w_1w_2)^*(v_1v_2))&=\innerprod{\pi_f(v_1)(v_2+L_f),\pi_f(w_1)(w_2+L_f)}\\
&= 
t\innerprod{\pi_f(v_1)(w_2+L_f),\pi_f(w_1)(w_2+L_f)}\\
&=tf(w_2^*(w_1^*v_1)w_2)\\
&=tf(w_2^*w_2) \fr(\phi_2)(w_1^*v_1) \\ 
&=
tf(w_2^*w_2) \fs(\phi_1)(w_1^*v_1)\\
&=tf(w_2^*w_2) g(w_1^*v_1)>0.
\end{align*} By Theorem~\ref{cefform},   
   $[v_1v_2,f]=[w_1w_2,f]$, so that the product is well defined.

\item For $\phi\in\ceoF(\C,\D)$, denote the map $\C\ni x\mapsto
  |\phi(x)|$ by $|\phi|$.   Observe that for $\phi,
  \psi\in\ceoF(\C,\D)$, $|\phi|=|\psi|$ if and only if there exists 
  $z\in\bbT$ such that $\phi=z\psi$; clearly $z$ is unique.  Let
\begin{equation*}\fRF(\C,\D):=\{|\phi|: \phi\in\ceoF(\C,\D)\}
\end{equation*}
and define
    $\fq:\ceoF(\C,\D)\rightarrow \fRF(\C,\D)$ by \[ \fq(\phi):=|\phi|.\]
We now define inverse and product maps
in $\fRF(\C,\D)$, as well as source and range maps.

Since a state on $\C$ is
determined by its values on the positive elements of $\C$,  we may
identify $f\in F$ with $|f|\in\fRF(\C,\D)$.  Define
$\fs(|\phi|)=\fs(\phi)$ and $\fr(|\phi|)=\fr(\phi)$.  Next we define
inversion in $\fRF(\C,\D)$ by $|\phi|^{-1}=|\phi^{-1}|$, and composable
pairs by
$\fRF(\C,\D)^{(2)}:=\{(|\phi|,|\psi|):(\phi,\psi)\in\ceo(\C,\D)^{(2)}\}$,
and the product by $\fRF(\C,\D)^{(2)}\ni (|\phi|,|\psi|)\mapsto
|\phi\psi|.$
Topologize $\fRF(\C,\D)$ with the topology of point-wise convergence:
$|\phi_\lambda|\rightarrow |\phi|$ if and only if
$|\phi_\lambda|(x)\rightarrow |\phi|(x)$ for every $x\in\C$. 
This topology is the quotient topology arising from $\fq$.
We call  $\fRF(\C,\D)$ the \textit{spectral
     groupoid over $F$} for $(\C,\D)$.

 \item We have already identified $\unit{\fR_F(\C,\D)}$ with $F$.
   Define
   $\iota:\bbT\times F\rightarrow \ceoF(\C,\D)$ by
   \[\iota(z, f)=zf.\]
Then the action of $\bbT$ on $\ceoF(\C,\D)$ is 
   $(z\cdot\phi)(x)= \phi(zx)$.  Notice that if $\phi$ is written as
   $\phi=[v,f]$, where $v\in \N(\C,\D)$ and $f\in F$, then
   $z\cdot\phi=[\overline{z}v,f]$.

\end{enumerate}
\end{definition}

We now show  that $(\ceoF(\C,\D),\fRF(\C,\D),\iota,\fq)$ is a twist.

\begin{theorem}\label{cefgroupoid} Let $\ceoF(\C,\D)$ and $\fRF(\C,\D)$ be as above. 
  Then $\ceoF(\C,\D)$ and $\fRF(\C,\D)$ are locally compact Hausdorff
  topological groupoids and $\fRF(\C,\D)$ is an \'{e}tale groupoid.
  Their unit spaces are $\ceoF(\C,\D)^{(0)}
  =\fRF(\C,\D)^{(0)}=F$.  Moreover, 
\[\bbT\times F
  \overset{\iota}\hookrightarrow\ceoF(\C,\D)\overset{\fq}\twoheadrightarrow
  \fR_F(\C,\D)\] is a locally
  trivial Hausdorff topological twist.
\end{theorem}
\begin{proof}

That inversion on $\ceoF(\C,\D)$ is continuous follows readily from the
definition of inverse map and the weak-$*$ topology.  Suppose
$(\phi_\lambda)_{\lambda\in\Lambda}$ and
$(\psi_\lambda)_{\lambda\in\Lambda}$ are nets in $\ceoF(\C,\D)$ 
converging to $\phi, \psi\in \ceoF(\C,\D)$ respectively, and such that
$(\phi_\lambda,\psi_\lambda)\in \ceoF(\C,\D)^{(2)}$ for all $\lambda$. Since the
 source and range maps are continuous, we find that
 $\fs(\phi)=\lim_\lambda
 \fs(\phi_\lambda)=\lim_\lambda\fr(\psi_\lambda)=\fr(\psi)$, so
 $(\phi,\psi)\in\ceoF(\C,\D)^{(2)}.$  Let
$v,w\in\N(\C,\D)$ be such that $\phi(v)>0$ and $\psi(w)>0$.  There exists
 $\lambda_0$, so that $\lambda\geq \lambda_0$ implies 
 $\phi_\lambda(v)$ and $\psi_\lambda(w)$ are non-zero.  For each
 $\lambda\geq \lambda_0$, there exists scalars
 $\xi_\lambda,\eta_\lambda\in\bbT$ such that
 $\phi_\lambda(v)=\xi_\lambda[v,\fs(\phi_\lambda)]$ and
 $\psi_\lambda=\eta_\lambda[v,\fs(\psi_\lambda)]$.  Since
 $$\lim_\lambda\phi_\lambda(v)=\phi(v)=\lim_\lambda
 [v,\fs(\phi_\lambda)](v)
\dstext{and} 
 \lim_\lambda\psi_\lambda(v)=\psi(v)=\lim_\lambda [v,\fs(\psi_\lambda)](v),$$ we
 conclude that $\lim\eta_\lambda=1=\lim\xi_\lambda$.  So for any $x\in\C$,
\begin{align*}
(\phi\psi)(x)&=\frac{\fs(\psi)((vw)^*x)}{(\fs(\psi)((vw)^*(vw)))^{1/2}}=
\lim_\lambda
\frac{\fs(\psi_\lambda)((vw)^*x)}{(\fs(\psi_\lambda)((vw)^*(vw)))^{1/2}}
=\lim_\lambda \left([v,\fs(\phi_\lambda)][w,\fs(\psi_\lambda)]\right)(x)\\
&=\lim_\lambda(\phi_\lambda\psi_\lambda)(x),
\end{align*} giving  continuity of multiplication.
Notice that for $\phi\in\ceoF(\C,\D)$, $\fs(\phi)=\phi^{-1}\phi$ and
  $\fr(\phi)=\phi\phi^{-1}$, and $F\subseteq \ceoF(\C,\D)$.
Thus, $\ceoF(\C,\D)$ is a
locally compact Hausdorff topological groupoid with unit space
$F$.

The definitions show that $\fRF(\C,\D)$ is a groupoid.  By
construction, the map $\fq$ is continuous and is a
surjective groupoid homomorphism.  The topology on $\fRF(\C,\D)$ is
clearly Hausdorff.  If $\phi\in \ceoF(\C,\D)$, and $v\in\N(\C,\D)$ is
such that $\phi(v)\neq 0$, then $W:=\{\alpha\in\fRF(\C,\D): \alpha(v)>
|\phi(v)|/2\}$ has compact closure so $\fRF(\C,\D)$ is locally compact.
Also, if $\alpha_1, \alpha_2\in W$ and $\fr(\alpha_1)=\fr(\alpha_2)=f$,
then writing $\alpha_i=|\psi_i|$ for $\psi_i\in \ceoF(\C,\D)$, we see
that $\psi_i(v)\neq 0$, so there exist $z_1, z_2\in\bbT$ so that for
$i=1,2$ and every $x\in\C$,
$\psi_i(x)= z_if(xv^*)f(v)^{-1}$.  Hence $\alpha_1=\alpha_2$ showing
that the range map is locally injective.  We already know that the
range map is continuous, so by local compactness, the range map is a
local homeomorphism.
 
Note that convergent nets in $\fRF(\C,\D)$ can be lifted to convergent
nets in $\ceoF(\C,\D)$.  Indeed, if $\fq(\phi_\lambda)\rightarrow
\fq(\phi)$ for some net $(\phi_\lambda)$ and $\phi$ in $\ceoF(\C,\D)$,
choose $v\in\N(\C,\D)$ so that $\phi(v)>0$.  Then for large enough
$\lambda$, $\phi_\lambda(v)\neq 0$.  Then $\fq([v,
\fs(\phi_\lambda)])=\fq(\phi_\lambda)$  and 
$[v,\fs(\phi_\lambda)]\rightarrow \phi$.  The fact that
the groupoid operations on $\fRF(\C,\D)$ are continuous now follows easily from the
continuity of the groupoid operations on $\ceoF(\C,\D)$.  Thus
$\fRF(\C,\D)$ is a locally compact Hausdorff \'{e}tale groupoid.

Finally, $\fq(\phi_1)=\fq(\phi_2)$ if and only if there exist $z\in\bbT$
so that $\phi_1=z\phi_2$.  Moreover, for each $v\in\N(\C,\D)$, let
$F_v:=\{\rho\in F: \rho(v^*v)>0\}$ and set
$\O_v:=\{[v,\rho]: \rho\in F_v\}$.  Then the map $f\mapsto [v,f]$,
where $f\in F_v$ is a continuous section for
$\fq|_{\O_v}$, so $\ceoF(\C,\D)$ is locally trivial.  Also, the action
of $\bbT$ on $\ceoF(\C,\D)$ given above makes $\ceoF(\C,\D)$ into a
$\bbT$-groupoid.  So $\ceoF(\C,\D)$ is a twist over $\fRF(\C,\D).$

\end{proof}

\begin{remark}{Remark}\label{WeylComp}
  Let $(\C,\D)$ be a Cartan pair.  Consider two twists
  associated with $(\C,\D)$: the Weyl twist, and the twist obtained
  from Theorem~\ref{cefgroupoid} applied with the family of strongly
  compatible states.  These twists  can be seen to be isomorphic as follows.
  Let $\bbE: \C\rightarrow \D$ be the conditional expectation, let
  $(\Sigma_W, G_W, \iota_W, \fq_W)$ be the Weyl twist associated to
  $(\C,\D)$, and let $F=\{\sigma\circ E: \sigma\in\hat\D\}$ be the
  family of strongly compatible states on $\C$.  Using direct
  arguments (or the results
  in~\cite[Section~4.2]{BrownFullerPittsReznikoffGrC*AlTwGpC*Al}), one
  can show that the maps
  $\Sigma_W\ni [\sigma_1,v,\sigma_2]_1\mapsto [v,\sigma_2]\in
  \Eigone_F(\C,\D)$ and
  $G_W\ni [\sigma_1,v,\sigma_2]_\bbT\mapsto |[v,\sigma_2]|\in
  \fR_F(\C,\D)$ are isomorphisms of topological groupoids.  Further,
  notice that for $z\in \bbT$,
  $[\sigma_1,\overline{z}v,\sigma_2]_1\mapsto z[v,\sigma_2]$.  Thus
  the twist $(\Eigone_F(\C,\D), \fR_F(\C,\D), \iota, \fq)$ is
  isomorphic to the conjugate Weyl twist,
  $(\overline{\Sigma}_W,G_W, \overline{\iota}_W,\fq_W)$.  In
  particular, by Renault's theorem and Proposition~\ref{kpm},
\begin{equation}\label{WeylComp1} (\C,\D)\simeq C^*_r(\Sigma_W, G_W, -1)\simeq
  C^*_r(\Eigone_F(\C,\D), \fR_F(\C,\D), 1).
\end{equation}
\end{remark}

\begin{remark}{Notation}
For the remainder of this section,  we  use the following notation.
  \begin{enumerate}
 \item  Write
$$\Sigma=\ceoF(\C,\D)\text{ and } G=\fRF(\C,\D), \dstext{so
  that}\unit{G}=F.$$ As in Section~\ref{twc*al}, for $\phi\in \Sigma$,
we will sometimes write $\dot\phi$ instead of $|\phi|$.

\item For $a\in\C$, define
$\fg(a):\ceoF(\C,\D)\rightarrow \bbC$ to be the `Gelfand' map:
for $\phi\in \ceoF(\C,\D)$, \[\fg(a)(\phi)=\phi(a).\] 
Then $\fg(a)$ is a continuous 1-equivariant
function on $\ceoF(\C,\D)$.
\item
  Because of~\eqref{WeylComp1}, we will write
\[C_c(\Sigma,G) \text{ (resp. $C^*_r(\Sigma, G)$) }\dstext{instead of} C_c(\Sigma,
  G,1) \text{ (resp. $C^*_r(\Sigma, G, 1)$)}\] unless there
is danger of confusion.
Also, instead of writing $L_1$ for the $1$-equivariant line bundle
associated to $(\Sigma,
G)$, we will write $L$.  Furthermore, when referring to
elements of $L$, we drop the
subscript and write $[\lambda, \phi]$ instead of $[\lambda,\phi]_1$ for the
equivalence class in $L$ associated to $(\lambda,\phi)\in \bbC\times
\Sigma$. 
\end{enumerate}
\end{remark}

We aim to show that $\fg$ determines a regular homomorphism of $\C$
into $C^*_r(\Sigma,G)$.  For this, it would be convenient if whenever
$v\in\N(\C,\D)$, $\fg(v)\in C_c(\Sigma, G, 1)$.  However, this need
not be the case.  (For an example, let $\H=\ell^2(\bbZ)$ with standard
orthonormal basis $\{e_n\}_{n\in\bbZ}$, let $\K$ be the compact
operators on $\H$, put $\C:=C^*(\K\cup\{I\})$ and
$\D=C^*(\{I\}\cup \{e_ne_n^*: n\in \bbZ\})$.  Then $(\C,\D)$ is a
regular MASA inclusion and
$\ds v:=\sum_{n\in\bbN}\frac{e_{-n}e_n^*}{n}$ is a normalizer whose
support is clopen and not compact.)  To circumvent this issue, we use
the following technical tool.

\begin{lemma}\label{wcmptsup}
  Let $w\in\N(\C,\D)$.  The following statements hold.
  \begin{enumerate}
    \item $\supp(\fg(w))=\{|[w,f]|: f(w^*w)>0\}$ and hence $\supp(\fg(w))$ is an open
      bisection of $G$.
      \item If $d\in\D$ satisfies
\begin{equation}\label{wcmptsup1}\overline{\{\sigma\in \hat\D: \sigma(d)\neq 0\}}\subseteq
  \{\sigma\in \hat\D: \sigma(w^*w)\neq 0\},
\end{equation}
then 
$\overline{\supp}(\fg(wd))$ is compact, that is, $\fg(wd)\in
C_c(\Sigma,G, 1)$.
\item If $\eps>0$, there exists $d\in \D$ satisfying~\eqref{wcmptsup1}
  such that $\norm{w-wd}<\eps$.  
\end{enumerate}
\end{lemma}
\begin{proof}
(a) We have defined $\supp \fg(w)=\{|\phi|\in G:
  |\phi(w)|\neq 0\}$, so $\supp\fg(w)$ is open.
By definition, for any $f\in F$ with $f(w^*w)\neq 0$, $|[w,f]|\in
\supp\fg(w)$.  Conversely, 
  for
  $\phi=[v,f]\in\ceoF(\C,\D)$,
  \begin{align} |\phi|(w)\neq 0 & \Rightarrow 
                                   f(v^*w)\neq 0 \notag \\ 
    & \Rightarrow |[w,f]| =|[v,f]|
      \dstext{(Theorem~\ref{cefform}(f))} \notag\\
          & \Rightarrow f(w^*w)\neq 0.\notag 
  \end{align}  Thus, $\supp(\fg(w))=\{|[w,f]|: f(w^*w)>0\}$.

  We now show $\fr|_{\supp(\fg(w)}$ and $\fs|_{\supp(\fg(w)}$ are one-to-one.
  Let $|\phi_1|, |\phi_2|$ belong to $\supp(\fg(w))$, let
  $f_i:=\fs(|\phi_i|)$ and $g_i:=\fr(|\phi_i|)$.  As just observed, we
  may write $|\phi_i|=|[w,f_i]|$.  If $f_1=f_2$, then
  $|\phi_1|=|[w,f]|=|\phi_2|$, so $\fs|_{\supp \fg(w)}$ is one-to-one.
  On the other hand, if $g_1=g_2$, then
\[|\phi_1|^*=|\phi_1^*|=|[w^*,g_1]|=|[w^*,g_2]|=|\phi_2|^*.\]
Therefore, $|\phi_1|=|\phi_2|$ so $\fr|_{\supp\fg (w)}$ is one-to-one.    Thus 
$\supp(\fg(w))$ is an open bisection.

(b) Let $B:=\{f\in F: f(w^*w)\neq 0\}$.  Then $\fs|_{\supp \fg(w)}:
\supp\fg(w)\rightarrow B$ is a homeomorphism.

Suppose $d\in \D$ satisfies~\eqref{wcmptsup1}.
Let $A:=\{f\in F: f(d)\neq 0\}$ and observe that
$\fs(\supp(\fg(wd)))=A$.  For $g\in \overline A$, we may find
a net $(g_\lambda)$ in $A$ such that $g_\lambda\rightarrow g$.  Then
$(g_\lambda)|_\D\rightarrow g|_\D$, so $g|_\D\in \{\sigma\in \D:
\sigma(w^*w)\neq 0\}$.  Thus, $g(w^*w)\neq 0$, whence $\overline
A\subseteq B$.  Since $F$ is compact, so is $\overline A$.  But
$\fs^{-1}$ is a homeomorphism of $B$ onto $\supp (\fg(w))$ and therefore
$\fs^{-1}(\overline A)=\overline{\supp (\fg(wd))}$ is compact, as
desired.

(c)  If $w=0$, this is obvious, so assume $w\neq 0$.  Without loss of
generality, we may assume $\norm{w}=1$.   Let $S\subseteq [0,1]$ be
the spectrum of $w^*w$.   Suppose first that $0$ is
not an isolated point of $S\cup \{0\}$.
The sets  $X_\eps:=\{\sigma\in \hat\D: \sigma(w^*w)\geq \eps^2\}$ and
$Y_\eps:=\{\sigma\in \hat\D: \sigma(w^*w)\leq \eps^2/4\}$ are closed,
disjoint, and non-empty, so there exists an element $d\in \D$ with $0\leq d\leq I$
such that $\sigma(d)=1$ for  $\sigma \in X_\eps$ and $\sigma(d)=0$ for
$\sigma\in Y_\eps$.  Then for any $\sigma\in \hat\D$,
\[\sigma((I-d)w^*w(I-d))=\sigma(w^*w)\sigma(I-d)^2<\eps^2,\] so the
result holds in this case.

If $0$ is an isolated point of $S\cup\{0\}$, then there is a
projection $d\in \D$ such that $\hat d$ is the characteristic function of
$S\setminus\{0\}$.  Then $v=vd$ and $d$ satisfies \eqref{wcmptsup1}.

\end{proof}

\rm  Before stating the main result of this section, recall that
Proposition~\ref{invideal} shows that
$\K_F=\{x\in\C:
f(x^*x)=0\text{ for all } f\in F\}$ is an ideal of $\C$ whose
intersection with $\D$ is  trivial.  
\begin{theorem}\label{inctoid}  Let $(\C,\D)$ be a regular covering
  inclusion, $F$ a compatible cover for $\hat\D$,
  and let $G:=\fRF(\C,\D)$ and $\Sigma:=\ceoF(\C,\D)$.  
The map $\fg$ extends
uniquely to
a regular $*$-homomorphism $\theta_F:(\C,\D)\rightarrow
(C^*_r(\Sigma,G), C(\unit{G}))$ with $\ker\theta= \K_F$.
Furthermore, the following statements hold.
\begin{enumerate}
\item The \cstaralg\ generated by $\theta_F(\C)\cup C(\unit{G})$ is
  $C^*_r(\Sigma,G)$.
  \item 
  $\theta_F(\C)$ is dense in $C^*_r(\Sigma,G)$ in the
  $\Sigma$-pointwise topology.
\end{enumerate}
\end{theorem}
\begin{proof}
  Throughout the proof, $\bbE$ will denote the (necessarily faithful)
  conditional expectation of $C^*_r(\Sigma,G)$ onto $C(\unit{G})$.
  Also, let $\N_0(\C,\D):=\{v\in \N(\C,\D): \fg(v)\in C_c(\Sigma,G)\}$
  and $\C_0:=\spn\N_0(\C,\D)$.  Finally, during the proof, we will
  write $\theta$ instead of $\theta_F$.
  
   Once again, recall $F=\unit{G}$.
   We regard $C_c(\Sigma, G)$ as a dense subalgebra of
$C^*_r(\Sigma, G)$.  
   A computation shows that any element of $C_c(\Sigma, G)$ supported in
   an open bisection of $G$ belongs to $\N(C^*_r(\Sigma, G), C(\unit{G}))$.
   Clearly 
$\fg: \C_0\rightarrow C_c(\Sigma,G)$.   A computation shows
that for $d\in\D$, $\fg(d)$ is supported on $\unit{G}$.  So 
$\fg(\D)\subseteq C(\unit{G})$.

Next we show that $\fg$ is a $*$-homomorphism of $\C_0$ into
$C_c(\Sigma, G)$.  To do this, it suffices to show that
for $w, w_1, w_2\in
\N_0(\C,\D)$,
\begin{equation}\label{nprod}
  \fg(w_1w_2)=\fg(w_1)\fg(w_2)
\end{equation}
and 
\begin{equation}\label{nprod*}
  \fg(w)^*=\fg(w^*).
\end{equation}
Indeed, these equalities imply $\N_0(\C,\D)$ is a $*$-semigroup, and as
$\fg$ is linear, it will follow that $\fg$ is a $*$-homomorphism.    We
now establish~\eqref{nprod} and~\eqref{nprod*}.

For $i=1,2$, view $\fg(w_i)$ as a continuous section of
the line bundle $L$, that is, $\fg(w_i)(\dot\phi)=[\phi(w_i),
\phi_i]_L$, where we have (temporarily) added a subscript to aid in
distinguishing elements of $L$ from  elements
$\phi=[v,\fs(\phi)]\in \Sigma$.

Suppose $\dot\phi\in\supp (\fg(w_1)\fg(w_2))$.
Lemma~\ref{wcmptsup} shows that for $i=1,2$, $\fg(w_i)$ are supported
on open bisections.  An examination of the definition of
multiplication in $C_c(\Sigma,G)$ (see~\eqref{Lkops}) shows that there
is exactly one composable pair $(\dot\phi_1,\dot\phi_2)\in G^{(2)}$
which contributes to the sum defining the product; that is, for
$i=1,2$, there there are unique $\dot \phi_i\in \supp(\fg(w_i))$
with
\[\fs(\dot\phi_2)=\fs(\dot\phi),\quad
  \fs(\dot\phi_1)=\fr(\dot\phi_2)\dstext{and}
  \dot\phi=\dot\phi_1\dot\phi_2.\]
As $\dot\phi\in\supp(\fg(w_1)\fg(w_2))$,
$\fg(w_i)(\dot\phi_i)\neq 0$, so   regardless of the choice of
$\phi_i\in \fq^{-1}(\dot\phi_i)$, we have $\phi_i(w_i)\neq 0$.  By
multiplying by appropriate elements of $\bbT$, we may choose  $\phi_i\in
\fq^{-1}(\dot\phi_i)$  so that $\phi_i(w_i)>0$.  With
these choices of $\phi_i$, we now
take $\phi:=\phi_1\phi_2\in \fq^{-1}(\phi)$.  In
particular, by
Theorem~\ref{cefform}, we may represent
\[\phi_i=[w_i, \fs(\phi_i)], \dstext{so that}\phi=[w_1w_2,\fs(\phi)].\] 
Then 
\begin{equation}\label{inctoidA}
  (\fg(w_1)\fg(w_2))(\dot\phi)=\fg(w_1)(\dot\phi_1)\fg(w_2)(\dot\phi_2)=[\phi_1(w_1),\phi_1]_L[\phi_2(w_2),\phi_2]_L=[\phi_2(w_1)\phi_2(w_2),\phi]_L.
\end{equation}
As
$\fs(\phi_2)=\fs(\phi)$, 
\begin{align*}
  \phi_1(w_1)\phi_2(w_2)&=\sqrt{\fs(\phi_1)(w_1^*w_1)\,  \fs(\phi_2)(w_2^*w_2)}\\
  &=\sqrt{\fr(\phi_2)(w_1^*w_1)\, \fs(\phi_2)(w_2^*w_2)}=
                         \left(\frac{\fs(\phi_2)(w_2^*w_1^*w_1w_2)}{\fs(\phi_2)(w_2^*w_2)}\right)^{1/2}(\fs(\phi_2)(w_2^*w_2))^{1/2}\\
  &=\sqrt{\fs(\phi_2)(w_2^*w_1^*w_1w_2)}=[w_1w_2,\fs(\phi)](w_1w_2)=\phi(w_1w_2).
\end{align*}
Therefore, when $\dot\phi\in \supp(\fg(w_1)\fg(w_2))$,  
\begin{equation}\label{inctoidD}
  (\fg(w_1)\fg(w_2))(\dot\phi) = 
  [\phi_1(w_1)\phi_2(w_2),\phi]_L
  =[\phi(w_1w_2),\phi]_L=\fg(w_1w_2)(\dot\phi),
\end{equation}
where the first equality uses~\eqref{inctoidA}.
It follows that $\supp(\fg(w_1)\fg(w_2))\subseteq \supp(\fg(w_1w_2))$.

Now suppose $\dot\phi\in \supp(\fg(w_1w_2))$.  Choose $\phi\in
\fq^{-1}(\dot\phi)$ so that $\phi(w_1w_2)>0$ and write
$\phi=[w_1w_2,\fs(\phi)]$, so that
\begin{equation}\label{inctoidB}
  0<  \phi(w_1w_2)=\sqrt{\fs(\phi)(w_2^*w_1^*w_1w_2)}.
\end{equation}
The
Cauchy-Schwartz inequality gives $\fs(\phi)(w_2^*w_2)\neq 0$, so
$\phi_2:=[w_2,\fs(\phi)]\in\Sigma$.  Furthermore,~\eqref{inctoidB}
shows $\fr(\phi_2)(w_1^*w_1)\neq 0$, so $\phi_1:=[w_1,\fr(\phi_2)]\in
\Sigma$.  This yields the factorization
$\dot\phi=\dot\phi_1\dot\phi_2$.  Also, for $i=1,2$, $\phi_i(w_i)\neq
0$, so $\dot\phi_i\in \supp(\fg(w_i))$.   But then 
\[0\neq
  \fg(w_1)(\dot\phi_1)\,\fg(w_2)(\dot\phi_2)=(\fg(w_1)\fg(w_2))(\dot\phi),\]
so $\dot\phi\in \supp(\fg(w_1)\fg(w_2))$.  We have now shown that
$\supp(\fg(w_1)\fg(w_2))= \supp(\fg(w_1w_2)$.  Then~\eqref{inctoidD}
gives~\eqref{nprod}, as desired.

For any $w\in \N_0(\C,\D)$ and $\phi=[v,\fs(\phi)]\in \Sigma$,
$\phi^{-1}=[v^*, \fr(\phi)]$, so 
\[\overline{\phi^{-1}(w)}=
  \frac{\overline{\fr(\phi)(vw)}}{\fr(\phi)(vv^*)^{1/2}}
  =\frac{\overline{\fs(\phi)(wv)}}{\fs(\phi)(v^*v)^{1/2}}
 =[v,\fs(\phi)](w^*)=\phi(w^*).\]  Therefore, 
\begin{align*}
  (\fg(w)^*)(\dot\phi)&=\overline{\fg(w)(\dot\phi^{-1})}=\overline{[\phi^{-1}(w),
                    \phi^{-1}]_L}\\
                  &=[\overline{\phi^{-1}(w)},\phi]_L=[\phi(w^*),\phi]_L=\fg(w^*)(\dot\phi).
\end{align*}
Thus~\eqref{nprod*} holds, and, as noted earlier, we conclude  $\fg$ is a
$*$-homomorphism.

We now turn to showing that $\fg$ is contractive.  The point is that the norms on
   $\C/\K_F$ and $C^*_r(\Sigma,G)$ both arise from the left regular
   representation on appropriate spaces.  Here are the details.

   Let $f\in F.$ Then $f$ can be regarded as either a state on $\C$ or as
determining a state on $C^*_r(\Sigma, G)$ via
evaluation at $f$.  We write $f_\C$ when viewing $f$ as a state on
$\C$, and $f_\Sigma$ when viewing $f$ as a state on
$C^*_r(\Sigma, G)$.

Let $(\pi_{\C,f},\H_{\C,f})$ be the GNS representation of $\C$ arising
from $f_\C$, and let $(\pi_{\Sigma,f},\H_{\Sigma,f})$ be the GNS
representation of $C^*_r(\Sigma,G)$ determined by $f_\Sigma$.

Now fix $f\in \unit{G}$.  For $a_1, a_2\in \C_0$,
\begin{align*}
  \innerprod{a_1+L_f,
    a_2+L_f}_{\H_\C}=f_\C(a_2^*a_1)&=
  \fg(a_2^*a_1)(f)=(\fg(a_2)^*\fg(a_1))(f)=f_\Sigma(\fg(a_2)^*\fg(a_1))\\
  &=\innerprod{\fg(a_1)+\N_f,
    \fg(a_2)+\N_f}_{\H_\Sigma}.
\end{align*}
It follows that the map $a+L_f\mapsto \fg(a)+\N_f$ extends to an
isometry $W_f:\H_{\C,f}\rightarrow \H_{\Sigma,f}$.  

Next, notice that for $a_1, a_2\in \C_0$,
\[\pi_{\Sigma,f}(\fg(a_1))W_f(a_2+L_f)=\fg(a_1a_2)+\N_f=W_f\pi_{\C,f}(a_1)
  (a_2+L_f).\]
Thus for every $a\in \C_0$,
\[\pi_{\C,f}(a)=W^*_f\pi_{\Sigma,f}(\fg(a))W_f, \dstext{so}
  \norm{a}_\C\geq \norm{\pi_{\C,f}(a)}\geq \norm{\pi_{\Sigma,f}(\fg(a))}.\]  We conclude
that for $a\in \C_0$,
\[\norm{a}_\C\geq \sup_{f\in
    F}\norm{\pi_{\Sigma,f}(\fg(a))}=\norm{\fg(a)}_{C^*_r(\Sigma,G)}.\]

Lemma~\ref{wcmptsup} implies  $\C_0$ is norm-dense in $\C$, so $\fg$
extends by continuity to a $*$-homomorphism $\theta:\C\rightarrow
C^*_r(\Sigma,G)$.   
For $v\in\N(\C,\D)$,
Lemma~\ref{wcmptsup} shows that $v\in \overline{\N_0(\C,\D)}$, so
$\theta(v)\in \overline{\fg(\N_0(\C,\D))}\subseteq \N(C^*_r(\Sigma,G),
C(\unit{G}))$.  Thus,
$\theta$ is a regular $*$-homomorphism. 

Let us show $\K_F=\ker\theta$.   Then for $f\in  F$,
$f_\Sigma\circ\bbE=f_\Sigma$.  Furthermore,  $f_\C$
and $f_\Sigma\circ \theta$ agree on the dense set $\C_0$, so
$f_\C=f_\Sigma\circ \theta=f_\Sigma\circ\bbE\circ \theta$.  
For $x\in \C$,
\[x\in \K_F\Leftrightarrow 
  \bbE(\theta(x^*x))=0\Leftrightarrow x\in \ker\theta.\]

We now  establish statement (a), that is,
the \cstaralg\ generated by $\theta(\C)\cup
C(\unit{G})$ is $C^*_r(\Sigma,G)$.  
Let
\[\M:=\{h\fg(v) k: v\in \N_0(\C,\D) \text{ and } h, k\in
  C(\unit{G})\}.\] Then $\M$ contains $C(\unit{G})$ and is a
$*$-semigroup of normalizers in $C^*_r(\Sigma,G)$.  We will show
$\spn \M$ is dense in $C^*_r(\Sigma,G)$.
We require the following fact. 

\begin{remark*}{Fact}
\textit{Suppose $U\subseteq G$ is an open bisection and $u\in
  C_c(\Sigma,G)$ satisfies $\overline{\supp(u)}\subseteq U$.  If
  $\dot\phi\in \supp(u)$, then there exists $k\in C(\unit{G})$ such
  that  $uk\in \M$ and \[\dot\phi\in \supp(uk)\subseteq
  \overline{\supp(uk)}\subseteq \supp(u).\]
}
\end{remark*}
\proof[Proof of the Fact.]
Note that if
$w\in C_c(\Sigma,G)$ satisfies $\supp w\subseteq U$, then
$\dot\tau\in \supp(w)
    \Leftrightarrow \fs(\dot\tau)\in \supp(w^*w)$.  We will repeatedly
    use this.

  Choose $\phi\in \fq^{-1}(\dot\phi)$ and write $\phi=[v,\fs(\phi)]$
  for some $v\in \N(\C,\D)$.  As $\phi(v)>0$, by Lemma~\ref{wcmptsup},
  we may assume without loss of generality that $v\in \N_0(\C,\D)$.
  Then $\dot\phi\in \supp(\fg(v))\cap \supp(u)$.  Therefore, there
  exists $h\in C(\unit{G})$ so that $h(\fs(\dot\phi))=1$ and
  $\overline{\supp h}\subseteq (\supp(\fg(v^*v))\cap \supp(u^*u))$.  Thus, 
\[\dot\phi\in \supp(\fg(v)h)\subseteq \supp(u).\]
By Lemma~\ref{ceno},  $\bbE(u h^*\fg(v^*)) \, \fg(v)h= u h^*\fg(v^*v) h$.
Take $k=h^*\fg(v^*v)h$.  Then $k\in C(\unit{G})$ and
$uk\in \M$.  Since $\fs(\dot\phi)\in\supp(\fg(v^*v))\cap
\supp(h)\subseteq\supp(\fg(v^*v)) \cap \supp(u^*u)$, we
have $\dot\phi\in \supp(uk)$.
This establishes the fact.
\hfill$\diamondsuit$

\vskip 6pt

Now fix $u\in C_c(\Sigma,G)$ with the property that its closed support
is contained in an open bisection $U\subseteq G$.  Let
\[J:=\{k\in C(\unit{G}): uk\in \overline{\spn}\M\}.\]
As $\M$
is a $*$-semigroup, $J$ is a closed ideal of $C(\unit{G})$.   The
fact shows that if $\dot\phi\in \supp(u)$, then there exists $k\in J$
such that $\dot\phi\in \supp(uk)$.   Thus $\fs(\dot\phi)$ does
not annihilate $J$ because $(uk)(\dot\phi)=u(\dot\phi)\,
k(\fs(\dot\phi))$.  Therefore, $C_0(\supp(u^*u))\subseteq J$.  
Let $(h_\lambda)_{\lambda\in \Lambda}$ be an approximate unit for
$C_0(\supp(u^*u))$.   Then
\[\norm{u-uh_\lambda}^2=\norm{u^*u-2u^*uh_\lambda+h_\lambda^2
    u^*u}\rightarrow 0.\]  Thus, $u\in \overline{\spn}\M$.

It follows from~\cite[Proposition~3.10]{ExelInSeCoC*Al} that $C_c(\Sigma,G)\subseteq \overline\spn\M$,
and therefore $\spn\M$ is dense in $C^*_r(\Sigma,G)$.  This
establishes part (a).

Finally, we turn to establishing part (b),  the $\Sigma$-pointwise  
density of $\theta(\C)$ in $C^*_r(\Sigma,G)$.  
Let $\X\subseteq\dual{C^*_r(\Sigma,G)}$ 
be the linear span of the evaluation functionals $\xi\mapsto
\xi(\phi)$ where $\xi\in C^*_r(\Sigma,G)$ and $\phi\in \Sigma$.  
Suppose $\mu\in \X$ annihilates $\theta(\C)$.  Then there exists
$n\in\bbN$, scalars $\lambda_1,\dots, \lambda_n$, elements $v_1,\dots,
v_n\in\N(\C,\D)$ and $f_1,\dots, f_n\in F$ such that  for any $\xi\in
C^*_r(\Sigma,G)$, 
$$\mu(\xi)=\sum_{k=1}^n \lambda_k\xi([v_k,f_k]).$$   Without loss of
generality, we may assume that $[v_i,f_i]\neq [v_j,f_j]$ if $i\neq j$.
 Since $\mu$ annihilates
$\theta(\C)$, for every $a\in \C_0$, 
$$0=\mu(\fg(a))=\sum_{k=1}^n \lambda_k[v_k,f_k](a).$$  Fix $1\leq
j\leq n$, and let $d,e\in \D$ be such that
$\beta_{v_j}(f_j)(d)=f_j(e)=1$.  For  $i\neq j$, since $[v_i,f_i]\neq
[v_j,f_j]$, either $f_j\neq f_i$ or $\beta_{v_i}(f_i)\neq
\beta_{v_j}(f_j)$.  Hence we assume that  $d$ and $e$ have been chosen
so
that if $i\neq j$, then $[v_i,f_i](dv_je)=0$.  Then 
\[\mu(\fg(v_j))=\lambda_jf_j(v_j^*v_j)^{1/2}=0.\]  As
$f_j(v_j^*v_j)\neq 0$, we obtain $\lambda_j=0$.  It follows that
$\mu=0$.  Since the dual of $C^*_r(\Sigma,G)$ equipped with the
$\Sigma$-pointwise topology is $\X$, we conclude that $\theta(\C)$
is dense in the $\Sigma$-pointwise topology on $C^*_r(\Sigma,G).$      
This completes the proof.

\end{proof}

Recalling that $\K_F\cap \D=(0)$, we will abuse notation and view $\D$
as a subalgebra of $\C$ or $\C/\K_F$ depending on context. 
The following is immediate.

\begin{corollary}\label{inctoidCor1} $(C^*_r(\Sigma,
  G), C(\unit{G}),\theta_F)$ is a \cover\ for $(\C/\K_F,\D)$.
\end{corollary}

Recall that when $(\C,\D)$ has the unique pseudo-expectation property,
then $\fS_s(\C,\D):=\dual{E}(\widehat{I(\D)})$, where $E$ is the pseudo-expectation.  

When $(\C,\D)$ has a Cartan envelope, the following gives a
description
of the Cartan envelope as the
\cstaralg\ of a twist. 

\begin{corollary}\label{inctoidCor3}
  Suppose $(\C,\D)$ has  the unique pseudo-expectation
  property, and let $F:=\fS_s(\C,\D)$.   Then 
  $(C^*(\Sigma_F,G_F), C(\unit{G_F}),\theta_F)$ is the Cartan envelope for $(\C/\K_F,\D)$.
\end{corollary}
\begin{proof}
Theorem~\ref{upse=>cov} shows $F=\unit{G}$ is a compatible cover for $\hat\D$
and Theorem~\ref{2fP}(b) implies that $(F,r)$ is an essential cover
for $\hat\D$.  Therefore, the map $\alpha: \D\rightarrow C(F)$ given
by $d\mapsto \hat d\circ r$ yields an essential extension $(C(F),
\alpha)$ for $\D$.   As  $\theta_F|_\D= \alpha$,
Theorem~\ref{inctoid} shows that
$(C^*_r(\Sigma_F,G_F), C(F), \theta)$ is a Cartan envelope for $(\C/\K_F,\D)$.   

  \end{proof}

Different compatible covers yield different twists, and hence
different reduced \cstaralg s.   We now describe the relationship between
these objects when given a (set-theoretic) inclusion of compatible covers.

\begin{proposition}\label{twistprop}  Let $(\C,\D)$ be a covering
  inclusion and for $i=1,2$, suppose
  $F_i\subseteq \fS(\C,\D)$ are compatible covers for $\hat\D$.  Put
  $\Sigma_i=\Eigone_{F_i}(\C,\D)$, $G_i=\fR_{F_i}(\C,\D)$ and let
  $\theta_i:\C\rightarrow C^*_r(\Sigma_i,G_i)$ be the homomorphism
  described in Theorem~\ref{inctoid}.  If $F_1\subseteq F_2$, then
  $\Sigma_1\subseteq \Sigma_2$ and there exists a $*$-epimorphism $q:
  C^*_r(\Sigma_2,G_2)\twoheadrightarrow C^*_r(\Sigma_1,G_1)$ such that the
  following diagram commutes.
\begin{equation*}\label{difcov}\xymatrix{
    &C^*_r(\Sigma_2,G_2)\ar@{>>}[d]^q\\
    \C\ar[ur]^{\theta_2} \ar[r]_-{\theta_1}&
    C^*_r(\Sigma_1,G_1) }
\end{equation*}

\end{proposition}
\begin{proof}
Let  $F_1$ and $F_2$ be compatible covers for $\hat\D$ with
$F_1\subseteq F_2$.     Then 
\[\{[v,\rho_1]: \rho_1\in F_1, v\in\N(\C,\D), \rho_1(v^*v)\neq
  0\}\subseteq \{[v,\rho_2]: \rho_2\in F_2, v\in\N(\C,\D), \rho_2(v^*v)\neq
  0\}.\]  The continuity of the source map and the fact that $F_1$ is
closed implies that 
$\Eigone_{F_1}(\C,\D)$ is a closed subgroupoid of $\Eigone_{F_2}(\C,\D)$.
Similarly, $\fR_{F_1}(\C,\D)$ is a closed subgroupoid of
$\fR_{F_2}(\C,\D)$.  In other words, $G_1$ is a closed subgroupoid of
$G_2$ and $\Sigma_1$ is a closed subgroupoid of $\Sigma_2$.

  Suppose $|[v,\rho_1]|\in G_1$ and that for some $|[w,\rho_2]|,
|[w',\rho_2']|\in G_2$, $|[v,\rho_1]|$ factors as 
\[|[v,\rho_1]|=|[w,\rho_2]| \, \, |[w',\rho_2']|.\]
Then $\rho_2'=\rho_1$ and $\tilde\beta_{w'}(\rho_2')=\rho_2$.  As
$\rho_1\in F_1$ and $w'\in \N(\C,\D)$, the invariance of $F_1$ gives
$\rho_2\in F_1$.  Thus $|[w,\rho_1]|$ and $|[w',\rho_2']|$ belong to
$G_1$.  An application of Lemma~\ref{twistquot} shows that the
restriction mapping extends to a $*$-epimorphism $q$ of 
$C_c(\Sigma_2,G_2)$ onto $C_c(\Sigma_1,G_1)$.  That $q\circ
\theta_2=\theta_1$ follows from the definition of $\theta_i$.
\end{proof}

\def\cprime{$'$}
\providecommand{\bysame}{\leavevmode\hbox to3em{\hrulefill}\thinspace}
\providecommand{\MR}{\relax\ifhmode\unskip\space\fi MR }
\providecommand{\MRhref}[2]{%
  \href{http://www.ams.org/mathscinet-getitem?mr=#1}{#2}
}
\providecommand{\href}[2]{#2}

\def\cprime{$'$}

\end{document}